\documentclass[11pt]{amsart}

\usepackage{fullpage}
\usepackage[utf8]{inputenc}
\usepackage[T1]{fontenc}
\usepackage{graphicx}
\usepackage{amsmath,amsfonts,amssymb,mathtools} 
\usepackage{stmaryrd}
\usepackage{graphicx}
\usepackage{mathrsfs,dsfont}
\usepackage{amsmath,amsthm,amsthm}

\usepackage{hyperref}

\usepackage{enumerate}

\usepackage{color}

\newtheorem{theo}{Theorem}[section]

\newtheorem{cor}[theo]{Corollary}
\newtheorem{rem}[theo]{Remark}
\newtheorem{prop}[theo]{Property}
\newtheorem{propo}[theo]{Proposition}
\newtheorem{lemma}[theo]{Lemma}
\newtheorem{defi}[theo]{Definition}

\newtheorem{hyp}[theo]{Assumption}
\newcommand{\E}{\mathbb{E}}
\newcommand{\R}{\mathbb{R}}

\newcommand{\N}{\mathbb{N}}

\newcommand{\tore}{\mathbb{T}}
\newcommand{\md}{m}

\newcommand{\St}{\mathcal{S}} 
\newcommand{\nuSt}{\lambda} 

\newcommand{\En}{\mathcal{E}} 
\newcommand{\Dr}{\mathcal{D}} 

\newcommand{\Ma}{\mathbb{M}} 

\newcommand{\K}{\mathcal{K}} 

\newcommand{\No}{\mathcal{N}} 
\newcommand{\no}{{\rm n}}

\newcommand{\ph}{\varphi} 

\begin{document}

\title{Convergence analysis of Adaptive Biasing Potential methods for diffusion processes}

\author{Michel Bena\"im}
\address{Universit\'e de Neuch\^atel, Institut de Math\'ematiques, Rue Emile Argand 11, CH-2000 Neuch\^atel, Switzerland}
\email{michel.benaim@unine.ch}

\author{Charles-Edouard Br\'ehier}
\address{Univ Lyon, CNRS, Université Claude Bernard Lyon 1, UMR5208, Institut Camille Jordan, F-69622 Villeurbanne, France}
\email{brehier@math.univ-lyon1.fr}


\date{}

\maketitle

\begin{abstract}
This article is concerned with the mathematical analysis of a family of adaptive importance sampling algorithms applied to diffusion processes. These methods, referred to as Adaptive Biasing Potential methods, are designed to efficiently sample the invariant distribution of the diffusion process, thanks to the approximation of the associated free energy function (relative to a reaction coordinate). The bias which is introduced in the dynamics is computed adaptively; it depends on the past of the trajectory of the process through some time-averages.

We give a detailed and general construction of such methods. We prove the consistency of the approach (almost sure convergence of well-chosen weighted empirical probability distribution). We justify the efficiency thanks to several qualitative and quantitative additional arguments. To prove these results , we revisit and extend tools from stochastic approximation applied to self-interacting diffusions, in an original context.
\end{abstract}

\section{Introduction}\label{sec:intro}

In many applications in physics, biology, chemistry, etc... there is a huge interest in the two following problems. First, in sampling probability distributions (denoted by $\mu$), {\it i.e.} in constructing families of independent random variables with distribution $\mu$. Second, in computing averages $\int\varphi d\mu$ of real-valued functions $\varphi$. These questions lead to challenging computational issues, when the support of $\mu$ has large (possibly infinite) dimension -- for instance, when $\mu$ is the equilibrium distribution of a large system of particles, which is the typical situation in the field of molecular dynamics. The scientific literature contains many examples, as well as many approaches to construct efficient approximation procedures. We do not intend to provide an extensive review, but some relevant examples, which are connected to the methodology studied in this article, will be provided.

Many methods are based on stochastic simulation, also called MCMC methods. The idea is to run an ergodic Markov process having $\mu$ as invariant distribution, and to use empirical averages as estimators. A standard example (but there are others) of such a process is given  by the overdamped Langevin dynamics on $\R^d$,
\[
dx_t=-\nabla V(x_t)dt+\sqrt{2\beta^{-1}}dW_t,
\]
whose invariant measure (under appropriate growth and regularity assumptions on the function $V$) is the Boltzmann-Gibbs probability distribution
\[
\mu(dx)=\mu_\star(dx)=\frac{e^{-\beta V(x)}}{Z(\beta)}dx.
\]

One of the main limitations of standard MCMC approaches comes from the fact that the ergodic dynamics are metastable whenever $\mu$ is multimodal. In the example above, this happens when $V$ has several local minima. A direct simulation is not able to efficiently and accurately sample the rare transitions between the metastable states, hence the need for advanced Monte-Carlo methods.

\bigskip

Many strategies have been proposed, analyzed and applied, to overcome this issue. The associated variance reduction approaches may be divided into two main families. On the one hand, Importance Sampling strategies are based on changing the reference probability measure. The realization of the rare events is enhanced by appropriate reweighting of $\mu$. In our context, such strategies require to simulate modified processes, which are constructed by biasing the dynamics. On the other hand, Splitting strategies use interacting replicas, with mutation and selection procedures, without modifying the process dynamics.

The methods studied in this article are an example of Adaptive Biasing methods. They are based on the Importance Sampling strategy and the use of the so-called Free Energy function (which will be introduced below). When going into the details of the schemes and of the applications, there are many different versions; they all aim at flattening the free energy landscape, and to make free energy barriers disappear. We refer to the monograph~\cite{LRS_book} for an extensive review of such methods, and to~\cite[Section 4]{LS_acta} for a survey on mathematical techniques. To name a few of the versions, we mention the following examples of adaptive biasing methods: the adaptive biasing force~\cite{ABF_1},~\cite{ABF_2}~\cite{ABF14}; the Wang-Landau algorithm~\cite{WL_1},~\cite{WL_2}; metadynamics~\cite{M},~\cite{WTM}; the self-healing umbrella sampling method~\cite{SHUS}. For related mathematical results, see for instance~ \cite{CL11},~\cite{JLR10},~\cite{LM11},~\cite{LRS08}  (adaptive biasing force);~\cite{FJKLS14},~\cite{FJKLS15} (Wang-Landau),~\cite{FJLS17} (self-healing umbrella sampling). This list is not exhaustive. We also refer to the recent survey paper~\cite{D17} (and to references therein) for discussions and comparison of these methods.

Our aim in this article is to give a mathematical analysis of a family of methods, independently of a comparison with the other methods mentioned above: the Adaptive Biasing Potential methods, related to~\cite{DLLSFl10}. In such methods, one constructs adaptive approximation of a potential energy function, instead of a mean force (see~\cite{LRS07} for a discussion), hence the name. One of the key aspects of this work is that the dynamics is biased using quantities computed as time-averages over a single realization of the dynamics

Let us mention the type of mathematical properties such algorithms are required to satisfy (exactly or in an approximate sense) for the estimation of averages $\int \varphi d\mu$. On the one hand, the consistency is the long-time convergence to this quantity, in a strong (almost sure or $L^p$) sense, or in a weak sense (convergence of the expected value). On the other hand, the efficiency is generally studied in terms of the asymptotic mean-square error. It may also be considered from the point of view of the long-time behavior of occupation measures of the process.

A preliminary analysis of the Adaptive Biasing Potential methods considered in this article, has been performed in~\cite{BB16}, in a simplified framework, without proofs. The aim of the present article is to provide the missing arguments, in a more abstract framework, and to study substantial generalizations. Below we first present the methods in the simplified framework from~\cite{BB16}, see Section~\ref{sec:intro_ABP}: the strategy and the results are exposed. We then present the general framework of the article, and the associated results, in Sections~\ref{sec:Intro_general} and~\ref{sec:Intro_orga}.

\subsection{Adaptive Biasing Potential method in a simplified framework}\label{sec:intro_ABP}

This section is pedagogical: the ideas are introduced independently of the abstract notation which will allow us to consider many examples of diffusion processes. We are interested in sampling probability distributions on the flat $d$-dimensional torus $\tore^d=\bigl(\mathbb{R} /\mathbb{Z}\bigr)^d$, of the following form:
\[
\mu(dx)=\mu_\star(dx)=\frac{\exp\bigl(-\beta V(x)\bigr)}{Z(\beta)}dx,
\]
where $V:\tore^d\to \R$ is a smooth potential energy function, and $\beta\in(0,\infty)$ is referred to as the inverse temperature. Finally, $dx$ is the Lebesgue measure on $\tore^d$, and $Z(\beta)=\int_{\tore^d}e^{-\beta V(x)}dx$ is the normalization constant.

A natural choice of associated ergodic process is given by the overdamped Langevin dynamics (or Brownian dynamics)
\begin{equation}\label{eq:OL_intro}
dX_t^0=-\nabla V(X_t^0)dt+\sqrt{2\beta^{-1}}dW_t, \quad X_0^0=x_0,
\end{equation}
where $\bigl(W_t\bigr)_{t\ge 0}$ is a standard Wiener process on $\tore^d$.

By ergodicity, the empirical distribution $\overline{\mu}_t^0=\frac{1}{t}\int_0^t \delta_{X_r^0}dr$ converges (in distribution), almost surely, towards $\mu_\star$, when time $t$ goes to infinity. As already mentioned, the convergence may be slow when $V$ has several local minima.

\bigskip

To accelerate convergence to equilibrium, other stochastic processes need to be used. In this article, the dynamics is modified with an adaptive change of the potential energy function: the function $V$ is replaced with a time-dependent function $V_t$ -- hence the terminology {\it Adaptive Biasing Potential} (ABP) method,
\[
dX_t=-\nabla V_t(X_t)dt+\sqrt{2\beta^{-1}}dW_t.
\]
Compared with other methods mentioned above, one of the specificities of the method considered in this article, is the structure of the time-dependent potential energy function $V_t$. It is constructed as $V_t=V-A_t\circ \xi$, where $\xi:(x_1,\ldots,x_d)\in\tore^d\mapsto (x_1,\ldots,x_\md)\in \tore^\md$, with $\md\in\left\{1,\ldots,d-1\right\}$, is an auxiliary function, referred to as the reaction coordinate, and $A_t:\tore^\md\to \R$ is an approximation (in the regime $t\to \infty$) of the so-called Free Energy function. In applications, the dimension $\md$ is chosen much smaller than $d$, and typically $\md\in\left\{1,2,3\right\}$. It thus remains to explain how the function $A_t$ is constructed adaptively. This is done in terms of the values $\bigl(X_r\bigr)_{0\le r< t}$ of the process $X$ up to time~$t$.

Precisely, the dynamics of the ABP method, in the simplified framework considered in the current section (for the generalized version, see Equation~\eqref{eq:ABP_full}), is given by the following system:
\begin{equation}\label{eq:ABP_simple}
\begin{cases}
dX_t=-\nabla\bigl(V-A_t\circ \xi\bigr)(X_t)dt+\sqrt{2\beta^{-1}}dW(t)\\
\overline{\mu}_t=\frac{\overline{\mu}_0+\int_{0}^{t}\exp\bigl(-\beta A_r\circ\xi(X_r)\bigr)\delta_{X_r}dr}{1+\int_{0}^{t}\exp\bigl(-\beta A_r\circ\xi(X_r)\bigr)dr}\\
\exp\bigl(-\beta A_t(z)\bigr)=\int_{\tore^d}K\bigl(z,\xi(x)\bigr)\overline{\mu}_t(dx),~\forall z\in\tore^m,
\end{cases}
\end{equation}
where a smooth kernel function $K:\tore^\md\times \tore^\md \to (0,+\infty)$, such that $\int_{\tore^\md}K(z,\zeta)dz=1, \forall \zeta\in \tore^\md$, is introduced. The diffusion equation (first line of~\eqref{eq:ABP_simple}) depends on the gradient $\nabla A_t$ of the real-valued function $A_t$, whereas the probability distribution $\overline{\mu}_t$ on $\tore^d$ does not have a density (second line of~\eqref{eq:ABP_simple}), hence the need for a smooth kernel.

For a practical implementation of the method, an additional time discretization of the continous-time dynamics is required. However, in this article, we do not discuss this question, and we only perform the analysis at the continuous-time level.

The expression in the third line in~\eqref{eq:ABP_simple} is motivated by the definition of the Free Energy function:
\begin{equation}\label{eq:FE_simple}
\exp\bigl(-\beta A_\star(z)\bigr)=\int_{\tore^{d-\md}}\frac{\exp\bigl(-\beta V(z,x_{\md+1},\ldots,x_d)\bigr)}{Z(\beta)}dx_{\md+1}\ldots dx_d.
\end{equation}
To simplify notation, we omit the dependence of $A_\star$ with respect to the parameter $\beta$.

The unknows in~\eqref{eq:ABP_simple} are the stochastic processes $t\mapsto X_t \in\tore^d$, $t\mapsto \overline{\mu}_t\in \mathcal{P}(\tore^d)$ (the set of Borel probability distributions on $\tore^d$, endowed with the usual topology of weak convergence of probability distributions), and $t\mapsto A_t\in \mathcal{C}^{\infty}(\tore^\md)$ (the set of infinitely differentiable functions on $\tore^\md$). Initial conditions $X_{t=0}=x_0$ and $\overline{\mu}_{t=0}=\overline{\mu}_0$ are prescribed.

An important observation is that the third equation in~\eqref{eq:ABP_simple} introduces a coupling between the evolutions of the diffusion process $X_t$ and of the probability distribution $\overline{\mu}_t$. Thus the system defines a type of self-interacting diffusion process. However, a comparison with~\cite{BLR02} and subsequent articles~\cite{BR03},~\cite{BR05} and~\cite{BR11}, reveals a different form of coupling. One of the aims of this article is to study the new arguments which are required for the study of the system~\eqref{eq:ABP_simple}.

The most important quantity in~\eqref{eq:ABP_simple} is the random, time-dependent, probability distribution $\overline{\mu}_t$. Observe that its construction requires two successive operations to be performed. First, a weighted occupation measure $\mu_t=\overline{\mu}_0+\int_{0}^{t}\exp\bigl(-\beta A_r\circ\xi(X_r)\bigr)\delta_{X_r}dr$ is computed. Second, this measure is normalized to define a probability distribution, $\overline{\mu}_t=\frac{\mu_t}{\int_{x\in\tore^d}\mu_t(dx)}$. The weights $\exp\bigl(-\beta A_r\circ\xi(X_r)\bigr)$ in the definition of $\overline{\mu}_t$ are chosen so as to obtain the following consistency result for the estimation of $\mu_\star$.
\begin{theo}\label{th:cv_simple}
Almost surely, $\overline{\mu}_t$ converges to $\mu_\star$, in $\mathcal{P}(\tore^d)$.

Moreover, define the function $A_\infty$, such that $\exp\bigl(-\beta A_\infty(\cdot)\bigr)=\int K(\cdot,\xi(x))\mu_{\star}(dx)$. Then, almost surely, $A_t$ converges to $A_\infty$, in $\mathcal{C}^{k}(\tore^\md)$, $\forall~k\in\N$.
\end{theo}

The strategy of proof of this result is described in our previous work~\cite{BB16}. The present article, provides the technical details in a more general context.

Let us explain the role that the weights play in this result. On the one hand, they are fundamental for the consistency of the method. This observation is not surprising, indeed it is a standard feature of importance sampling approaches. Indeed, Theorem~\ref{th:cv_simple} is easily seen to be valid for non-adaptive versions of~\eqref{eq:ABP_simple} (see~\eqref{eq:BP} and the convergence result~\eqref{eq:mu_t^A}), where a bias is initially given and not modified ($A_t=A$ for all $t\ge 0$ in the definitions of $X_t$ and $\overline{\mu}_t$ in~\eqref{eq:ABP_simple}). On the other hand, the convergence of $A_t$ to $A_\infty$ comes from the way the evolutions of $X_t$ and $\overline{\mu}_t$ are coupled, in the third equation of~\eqref{eq:ABP_simple}. The convergence of $A_t$ to $A_\infty$ reveals the efficiency of the method: indeed, $A_\infty$ is an approximation of the Free Energy function $A_\star$, defined by~\eqref{eq:FE_simple}. Note that, by construction, $\exp\bigl(-\beta A_\star(z)\bigr)dz$ is a probability distribution on $\tore^\md$, which is the image of $\mu_\star$ by the reaction coordinate $\xi$. As will be explained below (see Section~\ref{sec:free_energy}), biasing the dynamics (in a non-adaptive way) using the function $A_\star$ (which is not known in practice) is a natural choice. The adaptive method can thus naturally be seen as a stochastic approximation algorithm, with a parameter being learnt on-the-fly, see {\it e.g.}~\cite{B99},\cite{BMP90},\cite{D97},\cite{KY03}.

\subsection{General framework}\label{sec:Intro_general}

The observations made in Section~\ref{sec:intro_ABP}, concerning the system~\eqref{eq:ABP_simple} and the consistency result, Theorem~\ref{th:cv_simple}, can be generalized as follows.

First, contrary to~\eqref{eq:OL_intro}, the state space of the dynamics may not be compact. It may also be of infinite dimension.

The most important generalization concerns the type of diffusion processes which are considered: our general framework also encompasses the following examples (this list is not exhaustive): (hypoelliptic) Langevin dynamics with position and momenta variables, extended dynamics -- where an auxiliary variable is associated with the mapping $\xi$, see~\cite{LLSH16}) -- and Stochastic Partial Differential Equations (SPDEs) -- which are infinite dimensional diffusion processes. It may also be possible to study diffusions on smooth manifolds, however to simplify the presentation this situation is not treated.

Abstract notation and analysis allow us to treat simulataneously these examples in a general framework. However note that the SPDE example is studied separately, in Section~\ref{sec:SPDE} to simplify the exposition.

The reaction coordinate $\xi$ can be an arbitrary, smooth function, with values in a $\md$-dimensional compact manifold. The associated free energy is then defined in terms of a Radon-Nikodym derivative of the image of the invariant distribution $\mu_\star$ by $\xi$, with respect to a reference measure.

Note that the general framework and the associated abstract notation are constructed to emphasize the most important assumptions, made on the models and on the algorithm, which are required for the well-posedness and the consistency of the approach.

\subsection{Organization of the paper}\label{sec:Intro_orga}

In Section~\ref{sec:framework}, the abstract framework is introduced, with emphasis on the following objects: the diffusion process dynamics (Section~\ref{sec:dynamics}), the main examples (Section~\ref{sec:diffusions}), the invariant probability distribution $\mu$ (Section~\ref{sec:invariant}), and the Free Energy function (Section~\ref{sec:free_energy}).

The construction of generalized versions of the Adaptive Biasing Potential method, given by~\eqref{eq:ABP_simple}, is provided in Section~\ref{sec:ABP}. In particular, well-posedness results and important estimates are stated precisely there.

Section~\ref{sec:results} contains the main results of this article, in finite dimensional cases, concerning the long-time behavior of the method. On the one hand, the consistency of the approach, {\it i.e.} the almost sure convergence (in distribution) of $\overline{\mu}_t$ to $\mu$, is given in Theorem~\ref{th:cv_as_phi} and Corollary~\ref{cor:cv_as}. On the other hand, the efficiency is analyzed first in terms of the convergence of the approximation $A_t$ of the free energy function, Corollary~\ref{cor:cv_A}, and of occupation measures, Corollary~\ref{cor:conv_rho}; second, in terms of the asymptotic mean-square error, Proposition~\ref{propo:variance}.

Section~\ref{sec:proof_consistency} is devoted to the proof of the consistency. We first describe (see Section~\ref{sec:SA_perspective}) how to eliminate the weights in the definition of $\overline{\mu}_t$, thanks to a random change of time variable. As explained in~\cite{BB16}, the new system may then be treated using the ODE method from stochastic approximation (see~\cite{B99},~\cite{BMP90},~\cite{D97},~\cite{KY03}), thanks to an asymptotic time scale separation into slow (occupation measure) and fast (diffusion process) evolutions, like in~\cite{BLR02}. Section~\ref{sec:proof_decomp_error} then contains all the technical details for a direct proof of the consistency result. Auxiliary properties for solutions of Poisson equations are provided.

Section~\ref{sec:proof_variance} is devoted to the analysis of the asymptotic mean-square error. As expected, the behavior of the variance for the adaptive system is asymptotically the same as for a non-adaptive system where the bias is chosen as the limit $A_\infty$ of the adaptive bias $A_t$.

Finally, in Section~\ref{sec:SPDE} we consider infinite dimensional diffusion processes, which are solutions of SPDEs. This formally fits in the general framework, but a rigorous analysis needs to be performed separately.

\section{Framework}\label{sec:framework}

\subsection{Dynamics and abstract notation}\label{sec:dynamics}

To simplify notation, without loss of generality, from now on the parameter $\beta$ is set equal to $1$.

\subsubsection{Unbiased dynamics}

The unbiased (or original) dynamics is a diffusion process $\bigl(X_t^0\bigr)_{t\ge 0}$, with values on a state space denoted by $\St$. The process is solution of a Stochastic Differential Equation (SDE), when the dimension of $\St$ is finite; or of a Stochastic Partial Differential Equation (SPDE), when the dimension of $\St$ is infinite. The SDE or the SPDE is written in the following form
\begin{equation}\label{eq:dyn}
dX_t^0=\Dr(V)(X_t^0)dt+\sqrt{2}\Sigma dW_t \quad , \quad X_0^0=x_0,
\end{equation}
where $\bigl(W_t\bigr)_{t\in\R^+}$ is a standard Wiener process on $\St$, and $\Sigma$ is a linear mapping which is specified in each example below.

In~\eqref{eq:dyn}, the initial condition $x_0\in \St$ is arbitrary, and is assumed to be deterministic for simplicity. The convergence results may be extended to a random initial condition (independent of the Wiener noise) by a standard conditioning argument. The value of $x_0$ plays no role in the analysis below.

The drift coefficient $\Dr(V)$ in~\eqref{eq:dyn} depends on the potential energy function $V:E_d\to \R$, defined on a set $E_d$, where $E_d=\tore^d$ (periodic, compact case) or $E_d=\R^d$ (non compact case). Note that in general $\St\neq E_d$. The functions $V$ and $\Dr(V)$ are assumed of class $\mathcal{C}^\infty$, also the results can be adapted to deal with the situation when $V$ and $\Dr(V)$ are merely of class $\mathcal{C}^n$ for sufficiently large $n\in\N$. In the non compact case $E_d=\R^d$, growth conditions, that will be described for each example, are required.

\subsubsection{Non-adaptively biased dynamics}

Having introduced the unbiased dynamics~\eqref{eq:dyn}, we now describe the family of biased dynamics which we consider in this article, first in a non-adaptive context. The drift coefficient $\Dr(V)$ in~\eqref{eq:dyn} is modified, being replaced by $\Dr(V,A)$, where the function $A$ depends only on a small number of degrees of freedom of the system. In the current section, the bias is non-adaptive: the function $A$ is deterministic and does not depend on time.

\bigskip

We now make precise how $\Dr(V,A)$ is defined. It depends on the mapping $A\circ\xi$, where
\begin{itemize}
\item $\xi:E_d\to \Ma_\md$ is a fixed smooth function, where $\Ma_\md$ is a $\md$-dimensional smooth, compact, manifold, and $\md\in\left\{1,\ldots,d-1\right\}$.
\item $A:\Ma_\md\to \R$ is a smooth function.
\end{itemize}

The mapping $\xi$ is called the reaction coordinate (following the terminology from molecular dynamics applications) and the variables $z=\xi(x)$ are often called collective variables. The functions $\xi$ and $A\circ \xi$ are defined on $E_d$, like the potential energy function $V$.

Finally, an extension
\[
\xi_{\St}:\St\to \Ma_\md
\]
of the reaction coordinate $\xi$, is also defined on the state space $\St$, with a procedure depending on the example of diffusion process.

In the non-compact case $E_d=\R^d$, all the derivatives of $\xi$ are assumed to be bounded.

As explained in the introduction, in practice one chooses $\md$ much smaller than $d$, and typically for concrete applications $\md\in\left\{1,2,3\right\}$.

Note that the compactness assumption on $\Ma_\md$ is crucial in this article. In particular, it allows us to establish some stability estimates and the well-posedness of the ABP system. In some cases, it might be possible to remove this restriction (and consider for instance $\Ma_\md=\R^\md$), proving appropriate estimates. We leave this non trivial technical issue for future works. The assumption that $\Ma_\md$ is a smooth manifold is required to define potential energy functions $V-A\circ\xi$ with nice regularity properties.

To simplify the discussion, from now on $\Ma_\md=\tore^\md$ is the flat $\md$-dimensional torus. However, we use the abstract notation and conditions to suggest possible straightforward generalizations.

\bigskip

We are now in position to define the biased dynamics $\bigl(X_t^A\bigr)_{t\ge 0}$, for any given $A:\Ma_\md\to \R$ of class $\mathcal{C}^\infty$:
\begin{equation}\label{eq:BP}
dX_t^A=\Dr(V,A)(X_t^A)dt+\sqrt{2}\Sigma dW_t \quad , \quad X_0^A=x_0.
\end{equation}

Consistently, we will have $\Dr(V,0)=\Dr(V)$: in the absence of bias, the biased dynamics~\eqref{eq:BP} is simply the unbiased dynamics~\eqref{eq:dyn}.

\subsection{Examples of diffusions processes}\label{sec:diffusions}

In this section, we present the three main examples of diffusion processes to be studied. We postpone the study of a fourth example, given by infinite dimensional diffusion processes (SPDE), to Section~\ref{sec:SPDE}.

From now on, except in Section~\ref{sec:SPDE}, the state space $\St$ is finite dimensional.

\subsubsection{Brownian dynamics}\label{sec:OL}

\begin{itemize}
\item {\bf State space:} $\St=E_d$.
\item {\bf Reaction coordinate:} $\xi_\St=\xi$.
\item {\bf Drift coefficient:} $\Dr(V,A)=\Dr(V-A\circ \xi)=-\nabla\bigl(V-A\circ \xi\bigr)$. {\bf Diffusion operator:} $\Sigma=I$, where $I$ denotes the identity matrix.
\end{itemize}

In the Brownian case, the dynamics~\eqref{eq:BP} is written as
\[
dx_t^A=-\nabla \bigl(V-A\circ \xi\bigr)(x_t^A)dt+\sqrt{2}dW_t.
\]

In the non compact case, $E_d=\R^d$, the potential energy function $V$ is assumed to satisfy the conditions below.
\begin{hyp}\label{ass:V_OL}
When $E_d=\R^d$, there exist $\alpha_V\in(0,\infty)$ and $C_V\in\R$, such that for all $x\in E_d$,
\begin{equation*}
\langle x,\nabla V(x)\rangle\ge \alpha_V|x|^2-C_V.
\end{equation*}
Moreover, $V$ is semi-convex: $V=V_1+V_2$ where $V_1$ is a smooth bounded function, with bounded derivatives, and $V_2$ is a smooth convex function.

For all $k\in\N$, $\int_{\R^d}|x|^{k}e^{-V(x)}dx<\infty$.

Finally, there exists $c_V\in(0,\infty)$, $\varpi\in\N$, such that for all $x\in\R^d$,
\[
|V(x)|\le c_V(1+|x|^\varpi).
\]
\end{hyp}

Assumption~\ref{ass:V_OL} is satisfied for instance for smooth potential functions $V$ which behave like $|\cdot|^\varpi$ at infinity.

\begin{rem}
Assume $E_d=\R^d$. The Brownian dynamics defined above is reversible, but the framework also encompasses non-reversible situations. For instance, let $J\neq 0$ be a $d\times d$ skew-symmetric matrix. Then one may also define the drift coefficient as $\Dr(V,A)=-(I+J)\nabla\bigl(V-A\circ\xi\bigr)$.
\end{rem}

\subsubsection{Langevin dynamics}\label{sec:L}

\begin{itemize}
\item {\bf State space:} $\St=E_d\times \R^d$ (which is not compact). Elements of $\St$ are denoted by $(q,p)$.
\item {\bf Reaction coordinate:} $\xi_\St(q,p)=\xi(q)$.
\item {\bf Drift coefficient:} $\Dr(V,A)(q,p)=\begin{pmatrix}p \\ -\nabla\bigl(V-A\circ\xi\bigr)(q) -\gamma p\end{pmatrix}$. {\bf Diffusion operator:} $\Sigma=\sqrt{\gamma}\begin{pmatrix} 0 & 0\\ 0 & I\end{pmatrix}$, for $\gamma\in(0,\infty)$ a damping parameter.
\end{itemize}

\bigskip

In the Langevin case, the dynamics~\eqref{eq:BP} is written
\[
\begin{cases}
dq_t^A=p_t^Adt \quad,\quad q_0^A=q_0,\\
dp_t^A=-\nabla\bigl(V-A\circ\xi\bigr)(q_t^A)dt-\gamma p_t^Adt+\sqrt{2\gamma}d\tilde{W}_t \quad,\quad p_0^A=p_0,
\end{cases}
\]
where $\bigl(\tilde{W}_t\bigr)_{t\ge 0}$ is a standard Wiener process on $\R^d$. In applications, the variable $q$ represents positions of particles, whereas the variable $p$ represents their momenta.

The value of the damping parameter $\gamma$ plays no role in the analysis below. We recall that in the limit $\gamma\to \infty$, one recovers (up to a rescaling of the time variable) the Brownian dynamics of Section~\ref{sec:OL}, which is thus often referred to as the overdamped Langevin dynamics. Recall also the analysis of these two cases is different, since the Langevin diffusion is hypoelliptic, whereas the Brownian dynamics is elliptic.

\bigskip

In the non-compact case, $E_d=\R^d$, the potential energy function $V$ is assumed to satisfy the conditions below.
\begin{hyp}\label{ass:V_L}
When $E_d=\R^d$, there exists $\kappa_V\in(0,\infty)$ and $V_{-}\in\R$ such that $V(q)\ge \kappa_V|q|^2+V_{-}$ for all $q\in \R_d$.

There exist $A_V,B_V,\kappa_V\in(0,\infty)$ and $C_V\in\R$, such that for all $q\in \R^d$,
\begin{equation*}
\langle q,\nabla V(q)\rangle \ge A_VV(q)+B_V|q|^2+C_V.
\end{equation*}

Moreover, $V$ is semi-convex: $V=V_1+V_2$ where $V_1$ is a smooth bounded function, with bounded derivatives, and $V_2$ is a smooth convex function.

Finally, there exists $c_V\in(0,\infty)$, $\varpi\in\N$, such that for all $q\in \R^d$,
\[
|V(q)|\le c_V(1+|q|^\varpi).
\]
\end{hyp}

\subsubsection{Extended dynamics}\label{sec:E}

This example is a modification of the Brownian dynamics from Section~\ref{sec:OL}. It is straightforward to build a similar modification of the Langevin dynamics of Section~\ref{sec:L}, the details are left to the reader.

\begin{itemize}
\item {\bf State space:} $\St=E_d\times \Ma_\md$. Elements of $\St$ are denoted by $(x,z)$.
\item {\bf Reaction coordinate:} $\xi_\St(x,z)=z$.
\item {\bf Drift coefficient:} $\Dr(V,A)(x,z)=\begin{pmatrix}-\nabla_x U_A(x,z)\\ -\nabla_z U_A(x,z)\end{pmatrix}$ where $U_A(x,z)=U(x,z)-A(z)$, $U(x,z)=V(x)+\frac{1}{2\epsilon}V_{\rm ext}\bigl(\xi(x),z\bigr)$ is the extended potential energy function. It depends on a smooth function $V_{\rm ext}:\Ma_\md\times \Ma_\md\to \R$, and on $\epsilon\in(0,\infty)$. {\bf Diffusion operator:} $\Sigma$ is the identity.
\end{itemize}

In the case $\Ma_{\md}=\tore^\md$ considered here, one may choose $V_{\rm ext}\bigl(\xi(x),z\bigr)=\bigl(\xi(x)-z\bigr)^2$. Then, in the (Brownian) extended case, the dynamics~\eqref{eq:BP} is written as
\[
\begin{cases}
dX_t^A=-\nabla V(X_t^A)dt-\frac{1}{\epsilon}\langle \nabla\xi(X_t^A),\xi(X_t^A)-Z_t^A\rangle dt+\sqrt{2}dW_t^x \quad,\quad X_0^A=x_0,\\
dZ_t^A=-\frac{1}{\epsilon}\bigl(Z_t^A-\xi(X_t^A)\bigr)dt+\nabla A(Z_t^A)dt+\sqrt{2}dW_t^z \quad,\quad Z_0^A=z_0,
\end{cases}
\]
for some arbitrary initial condition $z_0\in\Ma_\md$, and where $\bigl(W_t^x\bigr)_{t\ge 0}$ and $\bigl(W_t^z\bigr)_{t\ge 0}$ are independent standard Wiener processes, on $E_d$ and on $\Ma_\md$ respectively. The dynamics is thus obtained by considering the Brownian dynamics on $E_d\times \Ma_\md$, with potential energy function $U$.

As a consequence, one could directly write the extended dynamics in the framework of Section~\ref{sec:OL}. Our choice emphasizes the role of the extended dynamics as an algorithmic tool, which is available for the practitioners.

\bigskip

We will explain below why the extended dynamics is relevant, in the limit $\epsilon\to 0$, for the problem of sampling the initial distribution, and thus why it may be sufficient to deal with the extended dynamics case. Then, since $\xi_\St(x,z)=z$ in this example, it would not be restrictive to consider reaction coordinates of the form $\xi(x_1,\ldots,x_d)=(x_1,\ldots,x_\md)$.

\bigskip

In the non compact case, $E_d=\R^d$, it is assumed that $V$ satisfies Assumption~\ref{ass:V_OL} (or Assumption~\ref{ass:V_L} if one starts from the Langevin dynamics). Then the extended potential energy function $U$ also satisfies a similar condition (recall that $\Ma_\md$ is compact) on the extended state space $E_d\times \Ma_\md$.

\subsection{Invariant probability distributions of the diffusion processes}\label{sec:invariant}

In all the examples presented above in Section~\ref{sec:diffusions}, the diffusion processes, $\bigl(X_t^0\bigr)_{t\ge 0}$ and $\bigl(X_t^A\bigr)_{t\ge 0}$, given by~\eqref{eq:dyn} and~\eqref{eq:BP}, are ergodic. The associated unique invariant distributions, defined on $\St$ (equipped with the Borel $\sigma$-field), are denoted by $\mu_\star^0$ and $\mu_\star^A$. Since the notation is consistent when $A=0$, we only deal with $\mu_\star^A$, with arbitrary $A$, in the remainder of this section.

The ergodicity in our context is understood in the following sense:
\begin{itemize}
\item there exists a unique invariant probability distribution for the Markov process $X^A$ defined by~\eqref{eq:BP}, which is equal to $\mu_\star^A$;
\item for any initial condition $x_0\in \St$, almost surely,
\[
\frac{1}{t}\int_0^t\delta_{X_\tau^A}d\tau\underset{t\to \infty}\implies \mu_\star^A,
\]
where the notation $\implies$ stands for the convergence of probability distributions on $\St$. Recall that, if $\bigl(\mu_n\bigr)_{n\in\N}$ and $\mu$ are probability distributions on $\St$, then
\[
\mu_n\underset{n\to\infty}\implies \mu
\]
if $\mu_n(\varphi)\underset{n\to \infty}\to\mu(\varphi)$ for every bounded continuous function $\varphi:\St\to \R$.
\end{itemize}

The invariant distribution $\mu_\star^A$ is expressed explictly in terms of the following data:
\begin{itemize}
\item a reference Borel, $\sigma$-finite, measure $\nuSt$ on $\St$, which does not depend on $V$ and $A$;
\item a Total Energy function $\En(V,A):\St\to \R$.
\end{itemize}

The expression of $\mu_\star^A$ is then given by:
\begin{equation}\label{eq:muA}
\mu_\star^A(dx)=\frac{\exp\bigl(-\En(V,A)(x)\bigr)}{Z^A}\nuSt(dx),
\end{equation}
where $Z^A=\int_\St \exp\bigl(-\En(V,A)(x)\bigr)\nuSt(dx)$ is a normalizing constant.

Computing averages $\mu_\star^A(\varphi)=\int_\St \varphi d\mu_\star^A$, for instance with $A=0$, is typically a challenging computational task. This may be due to the large dimension of the state space, or to the multimodality of the measure. Importance sampling techniques, as considered in this article, consist in proposing choices of functions $A$ such that it is cheaper to sample $\mu_\star^A$ than $\mu_\star$.

\bigskip

Let us make precise the reference measure $\nuSt$ and the mapping $(V,A)\mapsto \En(V,A)$ for the diffusion processes of Section~\ref{sec:diffusions}. First, the total energy function satisfies the identity
\begin{equation}\label{eq:En}
\En(V,A)=\En(V,0)-A\circ\xi_\St.
\end{equation}

It thus remains to specify the mapping $\En(V)=\En(V,0)$ and the measure $\lambda$ for each example.
\begin{itemize}
\item \underline{Brownian dynamics.} The reference measure $\lambda$ is the Lebesgue measure on $\St$. The total energy function is $\En(V)=V$.
\item \underline{Langevin dynamics.} The reference measure $\lambda$ is the Lebesgue measure on $\St$. The total energy function is the Hamiltonian function, $\En(V)(q,p)=H(q,p)=V(q)+\frac{|p|^2}{2}$. The total energy is thus the sum of potential and kinetic energies.
\item \underline{Extended dynamics.} The reference measure $\lambda$ is the Lebesgue measure on $\St=E_d\times\Ma_\md$. The total energy function is $\En(V)(x,z)=U(x,z)=V(x)+\frac{1}{2\epsilon}\bigl(\xi(x)-z\bigr)^2$.

\end{itemize}

\bigskip

The ergodicity of the dynamics~\eqref{eq:dyn}, resp.~\eqref{eq:BP}, with unique invariant probability distribution $\mu_\star^0$, resp.~$\mu_\star^A$, is well-known. We refer to Appendix~\ref{sec:appendix_finie}. In addition (see Proposition~\ref{propo:AuxiliaryPoisson}), the convergence is exponentially fast, with a rate which can be controlled, uniformly with respect to $A$, thanks to the following auxiliary result.
\begin{prop}\label{prop:unif}
Let $m,M,M^{(1)},M^{(2)},\ldots\in \R$ denote real numbers, and
\[
\mathcal{A}\subset\left\{A\in \mathcal{C}^\infty(\Ma_\md,\R) ~;~ \min A\ge m, \max A\le M, \max |\partial^k A|\le M^{(k)}, \forall k\ge 1 \right\},
\]
where $\partial^k$ denotes the derivative of order $k$.

Then
\begin{itemize}
\item if $V$ satisfies Assumption~\ref{ass:V_OL}, there exists $\alpha_{V,\mathcal{A}}\in(0,\infty)$ and $C_{V,\mathcal{A}}\in(0,\infty)$ such that for every $A\in\mathcal{A}$ and every $x\in E_d=\R^d$,
\begin{equation*}
\langle x,\nabla \bigl(V-A\circ\xi\bigr)(x)\rangle\ge \alpha_{V,\mathcal{A}}|x|^2-C_{V,\mathcal{A}}.
\end{equation*}
\item if $V$ satisfies Assumption~\ref{ass:V_L}, there exists $A_{V,\mathcal{A}},B_{V,\mathcal{A}}\in(0,\infty)$ and $C_{V,\mathcal{A}}\in(0,\infty)$ such that for every $A\in\mathcal{A}$ and every $q\in E_d=\R^d$,
\begin{equation*}
\langle q,\nabla \bigl(V-A\circ\xi\bigr)(q)\rangle \ge A_{V,\mathcal{A}}V(q)+B_{V,\mathcal{A}}|q|^2+C_{V,\mathcal{A}}.
\end{equation*}
\item For every $k\ge 1$,
\begin{equation}
\underset{A\in\mathcal{A}}\sup\int_{\St}|x|^k\mu_{\star}^{A}(dx)<\infty.
\end{equation}
\end{itemize}
\end{prop}

\bigskip

The distribution of interest, in practice, is $\mu_\star=\mu_\star^0$. However, sampling the process $X^A$ provides an approximation of $\mu_\star^A$. The following expression provides a way to compute an average $\mu_\star(\varphi)=\int_{\St}\varphi(x)\mu_\star(dx)$ in terms of averages with respect to $\mu_\star^A$: for bounded and continuous functions $\varphi:\St\to \R$,
\begin{equation}\label{eq:mu_muA}
\mu_{\star}(\ph)=\frac{\mu_{\star}^{A}\bigl(\ph \exp(-A\circ\xi_\St)\bigr)}{\mu_{\star}^{A}\bigl(\exp(-A\circ\xi_\St)\bigr)}.
\end{equation}
Using~\eqref{eq:En}, averages with respect to $\mu_\star$ may therefore be approximated by temporal averages along the biased dynamics~\eqref{eq:BP}: indeed, the quantity $\overline{\mu}_t^A$ defined by the left-hand side below satisfies
\begin{equation}\label{eq:mu_t^A}
\overline{\mu}_t^{A}(\ph):=\frac{1+\int_{0}^{t}e^{-A(\xi_\St(X_{r}^{A}))}\ph(X_r^A)dr}{1+\int_{0}^{t}e^{-A(\xi_\St(X_{r}^{A}))}dr}\underset{t\to +\infty}\to
\frac{\mu_{\star}^{A}\bigl(\ph \exp(-A\circ\xi_\St)\bigr)}{\mu_{\star}^{A}\bigl(\exp(-A\circ\xi_\St)\bigr)}=\mu_{\star}(\ph).
\end{equation}
This expression serves as the guideline for the construction of the Adaptive Biasing Potential methods~\eqref{eq:ABP_simple}, and~\eqref{eq:ABP_full} in the general case: the empirical distributions $\overline{\mu}_t$ are weighted, to ensure consistency. For well chosen functions $A$, the convergence is expected to be faster than when $A=0$. In Section~\ref{sec:free_energy} below, we identify such a function $A$, the so-called Free Energy function.

\subsection{The Free Energy function}\label{sec:free_energy}

In this section, we introduce one of the key quantities in our study: the Free Energy function $A_\star:\Ma_\md\to \R$. We explain why this function is a quantity of interest for the computational problem we are interested in, and why it is expected than choosing $A=A_\star$ in the biased dynamics~\eqref{eq:BP} leads to efficient sampling. This property is indeed the guideline of the Adaptive Biasing Potential approach of this article: we construct an adaptive version which is both consistent and designed such that $A_t$ converges to an approximation $A_\infty$ of $A_\star$ when $t\to \infty$.

\bigskip

The definition of the Free Energy function depends on the choice of a reference probability distribution $\pi$ on $\Ma_\md$.  In this article, since $\Ma_\md=\tore^\md$, it is natural to choose the Lebesgue measure, but abstract notation suggests other possible choices, see Remark~\ref{rem:pi} below.

For every smooth $A:\Ma_\md\to \R$, let $\pi_\star^A$ denote the image by $\xi_\St:\St\to\Ma_\md$ of the probability distribution $\mu_\star^A$ on $\St$. Recall that this means that for any bounded, continuous function $\phi:\Ma_\md\to \R$,
\[
\int_{\Ma_\md}\phi(z)\pi_\star^A(dz)
=\int_{\St}\phi\bigl(\xi_\St(x)\bigr)\mu_\star^A(dx).
\]
The following assumption is required.
\begin{hyp}\label{ass:pi}
The measures $\pi_\star^0$ and $\pi$ are equivalent: $\pi_\star^0$ (resp. $\pi$) is absolutely continuous with respect to $\pi$ (resp. $\pi_\star^0$).
\end{hyp}
When Assumption~\ref{ass:pi} holds true, then $\pi_\star^A$ is equivalent to $\pi$, for all smooth functions $A:\Ma_\md\to \R$. Thanks to the smoothness conditions on $\xi$, and to growth conditions on $V$, Assumption~\ref{ass:pi} is satisfied in all the examples presented above, when $\pi$ is the Lebesgue measure on $\Ma_\md=\tore^\md$.

\begin{rem}\label{rem:pi}
Another natural choice, in the periodic case $E_d=\tore^d$, for finite dimensional dynamics, is as follows: $\pi$ is defined as the image by $\xi:E_d\to \Ma_\md$ of the Lebesgue measure on $\E_d$. With this definition, $\pi$ depends on $\xi$. In the non-compact case, for instance one may define (for instance) $\pi$ as the image by $\xi$ of the standard Gaussian distribution on $\R^d$.

With these examples, Assumption~\ref{ass:pi} is satisfied by construction of $\pi$.
\end{rem}

\bigskip

We are now in position to define the free energy function $A_\star$.
\begin{defi}\label{def:FE}
The Free Energy function $A_\star:\Ma_\md\to \R$ is defined by the following property: $\exp\bigl(-A_\star(\cdot)\bigr)$ is the Radon-Nikodym derivative of $\pi_\star^0$ with respect to $\pi$.

This means that for every bounded measurable function $\phi:\Ma_\md\to \R$,
\[
\int_{\Ma_\md}\phi(z)e^{-A_\star(z)}\pi(dz)=\int_{\St}\phi\bigl(\xi_\St(x)\bigr)\mu_\star(dx).
\]
\end{defi}
Observe that, thanks to Assumption~\ref{ass:pi}, $A_\star$ takes values in $(-\infty,\infty)$. Moreover, $e^{-A_\star(z)}\pi(dz)$ is by construction a probability distribution on $\Ma_\md$, thus no normalizing constant appears on the left-hand side.

It is then straightforward to check that the Radon-Nikodym derivative of $\pi_\star^A$ with respect to $\pi$ is equal to $\exp\bigl(-A_\star+A\bigr)$, thanks to the condition~\eqref{eq:En}.

\bigskip

The function $A_\star$ may be interpreted as an effective potential energy function, for the unbiased dynamics, depending on the variable $z=\xi_\St(x)$ only. Indeed, note that for any sufficiently smooth, bounded, function $\phi:\Ma_\md\to \R$, by ergodicity of the unbiased dynamics~\eqref{eq:dyn}, with respect to $\mu_\star$, almost surely
\[
\frac{1}{t}\int_0^t\phi\bigl(\xi_\St(X_r^0)\bigr)dr\underset{t\to \infty}\to \int_\St \phi\circ\xi_\St d\mu_\star^0=\int_{\Ma_\md}\phi d\pi_\star^0=\int_{\Ma_\md}\phi(z)e^{-A_\star(z)}d\pi(z).
\]

Similarly, when considering the biased dynamics,
\[
\frac{1}{t}\int_0^t\phi\bigl(\xi_\St(X_r^A)\bigr)dr\underset{t\to \infty}\to \int_{\Ma_\md}\phi(z)e^{-A_\star(z)+A(z)}d\pi(z).
\]

\bigskip

We now give an interpretation of the qualitative properties of the free energy function $A_\star$. Assume that $\pi$ is the Lebesgue measure on $\Ma_\md$, and that $A_\star$ admits several local minima: then the distribution $\pi_\star^0$ is multimodal, and the convergence to equilibrium, when using the unbiased dynamics, is slow. Indeed, the process must visit regions near all the local minima of $A_\star$, and transitions between these metastable states are rare events. Thus $A_\star$ encodes the metastability of the dynamics along the variable $z=\xi(z)\in \Ma_\md$.

On the contrary, if the biased dynamics with $A=A_\star$ is used, the associated ergodicity result indicates that convergence is expected to be faster -- at least if the convergence in the other variables is not slow due to metastability. Indeed, the repartition of the values of $\xi_\St(X_t^A)$ tends to be uniform when $t\to \infty$; this is the flat-histogram property which is the guideline of the strategies mentioned in Section~\ref{sec:intro}.

Note also that, in many applications (for instance in molecular dynamics), computing free energy differences, {\it i.e.} $A_\star(z_1)-A_\star(z_2)$, may be the ultimate goal of the simulation, instead of computing averages $\int\varphi d\mu_\star$. The Adaptive Biasing Potential methods of this article can also be seen as efficient Free Energy computation algorithms.

Since in general the free energy function is not known, the associated biased dynamics with $A=A_\star$ cannot be simulated in practice; the guideline of the adaptive version proposed and analyzed below is to (approximately) reproduce the nice flat-histogram property for variable $z=\xi(x)$ in the asymptotic regime $t\to \infty$, without a priori knowing the free energy function $A_\star$; moreover an estimation $A_\star$ is also computed.

\section{ABP: construction and well-posedness}\label{sec:ABP}

In this section, the construction of the Adaptive Biasing Potential (ABP) system is performed in the general framework of Section~\ref{sec:framework}. The rigorous construction of the process, and the statement of appropriate assumptions, is one of the contributions of this paper. The ABP system is built starting from the unbiased dynamics~\eqref{eq:dyn}, with an adaptive bias $A=A_t$ (random and depending on time $t$) introduced in the biased dynamics~\eqref{eq:BP}. The construction is a generalization of~\eqref{eq:ABP_simple}, considered in Section~\ref{sec:intro_ABP} in a simplified setting.

In an abstract framework, the coupling of the evolutions of the diffusion process $X_t$ and of the bias $A_t$ requires the introduction of several auxiliary tools, with details provided below.
\begin{itemize}
\item A kernel function $K:\Ma_\md\times \Ma_\md\to (0,\infty)$, see Assumption~\ref{ass:kernel}. Then, a mapping $\K:~\mathcal{P}(\St)\to \mathcal{C}^\infty(\Ma_\md)$ is defined by
\begin{equation}\label{eq:K}
\K(\overline{\mu})(z)=\int_{\St}K\bigl(z,\xi_\St(x)\bigr)\overline{\mu}(dx).
\end{equation}
\item A normalization operator $\No:\mathcal{C}^0(\Ma_\md,(0,\infty))\to \mathcal{C}^0(\Ma_\md,(0,\infty))$, on the set of continuous functions on $\Ma_\md$ with values in $(0,\infty)$.
\item The notation $\overline{F}$ is defined by
\[
\overline{F}(z)=\frac{F(z)}{\int_{\Ma_\md} F d\pi}.
\]
\end{itemize}

The ABP system in its general formulation is written as follows:
\begin{equation}\label{eq:ABP_full}
\begin{cases}
dX_t=\Dr\bigl(V,A_t\bigr)(X_t)dt+\sqrt{2}\Sigma dW_t,\\
\overline{\mu}_{t}=
\frac{\overline{\mu}_0+\int_{0}^{t}F_\tau(\xi_\St(X_{\tau}))\delta_{X_\tau}d\tau}{1+\int_{0}^{t}F_\tau(\xi_\St(X_{\tau}))d\tau},\\
F_t=\No\bigl(\K(\overline{\mu}_t)\bigr)=\No\bigl(\int_{\St}K\bigl(\cdot,\xi_\St(x)\bigr)\overline{\mu}_t(dx)\bigr),\\
A_t=-\log\bigl(\overline{F}_t\bigr),
\end{cases}
\end{equation}
where there are four unknown processes: $\bigl(X_t\bigr)_{t\ge 0}$ (with values in $\St$), $\bigl(\overline{\mu}_t\bigr)_{t\ge 0}$ (with values in $\mathcal{P}(\St)$ the set of probability distributions on $\St$), $\bigl(F_t\bigr)_{t\ge 0}$ (with values in $\mathcal{C}^0(\Ma_\md,(0,\infty))$), and $\bigl(A_t\bigr)_{t\ge 0}$ (with values in $\mathcal{C}^\infty(\Ma_\md)$). Note that the initial conditions $F_0=\No\bigl(\K(\overline{\mu}_0)\bigr)$ and $A_0=-\log(\overline{F}_0)$ are prescribed by the initial condition $\overline{\mu}_0$; we also set $ X_0=x_0$.

\bigskip

Observe that it is not necessary to consider the four unknowns in~\eqref{eq:ABP_full}. Indeed, as will be explained below, $F_t$ and $\overline{F}_t=\exp(-A_t)$ only differ by a multiplicative constant (depending on $t$), which is determined only by the choice of the normalization operator $\No$. Moreover, it would be possible to consider only the processes $\bigl(X_t\bigr)_{t\ge 0}$ and $\bigl(A_t\bigr)_{t\ge 0}$ to define the dynamics of the ABP system; however, we wish to emphasize the role of the probability distribution $\overline{\mu}_t$, this is why it is included explicitly in~\eqref{eq:ABP_full}.

\bigskip

Important observations concerning the system~\eqref{eq:ABP_full} are in order.

The diffusion process is biased, following~\eqref{eq:BP}, and the bias $A_t$ at time $t$ is defined in terms of the values $\bigl(X_r\bigr)_{0\le r\le t}$ of the diffusion process up to time $t$. As a consequence, the diffusion process in~\eqref{eq:ABP_full} can be considered as a self-interacting diffusion on $\St$. However, the standard framework of self-interacting processes does not encompass the system~\eqref{eq:ABP_full}, and we thus need to adapt and generalize the arguments concerning well-posedness and convergence in our setting.

The function $A_t$ is constructed in order to be an approximation, in the regime $t\to \infty$, of the Free Energy function $A_\star$, introduced in Section~\ref{sec:free_energy}; indeed, knowing $A_\star$ would lead to an optimal non-adaptive biased dynamics. The adaptive system is designed to approximate both adaptively and efficiently $A_\star$.

As already mentioned in the introduction (see Theorem~\ref{th:cv_simple}), the central object in the analysis is the probability distribution $\overline{\mu}_t$. Indeed, we will prove that it converges almost surely to $\mu_\star$, see Theorem~\ref{th:cv_as_phi}. Note that $\overline{\mu}_t$ is defined as a weighted empirical distribution, with weights $F_\tau\bigl(\xi_\St(X_\tau)\bigr)$; this choice is motivated by~\eqref{eq:mu_muA} (in the non-adaptive setting).

Below, we state assumptions on the kernel and on the normalization operator, which play a key role first for the well-posedness of the algorithms, second for the analysis of its asymptotic behavior. In the sequel, the Assumptions on the model, stated in Section~\ref{sec:framework}, are satisfied.

\subsection{Kernel}

The kernel function $K:(z,\zeta)\in \Ma_{\md}\times \Ma_{\md}\mapsto K(z,\zeta)\in (0,+\infty)$ is a continuous, positive, smooth function. In the following, this function is often referred to as the regularization kernel, and it is assumed to satisfy the conditions below.
\begin{hyp}\label{ass:kernel}
The function $K$ is positive, of class $\mathcal{C}^\infty$ on $\Ma_\md\times\Ma_\md$. Moreover, for all $\zeta\in \Ma_\md$, the normalization condition $\int_{\Ma_\md}K(z,\zeta)\pi(dz)=1$ is satisfied.
\end{hyp}

Since $\Ma_\md$ is compact, one has $m(K)=\min_{z,\zeta \in \Ma_{\md}} K(z,\zeta)>0$, and, for all integers $r\in\left\{0,1,\ldots\right\}$, $M^{(r)}(K)=\sup_{z,\zeta \in \Ma_{\md}} |\partial_{z}^{r}K(z,\zeta)|<+\infty$. Moreover, $\sup_{z\in\Ma_\md}\sup_{\zeta_1,\zeta_2}\frac{K(z,\zeta_1)-K(z,\zeta_2)}{d(\zeta_1,\zeta_2)}<+\infty$ (Lipschitz continuity in the second variable, uniformly in the first variable).

\bigskip

The mapping $\K:\overline{\mu}\in\mathcal{P}(\St)\mapsto \K(\overline{\mu})\in \mathcal{C}^\infty(\Ma_\md)$, is then defined by~\eqref{eq:K} above. Note that $\int_\St \K(\overline{\mu})(z)\pi(dz)=1$, and that the mapping $\K(\overline{\mu})$ is of class $\mathcal{C}^\infty$, thanks to Assumption~\ref{ass:kernel}. Note also that~\eqref{eq:K} also makes sense if the probability distribution $\overline{\mu}$ is replaced with a positive, finite, measure $\mu$.

\bigskip

One may consider the following example of kernel $K$, in the case $\Ma_\md=\tore^\md$. Let $k:\R^\md\to (0,\infty)$ be an even function of class $\mathcal{C}^\infty$, with bounded derivatives, such that $\int_{\Ma_\md}k(z)\pi(dz)=1$. For $\epsilon\in(0,1)$, let $K(z,\zeta)=\frac{1}{\epsilon}k\bigl(\frac{z-\zeta}{\epsilon}\bigr)$. In the regime $\epsilon\to 0$, such kernels $K=K_\epsilon$ are smooth mollifiers. If the function $k$ is chosen with compact support, the positivity condition on $K$ is satisfied by choosing $K(z,\zeta)=\frac{\alpha}{\epsilon}k\bigl(\frac{z-\zeta}{\epsilon}\bigr)+1-\alpha$, with $\alpha\in(0,1)$.

It may also be useful to consider kernel functions which are not homogeneous, {\it i.e.} $K(z,\zeta)$ does not depend only on $z-\zeta$. For instance, set $K(z,\zeta)=\sum_{n=1}^{N}K_n(z,\zeta)\theta_n(\zeta)$, where $N\in\N$, $K_1,\ldots,K_N$ are kernel functions satisfying Assumption~\ref{ass:kernel}, and $\theta_1,\ldots,\theta_N$ are smooth functions $\Ma_\md\to (0,\infty)$, such that $\sum_{n=1}^{N}\theta_n(\zeta)=1$ for all $\zeta\in \Ma_\md$. Such examples are useful to build a bias which takes into account local properties.

\bigskip

Note that a symmetry assumption for the kernel -- $K(z,\zeta)=K(\zeta,z)$ -- is not required to prove the consistency of the approach. For instance, assume that $K(z,\zeta)=\tilde{K}(z)$ does not depend on $\zeta$; in this case, one checks that $\K(\overline{\mu}_t)=\K(\overline{\mu}_0)=\tilde{K}(\cdot)$ does not depend on $t$, and thus $A_t=A_0$: the adaptive system~\eqref{eq:ABP_full} reduces for this choice of kernel to the non-adaptive biased dynamics~\eqref{eq:BP}. Based on this observation, it is clear that the kernel $K$ is the object which governs the coupling of the evolutions of $X$ and $A$ in the adaptive dynamics~\eqref{eq:ABP_full}, and that its choice may be crucial in practice to define an efficient algorithm. In the sequel, we consider that a kernel function $K$, satisfying Assumption~\ref{ass:kernel}, is given, and do not study quantitatively the dependence with respect to $K$ of the asymptotic results.

\subsection{Normalization}

The aim of this section is to introduce normalization operators, denoted by $\No:\mathcal{C}^0(\Ma_\md,(0,\infty))\to \mathcal{C}^0(\Ma_\md,(0,\infty))$ on the set of continuous functions from $\Ma_\md$ to $(0,\infty)$. The compactness of $\Ma_\md$ plays a crucial role again. We provide below several natural families of normalization operators. However, the presentation remains abstract to emphasize the key assumptions which will lead to the stability estimates provided below.

We will use the following convention: $f$ denotes an arbitrary element in $\mathcal{C}^0(\Ma_\md,(0,\infty))$, whereas $F=\No(f)$ (capital letter) denotes its normalized version.

\bigskip

The most important example, for which a specific notation is introduced, is when normalization is meant to construct probability distributions $\overline{f}d\pi$ which are equivalent to the reference measure $\pi$ on $\Ma_\md$:
\[
\overline{f}(z)=\frac{f(z)}{\int_{\Ma_\md} f(\zeta)\pi(d\zeta)}.
\]
In the ABP system~\eqref{eq:ABP_full}, $\exp\bigl(-A_t\bigr)$ is thus the density (with respect to $\pi$) of a probability distribution on $\Ma_\md$, for every $t\ge 0$.

\bigskip

More generally, the normalization operator $\No$ is defined by
\[
\No(f)=\frac{f}{\no(f)},
\]
where $\no:\mathcal{C}^0(\Ma_\md,(0,\infty))\to (0,\infty)$ is a function which satisfies the technical (but easy to check in practice) conditions presented below.
\begin{hyp}\label{ass:normalization}
The operator $\no:\mathcal{C}^0(\Ma_\md,(0,\infty))\to (0,\infty)$ satisfies the following conditions.
\begin{itemize}
\item There exists a sequence $\bigl(\no^{(k)}\bigr)_{k\in \N}$, such that, for every $k\in\N$, $\no^{(k)}:\mathcal{C}^0(\Ma_\md,(0,\infty))\to (0,\infty)$ is continuously differentiable, and for every $f\in \mathcal{C}^0(\Ma_\md,(0,\infty))$,
\[
\no^{(k)}(f)\underset{k\to \infty}\to \no(f);
\]
moreover the convergence is assumed to be uniform on sets of the form
\[\left\{f\in \mathcal{C}^0(\Ma_\md,(0,\infty))~;~ \min f\ge m,~\max f\le M\right\},\]
for every $0<m\le M<\infty$.
\item There exists $\gamma_\no\in (0,\infty)$ such that for all $f\in \mathcal{C}^0(\Ma_\md,(0,\infty))$ and $k\in \N^*$
\begin{equation*}
\frac{1}{\gamma_\no}\min f\leq \no^{(k)}(f) \leq \gamma_\no \max f.
\end{equation*}
\item For all $f\in \mathcal{C}^0(\Ma_\md,(0,\infty))$, $\alpha\in (0,\infty)$ and $k\in \N^*$
\begin{equation*}
\no^{(k)}(\alpha f)=\alpha \no^{(k)}(f).
\end{equation*}
\item There exists $C_\no\in (0,\infty)$ such that for all $f_1,f_2\in \mathcal{C}^0(\Ma_\md,(0,\infty))$ and $k\in \N^*$
\begin{equation*}
\big|\no^{(k)}(f_1)-\no^{(k)}(f_2)\big|\leq C_\no \max|f_1-f_2|.
\end{equation*}
\end{itemize}
\end{hyp}
Only the continuous differentiability condition is relaxed when considering the limit $k\to\infty$: $\no$ is not required to satisfy this condition. The three other conditions are satisfied when $\no^{(k)}$ is replaced with $\no$.

\bigskip

Let us provide some important consequences of the definition of $\No$ in terms of an operator $\no$ satisfying Assumption~\ref{ass:normalization}. First, note that $\No\circ\No=\No$: the normalization operator is a projection. Moreover, $F=\No(f)$ and $f$ are equal up to a multiplicative constant; more generally, for two different normalization operators $\No_1$ and $\No_2$, and any function $f$, the normalized versions $F_1=\No_1(f)$ and $F_2=\No_2(f)$ are equal up to a multiplicative constant. In particular, $F=\No(\overline{F})$, and thus in the ABP system~\eqref{eq:ABP_full}, the weights $F_\tau\bigl(\xi_\St(X_\tau)\bigr)$ are not necessarily equal to $\exp\bigl(-A_\tau(\xi_\St(X_\tau))\bigr)$ like in~\eqref{eq:ABP_simple} from the introduction; however it is important to have a fixed normalization operator, since by the second condition in Assumption~\ref{ass:normalization} the ratio between these quantities remains bounded from below and from above by positive constants.

\bigskip

We conclude this section with additional examples of normalization operators.

\begin{itemize}
\item Let $q\in[1,\infty)$, and define
\[
\no_q(f)=\bigl(\int_{\Ma_\md}f(z)^q \pi(dz)\bigr)^{\frac{1}{q}}.\]
In the case $q=1$, we recover the example introduced above: $\No_1(f)=\frac{f}{\no_1(f)}=\overline{f}$.
\item Let $z_0\in\Ma_\md$, then define
\[
\no^{z_0}(f)=f(z_0)=\int_{\Ma_\md}f(z)\delta_{z_0}(dz).
\]
\item Let also
\begin{equation*}
\no_{\min}(f)=\min_{z\in\Ma_\md}f(z) \quad , \quad \no_{\max}(f)=\max_{z\in \Ma_{\md}}f(z).
\end{equation*}
For these examples, the relaxation of the continuous differentiability condition in Assumption~\ref{ass:normalization} is essential: continuously differentiable approximations are given by
\begin{equation*}
\no_{\min}(f)=\lim_{q\to +\infty}\frac{1}{\no_q(1/f)} \quad , \quad \no_{\max}(f)=\lim_{q\to +\infty}\no_q(f).
\end{equation*}
\end{itemize}

\subsection{Well-posedness}

This section is devoted to the analysis of the well-posedness of the ABP system~\eqref{eq:ABP_full}. First, Lemma~\ref{lemma:apriori} below, is stated and proved. Second, this result is combined with a Picard iteration scheme to establish global well-posedness of the self-interacting diffusion process~\eqref{eq:ABP_full}, under stronger global Lipschitz continuity conditions for the drift coefficient. Finally, a localization argument implies global well-posedness under the assumptions on $V$ stated in Section~\ref{sec:framework}.

\begin{lemma}\label{lemma:apriori}
Let $m=\frac{\min\bigl(\min h_0,m(K)\bigr)}{\gamma_\no \max\bigl(\max h_0,M^{(0)}(K)\bigr)}$ and $M^{(k)}=\frac{\max\bigl(\max h_0,m^{(k)}(K)\bigr)\gamma_\no}{\max\bigl(\min h_0,m(K)\bigr)}$, for $k\in\left\{0,1,\ldots\right\}$, where $h_0=\mathcal{K}(\overline{\mu}_0)$, $m(K)$, $M^{(k)}(K)$ are given by Assumption~\ref{ass:kernel}, and $\gamma_\no$ is given by Assumption~\ref{ass:normalization}

Let $\tau\mapsto x_\tau\in \St$ and $\tau\mapsto F_\tau\in\mathcal{C}^{0}(\Ma_\md,(0,\infty))$ be continuous mappings, such that $\no(F_\tau)=1$ for all $\tau\in\R^+$. Define
\begin{equation*}
\mu_t=\overline{\mu}_0+\int_{0}^{t}F_\tau\bigl(\xi_\St(x_\tau)\bigr)
\delta_{x_\tau}d\tau \quad , \quad h_t=\mathcal{K}(\mu_t) \quad , \quad H_t=\No\bigl(h_t\bigr).
\end{equation*}

Then, for all $t\in \R^+$, $k\in\N$, $z\in \Ma_\md$,
\[m\leq H_t(z) \leq M^{(0)} \quad,\quad |\partial^{k}H_t(z)|\leq M^{(k)}.\]
\end{lemma}

Observe that the parameters $m$ and $M^{(k)}$, for $k\in\left\{0,1,\ldots\right\}$ only depend on the algorithmic objects (the kernel function $K$ and the normalization operator $\no$) introduced in Section~\ref{sec:ABP}. On the contrary, they do not depend on the assumptions on the model from Section~\ref{sec:framework}.

\begin{proof}
We only prove the estimates on $\min H_t$ and $\max H_t$, since the proof of the estimates on the derivatives is similar. Note that
\begin{equation*}
h_t(z)=h_0(z)+\int_{0}^{t}K\bigl(z,\xi_\St(x_\tau)\bigr)F_\tau\bigl(\xi(x_\tau)\bigr)d\tau,
\end{equation*}
where $h_0=\K(\overline{\mu}_0)$, resp. $K$, are positive and continuous on $\Ma_\md$, resp. $\Ma_\md\times\Ma_\md$. Thus for all $t\in\R^+$
\begin{gather*}
\min_{z\in\Ma_\md}h_t(z)\geq \min\bigl(\min h_0,m(K)\bigr)\bigl(1+\int_{0}^{t}F_\tau\bigl(\xi_\St(x_\tau)\bigr)d\tau\bigr)\\
\max_{z\in\Ma_\md}h_t(z)\leq \max\bigl(\max h_0,M^{(0)}(K)\bigr)\bigl(1+\int_{0}^{t}F_\tau\bigl(\xi_\St(x_\tau)\bigr)d\tau\bigr).
\end{gather*}

Then the claim follows since $H_t=\frac{h_t}{\no(h_t)}$, and using the second condition in Assumption~\ref{ass:normalization}.
\end{proof}

Define sets of functions $\mathcal{F}$ and $\mathcal{A}$ as follows:
\begin{equation}\label{eq:fonctions_F}
\begin{cases}
\mathcal{F}=\left\{F\in \mathcal{C}^\infty(\Ma_\md); \min F\geq m>0, \max|\partial^{k}F|\leq M^{(k)}, k\ge 0\right\},\\
\mathcal{A}=\left\{A=-\log(\overline{F});~F\in\mathcal{F}\right\}.
\end{cases}
\end{equation}
Note that Lemma~\ref{lemma:apriori} may be combined with Property~\ref{prop:unif}.

We are now in position to state the main result of this section.
\begin{theo}[Well-posedness of~\eqref{eq:ABP_full}]\label{th:well-posed}
Grant assumptions of Section~\ref{sec:framework} concerning the model, and assumptions of Section~\ref{sec:ABP} concerning the algorithm.
\begin{itemize}
\item There exists a unique continuous process $t\in[0,\infty)\mapsto (X_t,\overline{\mu}_t,F_t,A_t)$, with values in $\St\times \mathcal{P}(\St) \times \mathcal{C}^{0}(\Ma_\md,(0,\infty))^2$, which is solution of the ABP system~\eqref{eq:ABP_full}.
\item For all $k\ge 1$, $\underset{t\ge 0}\sup~\E|X_t|^k<+\infty$.
\item For all $t\in\R^+$, $F_t\in\mathcal{F}$ and $A_t\in\mathcal{A}$, almost surely, where $\mathcal{F}$ and $\mathcal{A}$ are given by~\eqref{eq:fonctions_F}.
\end{itemize}
\end{theo}

We provide a sketch of proof of~\ref{th:well-posed}. In the arguments presented below, we emphasize the key role played by Lemma~\ref{lemma:apriori} combined with Property~\ref{prop:unif}.
\begin{proof}
Let $T\in(0,\infty)$, and define the mapping $\Psi^T$ as follows. For all $(X,F)\in L^2\bigl(\Omega,\mathcal{C}([0,T],\St)\bigr)\times L^2\bigl(\Omega,\mathcal{C}([0,T],\mathcal{C}^{1}(\Ma_\md,(0,\infty)))\bigr)$, set $\Psi^{T}(X,F)=(Z,H)$ with
\begin{gather*}
Z_t=x+\sqrt{2}W_t+\int_{0}^{t}\Dr\bigl(V,A_\tau\bigr)(X_\tau)d\tau,~A_\tau=-\log(\overline{F}_\tau),\\
\mu_t=\overline{\mu}_0+\int_{0}^{t}F_\tau\bigl(\xi_\St(X_\tau)\bigr)
\delta_{X_\tau}d\tau \quad , \quad H_t=\No\bigl(\mathcal{K}(\mu_t)\bigr),
\end{gather*}
where the mapping $\mathcal{K}$ defined by~\eqref{eq:K} is extended to positive measures. Thanks to Lemma~\ref{lemma:apriori}, the process $H$ takes values in $\mathcal{F}$. Thus any fixed point $(X,F)$ of the mapping $\Psi^T$ satisfies $F_t\in\mathcal{F}$ for all $t\ge 0$, and in the sequel we may assume that $F\in L^2\bigl(\Omega,\mathcal{C}([0,T],\mathcal{F}\cap\mathcal{C}^{1}(\Ma_\md,(0,\infty)))\bigr)$.

First, assume that $V$ has a bounded second order derivative: then $\nabla V$ is globally Lipschitz continuous. More precisely, $\mathcal{D}(V,A)$ is globally Lipschitz continuous, uniformly with respect to $A\in\mathcal{A}$:
\[
\underset{A\in\mathcal{A}}\sup~\sup_{x_1\neq x_2}\frac{|\mathcal{D}(V,A)(x_2)-\mathcal{D}(V,A)(x_1)|}{|x_2-x_1|}<\infty.
\]
We claim that there exists $C\in(0,\infty)$ such that for all $T\in(0,\infty)$, for all $(X^1,F^1)$ and $(X^2,F^2)$, such that $F_t^1,F_t^2\in\mathcal{F}$ for all $t\ge 0$, then
\[
\bigl(\E\sup_{0\le t\le T}~|Z_t^2-Z_t^1|^2\bigr)^{\frac12}\le CT\Bigl(\bigl(\E\sup_{0\le t\le T}~|X_t^2-X_t^1|^2\bigr)^{\frac12}+\E\sup_{0\le t\le T}~\|\partial A_t^2-\partial A_t^1\|_{\infty}^2\bigr)^{\frac12}\Bigr)
\]
where $A_t^2=-\log(\overline{F}_t^2)$, $A_t^1=-\log(\overline{F}_t^1)$. The structure of the mapping $A\mapsto \mathcal{D}(V,A)$ for each example of diffusion processes is exploited to obtain this estimate.

Since $F_t^1\in\mathcal{F}$ and $F_t^2\in\mathcal{F}$, note that there exists $C'\in(0,\infty)$ such that for all $t\ge 0$,
\[
\|\partial A_t^2-\partial A_t^1\|\le C'(\|F_t^2-F_t^1\|_\infty+\|\partial F_t^2-\partial F_t^1\|_\infty).
\]
Moreover, let $h_t^1=\mathcal{K}(\mu_t^1)$ and $h_t^2=\mathcal{K}(\mu_t^2)$. Then
\[
\|h_t^2-h_t^1\|_\infty\le M^{(0)}(K)T\|F_t^2-F_t^1\|_\infty+M^{(0)}M^{(1)}(K)T\underset{s\in[0,t]}\sup~\|X_s^2-X_s^1\|,
\]
more generally, for all $k\in\left\{0,1,\ldots\right\}$,
\[
\|\partial^kh_t^2-\partial^kh_t^1\|_\infty\le M^{(k)}(K)T\|F_t^2-F_t^1\|_\infty+M^{(0)}M^{(k+1)}(K)T\underset{s\in[0,t]}\sup~\|X_s^2-X_s^1\|.
\]
From the proof of Lemma~\ref{lemma:apriori} and thanks to Assumption~\ref{ass:normalization},
\[
\min\bigl(\no(h_t^1),\no(h_t^2)\bigr)\ge \gamma_{\no}^{-1}\min\bigl(\min h_0,m(K)\bigr).
\]
Then, writing
\[
H_t^2-H_t^1=\frac{h_t^2-h_t^1}{\no(h_t^2)}+h_t^1\frac{\no(h_t^1)-\no(h_t^2)}{\no(h_t^1)\no(h_t^2)},
\]
and thanks to Assumption~\ref{ass:normalization},
\[
\underset{0\le t\le T}\sup~\|H_t^2-H_t^1\|_\infty+\underset{0\le t\le T}\sup~\|\partial H_t^2-\partial H_t^1\|_\infty\le CT\bigl(\underset{0\le t\le T}\sup~\|F_t^2-F_t^1\|_\infty+\underset{0\le t\le T}\sup~\|X_t^2-X_t^1\|\bigr).
\]
Note that the parameter $C\in(0,\infty)$ does not depend on the time $T$. If $CT<1$, $\Psi^T$ is a contraction mapping, and thus admits a unique fixed point, which yields a unique local solution for the ABP sytem~\eqref{eq:ABP_full}.

In fact, a proof that the solution is in fact global, with no restriction on $T$, can be obtained by introducing a family of equivalent metrics $d_{\alpha}$ on $L^2\bigl(\Omega,\mathcal{C}([0,T],\St)\bigr)\times L^2\bigl(\Omega,\mathcal{C}([0,T],\mathcal{C}^{1}(\Ma_\md,(0,\infty)))\bigr)$:
\begin{align*}
d_{\alpha,T}\bigl((X^1,F^1),(X^2,F^2)\bigr)&=\big\| \sup_{0\leq t\leq T}e^{-\alpha t}\|X^2_t-X_t^1\|\big\|_{L^2(\Omega)}\\
&+\big\| \sup_{0\leq t\leq T}e^{-\alpha t}\|F_t^2-F_t^1\|\big\|_{L^2(\Omega)}+\big\| \sup_{0\leq t\leq T}e^{-\alpha t}\|\partial F_t^2-\partial F_t^1\|\big\|_{L^2(\Omega)}.
\end{align*}
For any fixed $T$, for large enough $\alpha$, the mapping $\Psi^T$ is a contraction when the distance $d_{\alpha,T}$ is used. The computations are left to the reader.

This argument concludes the treatment of the simpler case where $\nabla V$ is globally Lipschitz continuous (and in particular the case where the state space is compact).

The general case, when the state space is not compact, may be treated by a localization procedure. Precisely, this consists in replacing the drift coefficient $\mathcal{D}(V,A)$ with $\mathcal{D}_R(V,A)$, where $R\in(0,\infty)$, such that $\mathcal{D}_R(V,A)$ is globally Lipschitz continuous and coincides with $\mathcal{D}(V,A)$ on a ball $\mathcal{B}(0,R)$ of radius $R$. Let $(X_t^R,F_t^R)_{t\ge 0}$ denote the unique solution of the system~\ref{eq:ABP_full} where $\mathcal{D}(V,A)$ is replaced with $\mathcal{D}_R(V,A)$. This solution is global.

In each of the examples treated in this article (see Section~\ref{sec:diffusions}), the modified coefficients are constructed with applying a truncation operator to $\nabla V$ only. The result of Lemma~\ref{lemma:apriori} is not modified by this procedure.

It remains to consider exit times $\tau_R=\inf\left\{t;~X_t^R\notin \mathcal{B}(0,R)\right\}$, and to prove that, for any $T\in(0,\infty)$, $\underset{R\to\infty}\lim \mathbb{P}\bigl(\tau_R<T)\to 0$.
This result is proved thanks to moment estimates of the type
\[
\sup_{R\in(0,\infty)}~\E\bigl[\sup_{0\le t\le T}~|X_t^R|^2\bigr]<\infty.
\]
Such estimates are consequences of the assumptions on the potential energy function $V$, see Assumption~\ref{ass:V_OL} and~\ref{ass:V_L}. Details are left to the readers (see also the proof of Lemma~\ref{lem:langevin}).

Note also that for $R\le R'$, then $(X_t^R,F_t^R)=(X_t^{R'},F_t^{R'})$ for $t\le \tau_R$. Thanks to this property and the result above, it is straightforward to check that passing to the limit $R\to \infty$ provides the unique solution of~\eqref{eq:ABP_full}, on arbitrary $T\in(0,\infty)$.

This concludes the sketch of proof of Theorem~\ref{th:well-posed}.
\end{proof}

\section{Convergence results}\label{sec:results}

This section contains the main results of this article, concerning the asymptotic behavior, when $t\to\infty$, of the solution of the ABP system~\eqref{eq:ABP_full}. We first study consistency, then the efficiency, of the approach. The most important result dealing with consistency is Theorem~\ref{th:cv_as_phi}: it states almost sure convergence of averages $\overline{\mu}_t(\varphi)$ to $\mu_\star(\varphi)$ (where $\mu_\star=\mu_\star^0$, see~\eqref{eq:muA}).

Section~\ref{sec:ABF} is devoted to an interpretation of the ABP system~\eqref{eq:ABP_full} as an Adaptive Biasing Force method, and to the interpretation of the consistency results presented here in this context.

In the remainder of this section, all the Assumptions from Section~\ref{sec:framework}, on the model, and of Section~\ref{sec:ABP}, on the algorithm, are considered to be satisfied. In particular, Theorem~\ref{th:well-posed} ensures that the ABP system~\eqref{eq:ABP_full} is well defined. Moreover, the state space $\St$ is finite dimensional.

\subsection{Consistency of ABP}

\subsubsection{Convergence of weighted empirical averages}

The main result of this article concerns the consistency of the approach, for estimating averages $\mu_\star(\varphi)$ using weighted empirical averages $\overline{\mu}_t(\varphi)$ (defined by~\eqref{eq:ABP_full}).
\begin{theo}\label{th:cv_as_phi}
Let $\varphi\in\mathcal{C}^\infty(\St,\R)$ be a bounded function, with bounded derivatives of any order. Then, almost surely,
\[
\overline{\mu}_t(\varphi)\underset{t\to \infty}\to \mu_\star(\varphi).
\]
\end{theo}

This result is a generalization in the adaptive case of~\eqref{eq:mu_t^A}. The proof of Theorem~\ref{th:cv_as_phi} requires the introduction of auxiliary tools, and is provided in Section~\ref{sec:proof_consistency}. Several straightforward consequences of Theorem~\ref{th:cv_as_phi} are stated and proved in the next sections.

\subsubsection{Consequences of Theorem~\ref{th:cv_as_phi}}

\begin{cor}\label{cor:cv_as}
We have the almost sure convergence
\[
\overline{\mu}_t\underset{t\to\infty}\implies \mu_\star.
\]
\end{cor}

The notation $\implies$ for convergence of probability distributions is introduced in Section~\ref{sec:invariant}.

\begin{proof} We first state an auxiliary result: for every $\varphi:\St\to \R$, bounded and Lipschitz continuous function, almost surely
\[
\overline{\mu}_t(\varphi)\underset{t\to \infty}\to \mu_\star(\varphi)~,~\text{almost surely}.
\]
Indeed, apply Theorem~\ref{th:cv_as_phi} for an approximating sequence $\varphi_\epsilon=\rho_\epsilon\star \varphi$, defined by convolution with smooth functions $\rho_\epsilon(\cdot)=\frac{1}{\epsilon}\rho_1\bigl(\frac{\cdot}{\epsilon}\bigr)$, where $\rho$ is of class $\mathcal{C}^\infty$, with compact support, and $\int_\St\rho d\lambda=1$.

Let ${\rm BL}(\St,\R)=\left\{\varphi:\St\to \R~;~ \varphi~\text{bounded and Lipschitz continuous}\right\}$. Then there exists a sequence of functions $\bigl(\varphi_n\bigr)_{n\ge 0}$ defined from $\St$ to $\R$, bounded and Lipschitz continuous, such that
\[
\overline{\mu}_t\underset{t\to \infty}\implies \mu_\star ~\Longleftrightarrow~ d(\overline{\mu}_t,\mu_\star)\underset{t\to \infty}\to 0,
\]
where
\[
d(\mu^1,\mu^2)=\sum_{n=0}^{\infty}\frac{1}{2^n}\min\bigl(1,\big|\int_\St \varphi_n d\mu^1-\int_\St \varphi_n d\mu^2\big|\bigr).
\]
Thanks to the convergence result above, almost surely, for every $n\ge 0$, $\overline{\mu}_t(\varphi_n)\underset{t\to\infty}\to \mu_\star(\varphi)$, and thus $d(\overline{\mu}_t,\mu_\star)\underset{t\to \infty}\to 0$ almost surely.

This concludes the proof of Corollary~\ref{cor:cv_as}.
\end{proof}

The following result deals with the almost sure convergence of the functions $\overline{F}_t$ and $A_t$. Note that contrary to Theorem~\ref{th:cv_as_phi} and Corollary~\ref{cor:cv_as}, the limits $\overline{F}_\infty$ and $A_\infty$ depend on the parameters of the algorithm, precisely on the kernel function $K$. Note that these almost sure limits are not random.

The convergence of $A_t$ to $A_\infty$, which is close to the Free Energy function $A_\star$ for well-chosen kernel functions, is one of the nice features of the ABP method, in particular when one is interested in computing free energy differences.

\begin{cor}\label{cor:cv_A}
Define, for all $z\in \Ma_\md$,
\[
\begin{cases}
\overline{F}_\infty(z)=\mu_\star\bigl(K(z,\cdot)\bigr),\\
A_\infty(z)=-\log(\overline{F}_\infty(z)).
\end{cases}
\]
Then, almost surely, for every $\ell\in\left\{0,1,\ldots\right\}$, uniformly on $\Ma_\md$,
\[
\begin{cases}
\partial^\ell\overline{F}_t\underset{t\to \infty}\to \partial^\ell\overline{F}_\infty,\\
\partial^\ell A_t\underset{t\to \infty}\to \partial^\ell A_\infty.
\end{cases}
\]
\end{cor}

\begin{proof}
The result is a consequence of the regularity properties of the kernel mapping $K$, of Ascoli's theorem, and of Theorem~\ref{th:cv_as_phi}.

Let $\K:\mathcal{P}(\St)\to\mathcal{C}^\infty(\Ma_\md)$ be the mapping defined by~\eqref{eq:K}.

Let $\bigl(z_n\bigr)_{n\in\N}$ denote a dense sequence in $\Ma_\md$, and define, for all $\mu^1,\mu^2\in\mathcal{P}(\St)$,
\[
d_\infty(\mu^1,\mu^2)=\sum_{\ell,n=0}^{\infty}\frac{1}{2^{\ell+n}}\min\bigl(1,\big|\int_{\St}\partial_z^\ell K(z_n,\xi_\St(\cdot))d\mu^1-\int_{\St}\partial_z^\ell K(z_n,\xi_\St(\cdot))d\mu^2\big|\bigr).
\]
Then for any sequence $\bigl(\mu^k\bigr)_{k\in\N}$ and any $\mu$ in $\mathcal{P}(\St)$,
\begin{itemize}
\item if $\mu^k\underset{k\to \infty}\implies\mu$, then $d_\infty(\mu^k,\mu)\underset{k\to \infty}\to 0$;
\item if $d_\infty(\mu^k,\mu)\underset{k\to \infty}\to 0$, then for every $\ell\in\left\{0,1,\ldots\right\}$,
\[
\partial^\ell\K(\mu^k)\underset{k\to\infty}\to\partial^\ell\K(\mu),
\]
uniformly on $\Ma_\md$, thanks to Ascoli's theorem and the bound $\|\partial_z^{k+1}K\|_\infty\le M^{(k+1)}(K)$.
\end{itemize}

Thanks to Theorem~\ref{th:cv_as_phi}, it is straightforward to conclude that almost surely
\[
d_\infty(\overline{\mu}_t,\mu_\star)\underset{t\to\infty}\to 0.
\]
These arguments yield the convergence of $\overline{F}_t$. The convergence of $A_t=-\log(\overline{F}_t)$ is then obtained thanks to the almost sure lower bound from Theorem~\ref{th:well-posed},
\[
\underset{\Ma_\md}\min \overline{F}_t\ge m>0.
\]
\end{proof}

\subsubsection{A remark concerning convergence of the gradient of $A_t$}\label{sec:ABF}

To keep notation simple, consider the framework of Section~\ref{sec:intro_ABP}: the diffusion process is the Brownian dynamics on $\tore^d$, and $\xi(x_1,\ldots,x_d)=x_1\in \tore$, {\it i.e.} $\md=1$. Assume in addition that the kernel $K$ is symmetric, $K(z,\zeta)=K(\zeta,z)$.

The main observation in this section is that, for the ABP method, one may write the derivative $\partial_{x_1}A_t(x_1)$ of $A_t$ as a conditional expectation, up to introducing an additional variable.

This observation is motivated by the following statement: the Free Energy function $A_{\star}$ satisfies the identity (expression of the equilibrium mean force)
\begin{equation*}
A_{\star}'(x_1)=\int_{\tore^{d-1}}(\partial_{x_1}V(x))e^{-(V(x)-A_{\star}(x_1))}dx_2\ldots dx_d=\E_{X\sim \mu_{\star}}[\partial_{x_1}V(X) \big| X_1=x_1],
\end{equation*}
where in the conditional expectation the random variable $X$ is distributed according to $\mu_{\star}$. This identity is the starting point for constructions of Adaptive Biasing Force (ABF) methods mentioned in Section~\ref{sec:intro}.

Such a formula does not hold for $A_t$, when $t<\infty$. However, the following generalization may be used. On the one hand, for all $t\ge 0$,
\begin{equation*}
A_t'(z)=-\frac{\int_{\tore^{d-1}}\partial_zK(z,x_1)\overline{\mu}_t(dx)}{\int_{\tore^{d-1}}K(z,x_1)\overline{\mu}_t(dx)}=\E_{(X,Z)\sim\eta_t}\Bigl[-\frac{\partial_zK(Z,\xi(X))}{K(Z,\xi(X))} \Big| Z=z \Bigr],
\end{equation*}
where $\eta_t(dx,dz)=K\bigl(z,\xi(x)\bigr)\overline{\mu}_t(dx)dz$ is a probability distribution on $\tore^{d}\times \tore$, which depends on the kernel function $K$. Observe that if $(X,Z)\sim \eta_t$, in general $Z\neq \xi(X)$, hence the need of the new notation instead of conditional expectations. On the other hand, the expression above for the equilibrium mean force can be rewritten in the similar form
\[
A_{\star}'(z)=\E_{(X,Z)\sim\eta_\star}[\partial_{x_1}V(X) \big| Z=z],
\]
where $\eta_\star(dx,dz)=\mathds{1}_{z=x_1}\mu_{\star}(dx)dz$. Note that if $(X,Z)\sim\eta_\star$, then the equality $Z=\xi(X)$ is now satisfied.

Let us now check that these expressions are consistent with Corollary~\ref{cor:cv_A}. Letting $t\to \infty$, thanks to Corollary~\ref{cor:cv_as}, it is straightforward to check that $\eta_t$ converges almost surely to $\eta_{\infty}(dx,dz)=K\bigl(z,x_1\bigr)\mu_\star(dx)dz$. We thus obtain different expressions of $A_\infty'(z)$:
\begin{align*} A_{\infty}'(z)&=\E_{(X,Z)\sim\eta_t}\Bigl[-\frac{\partial_zK(Z,\xi(X))}{K(Z,\xi(X))} \Big| Z=z \Bigr]\\
&=-\frac{\int_{\tore^{d-1}}\partial_zK(z,x_1){\mu}_\star(dx)}{\int_{\tore^{d-1}}K(z,x_1)\mu_\star(dx)}\\
&=\frac{\int_{\tore^{d-1}}\partial_{x_1}V(x)K(z,x_1){\mu}_\star(dx)}{\int_{\tore^{d-1}}K(z,x_1)\mu_\star(dx)}\\
&=\E_{(X,Z)\sim \eta_{\infty}}[\partial_{x_1}V(X) \big| Z=z],
\end{align*}
thanks to the use of an integration by parts formula. Due to the presence of the kernel function $K$, $\eta_\star\neq\eta_\infty$, and thus $A_{\star}'(z)\neq A_{\infty}'(z)$.

The observation above may be the starting point for other types of Adaptive Biasing methods, based on a single realization of the stochastic process and a self-interaction mechanism using an empirical distribution.

\subsection{Applications to the diffusion processes of Section~\ref{sec:diffusions}}

The aim of this section is to specify, for each of the examples of diffusion processes from Section~\ref{sec:diffusions}:
\begin{itemize}
\item the convergence result of Theorem~\ref{th:cv_as_phi}, for well chosen test functions $\varphi$;
\item the expression of the limit $\overline{F}_\infty=e^{-A_\infty}$, in terms of the kernel $K$ and of the free energy function $A_\star$.
\end{itemize}

We introduce the probability distribution $\mu_\star^{{\rm ref}}(dx)=\frac{e^{-V(x)}}{\int_{E_d}e^{-V(y)}dy}dx$ on $E_d$. Observe that in all the examples $\mu_\star^{{\rm ref}}$ is the marginal of the distribution $\mu_\star$ with respect to its $E_d$-valued component (the equality $\mu_\star^{{\rm ref}}=\mu_\star$ holds true only in the Brownian case). As a consequence, the practitioner may choose one of the three dynamics (Brownian, Langevin or extended dynamics) of Section~\ref{sec:diffusions} to estimate averages $\mu_\star^{{\rm ref}}(\varphi)$.

We also denote by $A_\star^{{\rm ref}}$ the Free Energy function associated with the reaction coordinate $\xi$ and the probability distribution $\mu_\star^{{\rm ref}}$: by definition, $e^{-A_\star^{{\rm ref}}}$ is the Radon-Nikodym derivative of the image of $\mu_\star^{{\rm ref}}$ by $\xi$, with respect to the probability distribution $\pi$ on $\Ma_\md$.

Assume that the kernel $K=K_\delta$ depends on $\delta>0$, and is such that the probability distribution $K_\delta(z,\zeta)\pi(dz)\pi(d\zeta)$ converges when $\delta\to 0$, to $\delta_{z}(d\zeta)\pi(dz)$. Then, when $\delta\to 0$ (and also $\epsilon\to 0$, in the extended dynamics case), the expressions below prove that $A_\infty$ is an approximation of $A_\star^{{\rm ref}}$. We do not provide quantitative estimates.

\subsubsection{Brownian dynamics (Section~\ref{sec:OL})}

\begin{itemize}
\item \underline{Computation of averages:} for every $\varphi\in\mathcal{C}^\infty(E_d,\R)$, bounded and with bounded derivatives, almost surely
\[
\int\varphi d\mu_\star^{{\rm ref}}=\underset{t\to \infty}\lim \frac{1+\int_{0}^{t}F_\tau(\xi(X_\tau))\varphi(X_\tau)d\tau}{1+\int_{0}^{t}F_\tau(\xi(X_\tau))d\tau}.
\]
\item \underline{Free Energy function:}
\[
e^{-A_\infty(\cdot)}=\overline{F}_\infty(\cdot)=\int_{E_d}K_\delta(\cdot,\xi(x))\mu_\star^{{\rm ref}}(dx)=\int_{\Ma_\md}K_\delta(\cdot,\zeta)e^{-A_\star^{{\rm ref}}(\zeta)}\pi(d\zeta).
\]
\end{itemize}

In particular, Theorem~\ref{th:cv_simple}, stated in Section~\ref{sec:intro_ABP} and taken from~\cite{BB16}, is a consequence of Corollaries~\ref{cor:cv_as} and~\ref{cor:cv_A}, in the simplified context.

\subsubsection{Langevin dynamics (Section~\ref{sec:L})}

We use the notation $X_t=(q_t,p_t)$.

\begin{itemize}
\item \underline{Computation of averages:} for every $\varphi\in\mathcal{C}^\infty(E_d,\R)$, bounded and with bounded derivatives, almost surely
\[
\int\varphi d\mu_\star^{{\rm ref}}=\underset{t\to \infty}\lim \frac{1+\int_{0}^{t}F_\tau(\xi(q_\tau))\varphi(q_\tau)d\tau}{1+\int_{0}^{t}F_\tau(\xi(q_\tau))d\tau}.
\]
\item \underline{Free Energy function:}
\[
e^{-A_\infty(\cdot)}=\overline{F}_\infty(\cdot)=\int_{E_d}K_\delta(\cdot,\xi(q))\mu_\star^{{\rm ref}}(dq)=\int_{\Ma_\md}K_\delta(\cdot,\zeta)e^{-A_\star^{{\rm ref}}(\zeta)}\pi(d\zeta).
\]
\end{itemize}

Observe that the free energy function $A_\infty$ is the same for the Brownian and the Langevin dynamics. This identity is in fact obtained since $\xi_\St(q,p)=\xi(q)$ only depends on $q\in E_d$.

\subsubsection{Extended dynamics (Section~\ref{sec:E})}

We use the notation $(X_t,Z_t)$. Recall that $\xi_\St(x,z)=z$ in this case.

\begin{itemize}
\item \underline{Computation of averages:} for every $\varphi\in\mathcal{C}^\infty(E_d,\R)$, bounded and with bounded derivatives, almost surely
\[
\int\varphi d\mu_\star^{{\rm ref}}=\underset{t\to \infty}\lim \frac{1+\int_{0}^{t}F_\tau(Z_\tau)\varphi(X_\tau)d\tau}{1+\int_{0}^{t}F_\tau(Z_\tau)d\tau}.
\]
\item \underline{Free Energy function:}
\begin{align*}
e^{-A_\infty(\cdot)}=\overline{F}_\infty(\cdot)&=\int_{E_d\times \Ma_\md}K(\cdot,z)\mu_\star(dxdz)\\
&=\int_{E_d\times\Ma_\md}K(\cdot,z)K_{\epsilon}^{{\rm ext}}(z,\xi(x))\mu_{\star}^{{\rm ref}}(dx)\pi(dz)\\
&=\int_{\Ma_\md}\Bigl(\int_{\Ma_\md}K(\cdot,z)K_{\epsilon}^{{\rm ext}}(z,\zeta)\pi(dz)\Bigr)e^{-A_\star^{{\rm ref}}(\zeta)}\pi(d\zeta),
\end{align*}
\end{itemize}
where we have introduced the auxiliary kernel $K_\epsilon^{{\rm ext}}:\Ma_\md\times\Ma_\md\to (0,\infty)$, such that $\mu_\star(dxdz)=K_{\epsilon}^{{\rm ext}}(z,\xi(x))\mu_\star^{{\rm ref}}(dx)\pi(dz)$: up to a multiplicative constant, $K_\epsilon^{{\rm ext}}(z,\zeta)=\exp\bigl(-\frac{1}{2\epsilon}|z-\zeta|^2\bigr)$. Note that the expression of $A_\infty$ is not the same as in the previous examples, due to the additional term in the definition of the extended potential energy function on $E_d\times \Ma_\md$. However, when $\epsilon\to 0$, $A_\infty$ converges to $A_\star^{{\rm ref}}$: this observation justifies the use of the extended dynamics in the context of free energy computations.

\subsection{Efficiency}\label{sec:efficiency}

We now state and prove a series of results concerning the efficiency of the approach, first in a qualitative way, second with a more quantitative approach. Corollary~\ref{cor:conv_rho} deals with the convergence of the non-weighted empirical distribution $\overline{\rho}_t$, defined by~\eqref{eq:rho}; it is a straightforward consequence of Corollary~\ref{cor:cv_as}. Proposition~\ref{propo:variance} deals with the mean-square error, and identifies an asymptotic variance. Since the proof of Proposition~\ref{propo:variance} requires tools introduced in Section~\ref{sec:proof_consistency}, we postpone its proof to Section~\ref{sec:proof_variance}.

In terms of the behavior of the occupation measure and of the asymptotic variance, the results stated below may be interpreted as follows: in the asymptotic regime $t\to \infty$, the Adaptive Biasing Potential method~\eqref{eq:ABP_full} performs in the same way as the non-adaptive Biasing Potential method~\eqref{eq:BP}, with the bias $A=A_\infty$.

Note that these results are asymptotic, when $t\to\infty$; it would also be interesting to study more quantitatively the convergence, for each of the results. This question is left for future works.

\subsubsection{Convergence of non-weighted empirical distributions}

In this section, we focus on the convergence of non-weighted empirical averages $\overline{\rho}_t(\ph)$, where $\overline{\rho}_t$ is the probability distribution on $\St$ defined by
\begin{equation}\label{eq:rho}
\overline{\rho}_t=\frac{\overline{\mu}_0+\int_{0}^{t}\delta_{X_\tau}d\tau}{1+t}.
\end{equation}
We refer to $\overline{\rho}_t$ as the non-weighted empirical distribution, or as the occupation measure, associated with the diffusion process $\bigl(X_t\bigr)_{t\ge 0}$ defined by~\eqref{eq:ABP_full}. We have the following result.
\begin{cor}\label{cor:conv_rho}
Let $\varphi\in\mathcal{C}^\infty(\St,\R)$ be a bounded function, with bounded derivatives of any order. Then
\begin{equation}\label{eq:cor_conv_rho}
\overline{\rho}_t(\ph)\underset{t\to +\infty}\to\mu_{\star}^{A_\infty}(\ph)~,~\text{almost surely},
\end{equation}
where $A_\infty=\underset{t\to \infty}\lim A_t$ (see Corollary~\ref{cor:cv_A}), and $\mu_\star^{A_\infty}$ is given by~\eqref{eq:muA}.

Moreover, almost surely, $\overline{\rho}_t\underset{t\to \infty}\implies \mu_\star^{A_\infty}$.
\end{cor}

The arguments below justify that Corollary~\ref{cor:conv_rho} can be interpreted, qualitatively, as an efficiency property of the ABP method.

First, observe that considering the biased dynamics $\bigl(X_t^A\bigr)_{t\ge 0}$ given by~\eqref{eq:BP}, and setting
\[
\overline{\rho}_t^A=\frac{\overline{\mu}_0+\int_{0}^{t}\delta_{X_\tau^A}d\tau}{1+t},
\]
then almost surely $\overline{\rho}_t^A(\varphi)\underset{t\to \infty}\to \mu_{\star}^{A}(\varphi)$. The limit in~\eqref{eq:cor_conv_rho}, when the adaptive dynamics is used, is the same as when using the non-adaptive dynamics~\eqref{eq:BP}, with $A=A_\infty$.

Second, observe that the image by the mapping $\xi_\St:\St\to \Ma_\md$ of the probability distribution $\mu_\star^A$ has density with respect to $\pi$ proportional to
\[
\exp\bigl(-A_{\star}+A\bigr).
\]
This density is constant, equal to $1$, when $A=A_\star$: this means that in the asymptotic limit $t\to \infty$, the values of $\xi_\St(X_t^{A_\star})$ are distributed according to the reference probability distribution $\pi$. On the contrary, when $A=0$, the values of $\xi_\St(X_t^{0})$ are distributed according to $\pi_\star^0=e^{-A_\star}d\pi$.

Assume that $\pi$ is the uniform distribution on $\Ma_\md=\tore^\md$; assume also that all the metastability of the system is encoded by the reaction coordinate $\xi$. If $A_\star$ has several local minima, then $\pi_\star^0$ is a multimodal distribution, and the diffusion process $\bigl(X_t^0\bigr)_{t\ge 0}$ is metastable, and does not efficiently sample all the state space. Thus the convergence of $\overline{\rho}_t^0$ to $\mu_\star^0$ is expected to be slower than the convergence of $\overline{\rho}_t^{A_\star}$ to $\mu_{\star}^{A_\star}$. Indeed, the exploration of the metastable states tends to be uniform, when $t\to\infty$, when observed through the reaction coordinate mapping.

Since $A_\infty$ is an approximation of the Free Energy function $A_\star$, for well-chosen kernel functions $K$, efficiency of the ABP method is justified by the observations above.

We now provide the proof of Corollary~\ref{cor:conv_rho}, with elementary arguments. The proof of the almost sure convergence of the probability distributions is obtained as in the proof of Corollary~\ref{cor:cv_as}, therefore we only focus on the convergence of averages $\overline{\rho}_t(\varphi)$.

\begin{proof}[Proof of Corollary~\ref{cor:conv_rho}]

Introduce the auxiliary measure
\begin{equation}
\rho_t= \frac{\overline{\mu}_0+\int_{0}^{t}\delta_{X_\tau}d\tau}{1+\int_{0}^{t}F_\tau\circ\xi_{\St}(X_\tau)d\tau} =\frac{(1+t)\overline{\rho}_t}{1+\int_{0}^{t}F_\tau\circ\xi_{\St}(X_\tau)d\tau}.
\end{equation}
Since the measures $\rho_t$ and $\overline{\rho}_t$ only differ by a multiplicative (normalization) constant, one has the identity $\overline{\rho}_t=\frac{\rho_t}{\rho_t(1)}$. Then, note that
\begin{align*}
\rho_t(\ph)&=\frac{\overline{\mu}_0(\ph)+\int_{0}^{t}F_\tau\circ\xi_\St(X_\tau)\frac{\ph(X_\tau)}{F_\tau\circ\xi_\St(X_\tau)}d\tau}{1+\int_{0}^{t}F_\tau\circ\xi_{\St}(X_\tau)d\tau}\\
&=\frac{\overline{\mu}_0(\ph)+\int_{0}^{t}F_\tau\circ\xi_\St(X_\tau)\frac{\ph(X_\tau)}{F_\infty\circ\xi_\St(X_\tau)}d\tau}{1+\int_{0}^{t}F_\tau\circ\xi_{\St}(X_\tau)d\tau}\\
&~+\frac{1}{1+\int_{0}^{t}F_\tau\circ\xi_{\St}(X_\tau)d\tau}\int_{0}^{t}F_\tau\circ\xi_\St(X_\tau)\ph(X_\tau)\bigl(\frac{1}{F_\tau\circ\xi_\St(X_\tau)}-\frac{1}{F_\infty\circ\xi_\St(X_\tau)}\bigr)d\tau\\
&=\overline{\mu}_t\bigl(\frac{\ph}{F_{\infty}\circ\xi_\St}\bigr)+o(1),
\end{align*}
using the following version of Cesaro's Lemma: if $a:[0,\infty)\to \mathbb{R}$ is a continuous function such that $a(t)\underset{t\to \infty}\to 0$, then $\frac{1}{t}\int_0^t a(\tau)d\tau\underset{t\to \infty}\to 0$. This result may be applied, thanks to the almost sure lower bound $1+\int_{0}^{t}F_\tau\circ\xi_{\St}(X_\tau)d\tau\ge 1+mt$; moreover thanks to Corollary~\ref{cor:cv_A}, $F_t=\No(\overline{F}_t)\underset{t\to +\infty}\to F_\infty=\No(\overline{F}_\infty)$, uniformly on $\Ma_\md$, almost surely.

Moreover, the function $\frac{\ph}{F_\infty\circ \xi_\St}$ is bounded and of class $\mathcal{C}^\infty$, with bounded derivatives (using $\min F_\infty\ge m>0$ thanks to Theorem~\ref{th:well-posed}). Applying Theorem~\ref{th:cv_as_phi}, almost surely
\begin{align*}
\overline{\rho}_t(\ph)=\frac{\rho_t(\ph)}{\rho_t(1)}\underset{t\to +\infty}\to&\frac{\mu_\star\Bigl(\ph/F_{\infty}\bigl(\xi_\St(\cdot)\bigr)\Bigr)}{\mu_\star\Bigl(1/F_{\infty}\bigl(\xi_\St(\cdot)\bigr)\Bigr)}
=\frac{\mu_\star\Bigl(\ph/\overline{F}_{\infty}\bigl(\xi_\St(\cdot)\bigr)\Bigr)}{\mu_\star\Bigl(1/\overline{F}_{\infty}\bigl(\xi_\St(\cdot)\bigr)\Bigr)}\\
&=\frac{\int_{\St}\ph(x)
\exp\bigl(-\bigl(\En(V)(x)-A_\infty(\xi_\St(x))\bigr)\bigr)\lambda(dx)}{\int_{\St}
\exp\bigl(-\bigl(\En(V)(x)-A_\infty(\xi_\St(x))\bigr)\bigr)\lambda(dx)}\\
&=\frac{\int_{\St}\ph(x)
\exp\bigl(-\bigl(\En(V,A_\infty)(x)\bigr)\bigr)\lambda(dx)}{\int_{\St}
\exp\bigl(-\bigl(\En(V,A_\infty)(x)\bigr)\bigr)\lambda(dx)}\\
&=\mu_\star^{A_\infty}(\ph),
\end{align*}
thanks to the identity~\eqref{eq:En}, and to~\eqref{eq:muA}. This concludes the proof.
\end{proof}

\subsubsection{Asymptotic mean-square error}

This section is devoted to a more quantitative approach, concerning the behavior when $t\to\infty$ of the mean-square error
\[
\E\big|\overline{\mu}_t(\varphi)-\mu_\star(\varphi)\big|^2,
\]
for functions $\varphi\in\mathcal{C}^\infty(\St,\R)$, bounded and with bounded derivatives.

In order to compare the performance of the adaptive and non-adaptive versions of the biasing potential approach, introduce the following quantity
\[
\mathbf{V}_\infty(\varphi,A)=\underset{t\to \infty}\limsup ~t\E|\overline{\mu}_t^A(\varphi)-\mu_\star(\varphi)|^2\in[0,\infty],
\]
where $A:\Ma_\md\to\R$ is fixed, $\overline{\mu}_t^A(\varphi)$ is the estimator of $\mu_\star(\varphi)$ defined by the left-hand side of~\eqref{eq:mu_t^A}, for every $t\ge 0$, using the biased dynamics~\eqref{eq:BP}.

In Section~\ref{sec:proof_variance}, it will be proved that in fact
\[
\mathbf{V}_\infty(\varphi,A)=\underset{t\to \infty}\lim~t\E|\overline{\mu}_t^A(\varphi)-\mu_\star(\varphi)|^2\in(0,\infty)
\]
is a non-degenerate limit.

The following result, concerning the asymptotic mean-square error of the estimator $\overline{\mu}_t(\varphi)$ of $\mu_\star(\varphi)$, constructed using the adaptively biased dynamics~\eqref{eq:ABP_full}.
\begin{propo}\label{propo:variance}
Let $\varphi\in\mathcal{C}^\infty(\St,\R)$ be a bounded function, with bounded derivatives of any order. Then
\[
t\E|\overline{\mu}_t(\varphi)-\mu_\star(\varphi)|^2\underset{t\to \infty}\to\mathbf{V}_\infty(\varphi,A_\infty),
\]
where $A_\infty=\underset{t\to \infty}\lim A_t$ almost surely, see Corollary~\ref{cor:cv_A}.
\end{propo}

As already explained, the asymptotic mean-square error for the adaptive version is the same as for the non-adaptive version, where the bias is chosen as $A=A_\infty$. Note that the dependence of $\mathbf{V}_\infty(\varphi,A)$ with respect to $A$ depends a lot on the choice of the function $\varphi$; therefore no optimality result is stated.

The proof of Proposition~\ref{propo:variance} is postponed to Section~\ref{sec:proof_variance}; explicit expressions for $\mathbf{V}_\infty(\varphi,A)$, in terms of the solutions of Poisson equations, are given there.

\section{Proof of Theorem~\ref{th:cv_as_phi}}\label{sec:proof_consistency}

The aim of this section is to provide a detailed proof of Theorem~\ref{th:cv_as_phi}.

First, in Sections~\ref{sec:time_change} and~\ref{sec:ODE}, we present the strategy, and in particular we establish a connexion with the analysis of self-interacting diffusions from~\cite{BLR02}, and more generally of stochastic algorithms, see~\cite{B99},~\cite{BMP90},~\cite{D97},~\cite{KY03}. More precisely, Section~\ref{sec:time_change} presents a (random) change of time variable, $s=\theta(t)$, which transforms the weighted empirical distributions $\overline{\mu}_t$ associated with the process $X_t$, into non-weighted empirical distributions $\overline{\nu}_s$ associated with a process $Y_s$, with modified dynamics. In Section~\ref{sec:ODE}, we explain how the so-called ODE method can be exploited: the asymptotic behavior of $\overline{\nu}_s$, when $s\to \infty$, is related to the behavior of a differential equation of the type $\dot{\nu}=-\nu+\Pi(\nu)$. A crucial result, Proposition~\ref{propo:ergo_frozen}, states that $\Pi(\nu)=\mu_\star$ is a constant mapping, and the dynamics of the differential equation above is extremely simple.

The analysis is thus expected to be much simpler than in~\cite{BLR02}. Indeed, in Section~\ref{sec:proof_decomp_error}, we directly prove the almost sure convergence of $\overline{\mu}_t(\varphi)-\mu_\star(\varphi)$ to $0$ when $t\to \infty$. Results concerning Poisson equations are stated, their proofs being postponed to Section~\ref{sec:appendix_finie}.

Even if it is not explictly used in the technical part of the proof of Theorem~\ref{th:cv_as_phi}, the description of the change of time variable strategy is included for pedagogical purpose. Moreover, in our opinion, it is an elegant way to justify the consistency of the approach. Moreover, it may be a useful strategy in other similar situations. Readers only interested in the proof of Theorem~\ref{th:cv_as_phi} may skip Sections~\ref{sec:time_change} and~\ref{sec:ODE} -- except for Proposition~\ref{propo:ergo_frozen} which is used in the sequel.

\subsection{Approach from a stochastic approximation perspective}\label{sec:SA_perspective}

\subsubsection{Change of time variable}\label{sec:time_change}

In this section, we introduce a random change of time variable, and describe some of its nice properties. This is only a mathematical tool, and does not need to be performed in practice when implementing the method. In addition, as explained above, this change of variable has only a pedagogical role, and will not be used in the technical details of the proof.

Consider the solution of the ABP system~\eqref{eq:ABP_full}. Then the mapping $t\mapsto \overline{\mu}_t\in\mathcal{P}(\St)$ is the unique solution of the following Ordinary Differential Equation (ODE)
\begin{equation}\label{eq:ODE}
\frac{d\overline{\mu}_t}{dt}=\frac{\theta'(t)}{1+\theta(t)}\bigl(\delta_{X_t}-\overline{\mu}_t\bigr),\quad \theta(t)=\int_{0}^{t}F_\tau\bigl(\xi_\St(X_\tau)\bigr)d\tau.
\end{equation}

The ODE~\eqref{eq:ODE} is interpreted in the following weak sense: for every bounded continuous test function $\varphi:\St\to \R$, the real-valued function $t\mapsto \overline{\mu}_t(\varphi)=\int_{\St}\varphi d\overline{\mu}_t\in \R$ is the unique solution of the differential equation
\[
\frac{d\overline{\mu}_t(\varphi)}{dt}=\frac{\theta'(t)}{1+\theta(t)}\bigl(\varphi(X_t)-\overline{\mu}_t(\varphi)\bigr),
\]
with the initial condition $\overline{\mu}_0(\varphi)$.

Define the measure $\mu_t=\overline{\mu}_0+\int_{0}^{t}F_\tau\bigl(\xi_\St(X_\tau)\bigr)\delta_{X_\tau}d\tau$. Then observe that $\theta(t)=\mu_t(1)$ can be interpreted as a normalizing constant.

The presence of the random variable $\theta(t)$ in the ODE~\eqref{eq:ODE} suggests that the analysis will be not trivial. However, we can remove this quantity thanks to a change of time variable. Simultaneously, this procedure removes the weights in the definition of the measure $\overline{\mu}_t$, and the dynamics of the stochastic process $X_t$ is modified.

\bigskip

Thanks to Theorem~\ref{th:well-posed}, there exist two non-random real numbers $0<m\le M$ such that almost surely $\theta'(t)=F_t\bigl(\xi_\St(X_t)\bigr)\in[m,M]$ for all $t\ge 0$. Moreover, $\theta(0)=0$, and $\theta(t)\ge mt\underset{t\to \infty}\to\infty$. As a consequence, almost surely, $\theta:[0,\infty)\to [0,\infty)$ is a $\mathcal{C}^1$-diffeomorphism, with inverse denoted by $\theta^{-1}$. Define, for every $s\ge 0$, $\tilde{W}(s)=\int_{0}^{\theta^{-1}(s)}\sqrt{\theta(t)}dW(t)$. Note that for every $s\ge 0$, $\theta^{-1}(s)=\inf\left\{t\ge 0~;~\theta(t)\ge s\right\}$ is a bounded stopping time, associated with the filtration generated by the Wiener process $W$. Then, it is straightforward to check that $\bigl(\tilde{W}(s)\bigr)_{s\ge 0}$ is a standard Wiener process on $\St$.

We introduce the following system:
\begin{equation}\label{eq:ABP_time_changed}
\begin{cases}
dY_s=\Dr\bigl(V,B_s\bigr)(Y_s)\frac{1}{G_s(\xi_\St(Y_s))}ds+\sqrt{\frac{2}{G_s(\xi_\St(Y_s))}}\Sigma d\tilde{W}_s,\\
\overline{\nu}_{s}=\frac{1}{1+s}\bigl(\overline{\mu}_0+\int_{0}^{s}\delta_{Y_\sigma}d\sigma\bigr),\\
G_s=\No\bigl(\K(\overline{\nu}_s)\bigr),\\
B_s=-\log\bigl(\overline{G}_s\bigr).
\end{cases}
\end{equation}

Then the following identities are satisfied almost surely:
\begin{equation}\label{eq:change}
\begin{cases}
X_{t}=Y_{\theta(t)} \quad,\quad \overline{\mu}_t=\overline{\nu}_{\theta(t)} \quad,\quad F_t=G_{\theta(t)} \quad,\quad A_t=B_{\theta(t)}\quad,\quad \forall~t\ge 0\\
Y_s=X_{\theta^{-1}(s)} \quad,\quad \overline{\nu}_s=\overline{\mu}_{\theta^{-1}(s)} \quad,\quad G_s=F_{\theta^{-1}(s)} \quad,\quad B_s=A_{\theta^{-1}(s)}\quad,\quad \forall~s\ge 0.
\end{cases}
\end{equation}

The system~\eqref{eq:ABP_time_changed} may thus be considered as the time-changed version of the original ABP system~\eqref{eq:ABP_full}, with the new time variable $s=\theta(t)$, and the new unknowns $Y_s$, $\overline{\nu}_s$, $G_s$ and $B_s$, replacing $X_t$, $\overline{\mu}_t$, $F_t$ and $A_t$.

Observe that, in~\eqref{eq:ABP_time_changed}, the weight $F_t\bigl(\xi_\St(X_t)\bigr)=G_s\bigl(\xi_\St(Y_s)\bigr)$ does not appear anymore in the definition of the measure $\overline{\nu}_s$. Instead, the weight appears in the dynamics of the diffusion process~$\bigl(Y_s\bigr)_{s\ge 0}$. In terms of new variables, the ODE~\eqref{eq:ODE} has a simpler formulation:
\begin{equation}\label{eq:ODE_nu}
\frac{d\overline{\nu}_s}{ds}=\frac{1}{1+s}\bigl(\delta_{Y_s}-\overline{\nu}_s\bigr).
\end{equation}

\bigskip

We are interested in the convergence of $\overline{\mu}_t$ (or $\overline{\mu}_t(\varphi)$) when $t\to \infty$. Since $\overline{\mu}_t=\overline{\nu}_{\theta(t)}$, and $\theta(t)\underset{t\to \infty}\to \infty$ almost surely, the asymptotic behavior ($s\to \infty$) of $\overline{\nu}_s$ needs to be analyzed. In the remainder of this section, we work only with the system~\eqref{eq:ABP_time_changed}, and consider $s$ as the natural (but fictive in practice) time variable. Observe that proving Theorem~\ref{th:cv_as_phi} is equivalent to proving that
\[
\overline{\nu}_s(\varphi)\underset{s\to \infty}\to \mu_\star(\varphi)~,~\text{almost surely},
\]
which is done in Section~\ref{sec:ODE} using the ODE method.

\subsubsection{Consistency via the ODE method}\label{sec:ODE}

Thanks to the change of time variable $s=\theta(t)$ introduced above, the structure of the system~\eqref{eq:ABP_time_changed} is closer to the formulation of self-interacting diffusions (see~\cite{BLR02} for instance), depending on the normalized occupation measure, than for the initial system~\eqref{eq:ABP_time_changed}. However, in the specific situation considered in the present article, arguments need to be modified, in particular the coupling of the evolutions of the diffusion process and of the empirical distributions does not have the same structure (here it depends on the kernel $K$).

\bigskip

Thanks to the ODE~\eqref{eq:ODE_nu}, observe that there is an asymptotic time scale separation (in the limit $s\to \infty$) between slow variables $\overline{\nu}_s$, $G_s$ and $B_s$, and fast variables $Y_s$. It is reasonable to focus on the asymptotic behavior of the diffusion process when the other variables are frozen; when its unique invariant distribution (in general depending on the frozen variables) is introduced in place of the Dirac mass in~\eqref{eq:ODE_nu}, a limit ODE is obtained: the rationale behind the ODE method is that its asymptotic behavior provides information on the asymptotic behavior of the solution of~\eqref{eq:ODE_nu}.

The ODE method allows us to make rigorous the discussion above, and to identify the appropriate limit ODE. In this article, one of the main specific properties is that the invariant distribution of the fast equation with frozen variables is equal to $\mu_\star$, the target probability distribution, and thus does not depend on the frozen variables.

\begin{rem}
The asymptotic time scale separation (when $t\to \infty$) between slow variables $\overline{\mu}_t$, $F_t$ and $A_t$, and the fast variable $X_t$, already appears in the original system~\eqref{eq:ABP_simple}. The change of time variable $s=\theta(t)$ allows us to remove the random quantity $\theta(t)$, and to identify the correct limit equation for the application of the ODE method.
\end{rem}

\bigskip

Precisely, for every $G\in\mathcal{C}^\infty(\Ma_\md,\R)\cap\mathcal{C}^0(\Ma_\md,(0,\infty))$, let $\bigl(Y_s^{G}\bigr)_{s\ge 0}$ denote the diffusion process which is the unique solution of
\begin{equation}\label{eq:frozen}
dY_s^G=\frac{\Dr\bigl(V,B\bigr)(Y_s^G)}{G\bigl(\xi_\St(Y_s^G)\bigr)}ds+\sqrt{\frac{2}{G\bigl(\xi_\St(Y_s^G)\bigr)}}\Sigma d\tilde{W}_s,
\end{equation}
where $B=-\log(\overline{G})$.

\begin{propo}\label{propo:ergo_frozen}
For every $G\in\mathcal{C}^\infty(\Ma_\md,\R)\cap\mathcal{C}^0(\Ma_\md,(0,\infty))$, the unique invariant probability distribution for~\eqref{eq:frozen} is equal to $\mu_\star$.
\end{propo}

\begin{proof}
First note that $G=\frac{\overline{G}}{\no(\overline{G})}=\frac{\exp(-B)}{\no(\overline{G})}$ is equal to $\exp(-B)$ up to a multiplicative constant, and thus a probability distribution $\mu$ is invariant for~\eqref{eq:frozen} if and only if it is invariant for
\begin{equation}\label{eq:Y^B}
d\mathcal{Y}_s^B=\frac{\Dr\bigl(V,B\bigr)(\mathcal{Y}_s^B)}{e^{-B(\xi_\St(\mathcal{Y}_s^B))}}ds+\sqrt{\frac{2}{e^{-B(\xi_\St(\mathcal{Y}_s^B))}}}\Sigma d\tilde{W}_s.
\end{equation}
Let $\mathcal{L}_{\mathcal{Y}}^B$ denote the associated infinitesimal generator: then for every function $\varphi\in\mathcal{C}^{\infty}(\St,\R)$,
\begin{equation}\label{eq:generators}
\mathcal{L}_{\mathcal{Y}}^B\varphi(y)=\frac{1}{e^{-B(\xi_\St(y))}}\mathcal{L}_X^B\varphi(y),
\end{equation}
where $\mathcal{L}_X^B$ is the infinitesimal generator of the biased diffusion process $X^B$ defined by~\eqref{eq:BP}, with $A=B$.

Since the unique invariant probability distribution of~\eqref{eq:BP} with $A=B$ is $\mu_\star^B$, the unique invariant probability distribution of~\eqref{eq:frozen} is proportional to
\begin{align*}
e^{-B(\xi_\St(y))}\mu_\star^B(dy)&=e^{-B(\xi_\St(y))}\frac{\exp\bigl(-\En(V,B)(y)\bigr)}{Z^B}\nuSt(dy)\\
&=\frac{\exp\bigl(-\En(V)(y)\bigr)}{Z^B}\nuSt(dy)=\frac{Z^0}{Z^B}\mu_\star(dy),
\end{align*}
using~\eqref{eq:muA} (expression of expression of $\mu_\star^B$) and~\eqref{eq:En} (expression of $\En(V,B)$). Identifying the normalizing constants then concludes the proof of Proposition~\ref{propo:ergo_frozen}.
\end{proof}

Following the ODE method leads to study the following equation:
\begin{equation}\label{eq:ODE_gamma}
\frac{d\gamma_s}{ds}=\frac{1}{1+s}\bigl(\Pi(\gamma_s)-\gamma_s\bigr)=\frac{1}{1+s}\bigl(\mu_\star-\gamma_s\bigr).
\end{equation}
Indeed, thanks to Proposition~\ref{propo:ergo_frozen}, $\Pi(\gamma)=\mu_\star$ is the unique invariant distribution of~\eqref{eq:frozen}, where $G=\K(\gamma)$. This property justifies the consistency of the approach, {\it i.e.} the almost sure convergence of $\overline{\nu}_s$ to $\mu_\star$. Indeed, it is straightforward to check that, for any initial condition $\gamma_0\in\mathcal{P}(\St)$, one has
\[
\gamma_s=\frac{1}{1+s}\bigl(\overline{\gamma}_0+s\mu_\star)\underset{s\to \infty}\to \mu_\star.
\]
Moreover, a rigorous connexion between the asymptotic behaviors of $\overline{\nu}_s$ and of $\gamma_s$ may be stated for instance using the notion of asymptotic pseudo-trajectories (see~\cite{B99},~\cite{BLR02}); or by proving direct estimates on the $L^p$ norm of the random variable $\overline{\nu}_s(\varphi)-\mu_\star(\varphi)$.

In Section~\ref{sec:proof_decomp_error} below, instead, we prove directly estimates on the $L^p$ norm of the random variable $\overline{\mu}_t(\varphi)-\mu_\star(\varphi)$; indeed, thanks to Proposition~\ref{propo:ergo_frozen}, the situation is rather simple and the error is analyzed using straightforward computations, combined with a powerful auxiliary tool: the use of the solutions of associated Poisson equation.

\subsection{Analysis of the error and convergence}\label{sec:proof_decomp_error}

\subsubsection{The error in terms of the solutions of Poisson equations}

In order to prove that
\[
\overline{\mu}_t(\varphi)-\mu_\star(\varphi)
=\frac{\int_0^tF_\tau(\xi_\St(X_\tau))
\bigl[\varphi(X_\tau)-\mu_\star(\varphi)\bigr]d\tau}{1+\int_0^tF_\tau(\xi_\St(X_\tau))d\tau}
\]
converges to $0$ when $t\to\infty$, it is convenient to introduce a family of Poisson equations depending on the integrand on the numerator. Let $\Phi:(s,y)\in[1,\infty)\times\St\mapsto \Phi(s,y)\in\R$ be a $\mathcal{C}^{1,2}$ function, {\it i.e.} of class $\mathcal{C}^1$ with respect to the variable $s$ and of class $\mathcal{C}^2$ with respect to the variable $y$, with bounded associated derivatives. The application of It\^o's formula yields the equality
\begin{align*}
\Phi(t,X_t)-\Phi(0,X_0)&=\int_{0}^{t}\mathcal{L}_X^{A_\tau}(\tau,X_\tau)d\tau+\int_{0}^{t}\frac{\partial \Phi}{\partial \tau}(\tau,X_\tau)d\tau\\
&~+\int_0^{t}\sqrt{2}\langle \nabla \Phi(\tau,X_\tau),\Sigma dW(\tau)\rangle,
\end{align*}
where $\mathcal{L}_X^A$ is the infinitesimal generator of the biased diffusion process $X^A$,see~\eqref{eq:BP}.

Assume that the function $\Phi$ satisfies, for all $t\ge 0$, $x\in\St$
\begin{equation}\label{eq:Poisson_like}
\mathcal{L}_X^{A_t}\Phi(t,X_t)=F_t(\xi_\St(x))
\bigl[\varphi(x)-\mu_\star(\varphi)\bigr];
\end{equation}
then one obtains
\begin{align*}
\overline{\mu}_t(\varphi)-\mu_\star(\varphi)&
=\frac{\Phi(t,X_t)-\Phi(0,X_0)}{1+\int_0^tF_\tau(\xi_\St(X_\tau))d\tau}\\
&~-\frac{\int_0^{t}\sqrt{2}\langle \nabla \Phi(\tau,X_\tau),\Sigma dW(\tau)\rangle}{1+\int_0^tF_\tau(\xi_\St(X_\tau))d\tau}\\
&~-\frac{\int_{0}^{t}\frac{\partial \Phi}{\partial \tau}(\tau,X_\tau)d\tau}{1+\int_0^tF_\tau(\xi_\St(X_\tau))d\tau}.
\end{align*}
Recall that, $1+\int_{0}^{t}F_\tau(\xi_\St(X_\tau))d\tau\ge mt$, for all $t\ge 0$, almost surely, thanks to Theorem~\ref{th:well-posed}, with $m>0$. Convergence of $\overline{\mu}_t(\varphi)-\mu_\star(\varphi)$ to $0$, in a $L^p$ sense, then follows from appropriate estimates on the function $\Phi$ and its derivatives, which are stated below.

\subsubsection{Properties of solutions of Poisson equations}

This section is devoted to the statement of the properties concerning solutions of Poisson equations which are used in the analysis. We emphasize that the estimates are uniform with respect to $A\in\mathcal{A}$, which is defined by~\eqref{eq:fonctions_F}.

The equation~\eqref{eq:Poisson_like} can be written as
\begin{equation}\label{eq:PhiPsi}
\Phi(t,x)=\frac{1}{\no(\overline{F}_t)}\Psi(A_t,x)
\end{equation}
where, for any $A\in\mathcal{A}$ (see~\eqref{eq:fonctions_F}), $\Psi(A,\cdot)$ is solution of the Poisson equation
\begin{equation}\label{eq:Poisson}
\begin{cases}
\mathcal{L}_{\mathcal{X}}^{A}\Psi(A,\cdot)=e^{-A(\xi_\St(x))}\bigl[\varphi-\mu_\star(\varphi)\bigr],\\
\int \Psi(A,\cdot) d\mu_\star^A=0.
\end{cases}
\end{equation}

Introduce the set
\begin{equation}\label{eq:C}
\mathcal{C}=\mathcal{C}_{{\rm pol}}^\infty(\St,\R)=\left\{\varphi\in\mathcal{C}^\infty(\St,\R) ~;~ \forall~k\in\N, \exists~p_k\in\N, \sup_{x\in\St}\frac{|D^k\varphi(x)|}{1+|x|^{p_k}}<\infty\right\}
\end{equation}
of functions $\varphi:\St\to \R$, of class $\mathcal{C}^\infty$, with at most polynomial growth, and all derivatives with at most polynomial growth. Note that the average $\mu_\star(\varphi)$ is well-defined, since the probability distribution $\mu_\star$ admits finite moments of any order.

We first state the following well-posedness result.
\begin{propo}\label{propo:Poisson_1}
For every $A\in\mathcal{A}$ and every $\varphi\in\mathcal{C}$, there exists a unique solution $\Psi(A,\cdot)\in\mathcal{C}$ of the Poisson equation~\eqref{eq:Poisson}.
\end{propo}

We only provide a sketch of proof. Define the auxiliary function $\overline{\varphi}^A=e^{-A\circ\xi_\St}\bigl(\varphi-\mu_\star(\varphi)\bigr)$, and note that $\overline{\varphi}^A\in\mathcal{C}$. Thanks to~\eqref{eq:mu_muA}, the centering condition
\[
\int \overline{\varphi}^A d\mu_\star^A=\int \bigl[\varphi-\mu_\star(\varphi)\bigr]d\mu_\star^0=0.
\]
is satisfied. It is then well-known that the unique solution of the Poisson equation~\eqref{eq:Poisson} is given by
\begin{equation}\label{eq:Psi_c}
\Psi(A,x)=-\int_{0}^{\infty}\E_x\bigl[\overline{\varphi}^A(X_t^A)\bigr]dt,
\end{equation}
where $X^A$ is the biased process given by~\eqref{eq:BP}.

In Section~\ref{sec:convergence_details} below, bounds on $\Psi(A_t,\cdot)$, and its derivatives are required. The analysis is performed using the two following claims. On the one hand, thanks to Proposition~\ref{prop:unif}, almost surely $A_t\in\mathcal{A}$, where $\mathcal{A}$ is defined by~\eqref{eq:fonctions_F}. On the other hand, Proposition~\ref{propo:Poisson_2} states estimates which are uniform for $A\in\mathcal{A}$. The proof is postponed to Appendix~\ref{sec:appendix_finie}.

\begin{propo}\label{propo:Poisson_2}
Let $\varphi\in\mathcal{C}$, of class $\mathcal{C}^\infty$, with at most polynomial growth.

There exist $C\in(0,\infty)$ and $p\in\N^\star$, such that the following results hold true.
\begin{itemize}
\item[(i)] For every $A\in\mathcal{A}$ and every $x\in \St$
\begin{equation}\label{eq:propo_Poisson_2_Linfty}
\big|\Psi(A,x)\big|\le C(1+|x|^p).
\end{equation}
\item[(ii)] For every $A\in\mathcal{A}$ and every $x\in \St$
\begin{equation}\label{eq:propo_Poisson_2_Dy}
\big|\nabla_x\Psi(A,x)\big|\le C(1+|x|^p).
\end{equation}
\item[(iii)] The function $(t,x)\in [0,\infty)\times \St\mapsto \Psi(A_t,x)$ is of class $\mathcal{C}^{1,2}$, and for every $x\in\St$ and every $t\ge 0$, almost surely
\begin{equation}\label{eq:propo_Poisson_2_derivative}
\big|\frac{\partial \Psi(A_t,x)}{\partial t}\big|\le \frac{C(1+|x|^p)}{1+t},
\end{equation}
where $\bigl(A_t\bigr)_{t\ge 0}$ is the $\mathcal{A}$-valued process defined in~\eqref{eq:ABP_full}.
\end{itemize}
\end{propo}

\subsubsection{Proof of convergence}\label{sec:convergence_details}

An approximation procedure is required to deal with the low regularity properties of the normalization operator $\no$, see Assumption~\ref{ass:normalization}. For $k\in\N$, define
\[
\Phi^{(k)}(t,x)=\frac{1}{\no^{(k)}(\overline{F}_t)}\Psi(A_t,x).
\]
Then observe that that
\begin{align*}
\overline{\mu}_t(\varphi)-\mu_\star(\varphi)&=\frac{\int_0^tF_\tau(\xi_\St(X_\tau))
\bigl[\varphi(X_\tau)-\mu_\star(\varphi)\bigr]d\tau}{1+\int_0^tF_\tau(\xi_\St(X_\tau))d\tau}\\
&=\frac{\int_0^t \frac{1}{\no(\overline{F}_\tau)}\mathcal{L}_X^{A_\tau}\Psi(A_\tau,\cdot)(X_\tau)d\tau}{1+\int_0^tF_\tau(\xi_\St(X_\tau))d\tau}\\
&=\lim_{k\to\infty}\frac{\int_0^t \frac{1}{\no^{(k)}(\overline{F}_\tau)}\mathcal{L}_X^{A_\tau}\Psi^{(k)}(A_\tau,X_\tau)d\tau}{1+\int_0^tF_\tau(\xi_\St(X_\tau))d\tau}\\
&=\lim_{k\to\infty}\frac{\int_0^t \mathcal{L}_X^{A_\tau}\Phi^{(k)}(\tau,X_\tau)d\tau}{1+\int_0^tF_\tau(\xi_\St(X_\tau))d\tau}=:\lim_{k\to \infty}\epsilon_t^{(k)}(\varphi),
\end{align*}
where the limit $k\to \infty$ is understood in an almost sure sense, thanks to Assumption~\ref{ass:normalization}, and the fact that $\overline{F}_t\in\mathcal{F}$ for all $t\ge 0$, almost surely, thanks to Theorem~\ref{th:well-posed}.

It\^o's formula can be used, since $(t,x)\mapsto \Phi^{(k)}(t,x)$ is of class $\mathcal{C}^{1,2}$ thanks to Proposition~\ref{propo:Poisson_2}. Then
\begin{align*}
\epsilon_t^{(k)}(\varphi)&
=\frac{\Phi^{(k)}(t,X_t)-\Phi^{(k)}(0,X_0)}{1+\int_0^tF_\tau(\xi_\St(X_\tau))d\tau}\\
&~-\frac{\int_0^{t}\sqrt{2}\langle \nabla \Phi^{(k)}(\tau,X_\tau),\Sigma dW(\tau)\rangle}{1+\int_0^tF_\tau(\xi_\St(X_\tau))d\tau}\\
&~-\frac{\int_{0}^{t}\frac{\partial \Phi^{(k)}}{\partial \tau}(\tau,X_\tau)d\tau}{1+\int_0^tF_\tau(\xi_\St(X_\tau))d\tau}\\
&=:\epsilon_t^{(k),1}(\varphi)+\epsilon_t^{(k),2}(\varphi)+\epsilon_t^{(k),3}(\varphi).
\end{align*}

We now prove the following result.
\begin{lemma}\label{lem:epsilon}
Let $\varphi\in\mathcal{C}$. There exists $C(\varphi)\in(0,\infty)$ such that for every $t\ge 0$ and $k\in\N$
\[
\E|\epsilon_t^{(k)}(\varphi)|^2\le \frac{C(\varphi)}{t}.
\]
\end{lemma}

Observe that Lemma~\ref{lem:epsilon} is valid for test functions $\varphi$ in the set $\mathcal{C}$ defined by~\eqref{eq:C}. To prove Theorem~\ref{th:cv_as_phi}, {\it i.e.} an almost sure convergence result, we will only use it with test functions which are bounded and have bounded derivatives. However, to prove Proposition~\ref{propo:variance}, we will need this Lemma for test functions with polynomial growth.

\begin{proof}[Proof of Lemma~\ref{lem:epsilon}]
The proof of that result consists in using the estimates of Proposition~\ref{propo:Poisson_2}.
\begin{itemize}
\item Thanks to item (i) from Proposition~\ref{propo:Poisson_2}, for every $t\ge 0$, $\Phi^{(k)}(t,\cdot)$ has at most polynomial growth, and moments of the process $X$ are bounded, see Theorem~\ref{th:well-posed}. More precisely, the parameters $C$ and $p$ in the right-hand side of the inequality~\eqref{eq:propo_Poisson_2_Linfty} do not depend on $A=A_\tau$. Moreover, thanks to Assumption~\ref{ass:normalization}, for every $k\in\N$ and $\tau\ge 0$, one has $\no^{(k)}(\overline{F}_\tau)\ge m>0$ almost surely.

It is then straightforward to conclude that
\[
\E|\epsilon_t^{(k),1}(\varphi)|^2\le \frac{C(\varphi)}{t^2}.
\]
\item To have an estimate of the stochastic integral, we use It\^o's formula, and we obtain
\begin{align*}
\E|\epsilon_t^{(k),2}(\varphi)|^2&\le C\frac{1+\int_{0}^{t}|\Sigma^\star\nabla_x \Phi^{(k)}(\tau,X_\tau)|^2d\tau}{1+t^2}\\
&\le \frac{C(\varphi)}{t},
\end{align*}
thanks to~\eqref{eq:propo_Poisson_2_Dy}, and arguments similar to the term above.

\item Finally, using~\eqref{eq:propo_Poisson_2_derivative}, and similar arguments, one obtains
\[
\E|\epsilon_t^{(k),3}(\varphi)|^2\le \frac{C(\varphi)\Bigl(\int_{0}^{t}\frac{1}{1+\tau}d\tau\Bigr)^2}{t^2}\le \frac{C(\varphi)\bigl(1+\log(t)\bigr)^2}{t^2}.
\]
\end{itemize}
Gathering estimates then concludes the proof of Lemma~\ref{lem:epsilon}.
\end{proof}

We are now in position to deduce Theorem~\ref{th:cv_as_phi} from Lemma~\ref{lem:epsilon}. First, note that it is straightforward to obtain
\[
\E|\overline{\mu}_t(\varphi)-\mu_\star(\varphi)|^2\le \frac{C(\varphi)}{t}.
\]
Indeed, the right-hand side in the estimate of Lemma~\ref{lem:epsilon} does not depend on $k$, and taking the limit $k\to\infty$ in the right-hand side gives the result, thanks to Assumption~\ref{ass:normalization} which ensures the required uniform convergence properties for the application of the bounded convergence theorem.

This estimate ensures the convergence in mean-square sense, and in probability, of $\overline{\mu}_t(\varphi)$ to $\mu_\star(\varphi)$. To go further, and obtain the almost sure convergence, we use the following arguments. First, note that it is sufficient to prove that $\overline{\mu}_{\exp(t)}$ converges almost surely to $\mu_\star(\varphi)$ when $t\to \infty$. Using the estimate
\[
\E|\overline{\mu}_{\exp(t)}(\varphi)-\mu_\star(\varphi)|^2\le C(\varphi)e^{-t},
\]
and Borel-Cantelli's Lemma, then almost surely, for every $\delta\in\mathbb{Q}\cap(0,\infty)$,
\[
\overline{\mu}_{\exp(n\delta)}(\varphi)\underset{n\to \infty}\to \mu_\star(\varphi).
\]
Finally, thanks to the differential equation~\eqref{eq:ODE} and boundedness of the function $\varphi$, the mapping $t\mapsto \overline{\mu}_{\exp(t)}(\varphi)$ is Lipschitz continuous, with constant smaller than $C(\varphi)$, almost surely, for some $C(\varphi)\in(0,\infty)$ depending only on $\varphi$, and on the parameters appearing in the definition of the set $\mathcal{F}$, see~\eqref{eq:fonctions_F}. It is then straightforward to obtain the almost sure convergence
\[
\overline{\mu}_{\exp(t)}(\varphi)\underset{t\to \infty}\to \mu_\star(\varphi).
\]

This concludes the proof of Theorem~\ref{th:cv_as_phi}.

\section{Analysis of the mean-square error. Proof of Proposition~\ref{propo:variance}}\label{sec:proof_variance}

In this section, we give a proof of Proposition~\ref{propo:variance}, concerning the asymptotic behavior of the mean-square error, which is decomposed as
\begin{equation}
\E\big|\overline{\mu}_t(\varphi)-\mu_\star(\varphi)\big|^2=
\bigl(\E\overline{\mu}_t(\varphi)-\mu_\star(\varphi)\bigr)^2
+{\rm Var}\bigl(\overline{\mu}_t(\varphi)\bigr),
\end{equation}
when $t\to\infty$, for functions $\varphi\in\mathcal{C}$, of class $\mathcal{C}^\infty$, with at most polynomial growth.

In Section~\ref{sec:proof_variance_bias}, we prove that the bias satisfies
\begin{equation}\label{eq:bias}
\E\overline{\mu}_t(\varphi)-\mu_\star(\varphi)={\rm O}(\frac{1+\log(t)}{t}).
\end{equation}
In Section~\ref{sec:proof_variance_variance}, we then prove that
\begin{equation}\label{eq:var}
t\E\big|\overline{\mu}_t(\varphi)-\mu_\star(\varphi)\big|^2\underset{t\to \infty}\to \mathbf{V}_\infty(\varphi)\in[0,\infty).
\end{equation}
In particular, thanks to~\eqref{eq:bias}, we may interpret the limit as the asymptotic variance, since
\[
\mathbf{V}_\infty(\varphi)=\underset{t\to \infty}\lim t{\rm Var}\bigl(\overline{\mu}_t(\varphi)\bigr).
\]
The asymptotic variance is expressed in terms of the solution of a Poisson equation~\eqref{eq:Poisson}, with $A=A_\infty=\underset{t\to\infty}\lim A_t$ (defined in Corollary~\ref{cor:cv_A}).

In Section~\ref{sec:proof_variance_A}, we check that $\mathbf{V}_\infty(\varphi)=\mathbf{V}_\infty(\varphi,A_\infty)$, where $\mathbf{V}_\infty(\varphi,A)\in[0,\infty)$ is the asymptotic variance associated with the non-adaptively biasing method, using~\eqref{eq:BP} and~\eqref{eq:mu_t^A}, with $A=A_\infty$.

\subsection{Asymptotic behavior of the bias}\label{sec:proof_variance_bias}

Let us prove~\eqref{eq:bias}. Using the same arguments as in Section~\ref{sec:convergence_details}, note that
\begin{align*}
\E\overline{\mu}_t(\varphi)-\mu_\star(\varphi)
&=\E\bigl[\underset{k\to\infty}\lim \epsilon_t^{(k)}(\varphi)\bigr]\\
&=\underset{k\to\infty}\lim\E\bigl[\epsilon_t^{(k)}(\varphi)\bigr]\\
&=\underset{k\to\infty}\lim\E\bigl[\epsilon_t^{(k),1}(\varphi)+\epsilon_t^{(k),3}(\varphi)\bigr].
\end{align*}
Indeed, using Assumption~\ref{ass:normalization} and the property that $\overline{F}_t\in\mathcal{F}$, for all $t\ge 0$, almost surely, allows us to use the bounded convergence theorem. Moreover, note that $\E\bigl[\epsilon_t^{(k),2}(\varphi)\bigr]=0$, for all $k\in\N$ and $t\ge 0$. Then~\eqref{eq:bias} is an immediate corollary of Lemma~\ref{lem:epsilon}.

\subsection{Asymptotic behavior of the mean-square error}\label{sec:proof_variance_variance}

Let us now prove~\eqref{eq:var}. Like in Sections~\ref{sec:convergence_details} and~\ref{sec:proof_variance_bias} above, we use the decomposition of $\overline{\mu}_t(\varphi)-\mu_\star(\varphi)$ in terms of the auxiliary function $\Phi^{(k)}$; we prove error bounds which are uniform with respect to $k\in\N$, and pass to the limit $k\to \infty$, thanks to Assumption~\ref{ass:normalization} and Theorem~\ref{th:well-posed}.

It is straightforward to check that, uniformly in $k\in\N$,
\[
t\E\big|\overline{\mu}_t(\varphi)-\mu_\star(\varphi)\big|^2-t\E\Big|\frac{\int_{0}^{t}\sqrt{2}\langle \nabla_x\Phi^{(k)}(\tau,X_\tau),\Sigma dW(\tau)\rangle}{1+\int_{0}^{t}F_\tau\bigl(\xi_\St(X_\tau^A)\bigr)d\tau}\Big|^2={\rm O}\Bigl(\frac{1}{t}\Bigr),
\]
thanks to Lemma~\ref{lem:epsilon}, {\it i.e.} only the stochastic integral contributes to the asymptotic variance. We can directly pass to the limit $k\to\infty$ at this stage.

Let $R(t)=\frac{1+\int_{0}^{t}F_\tau\bigl(\xi_\St(X_\tau^A)\bigr)d\tau}{t}$. Then almost surely, thanks to Corollary~\ref{cor:cv_A} (uniform convergence of $F_\tau$ to $F_\infty$ when $\tau\to \infty$), Corollary~\ref{cor:conv_rho}, and the version of Cesaro's Lemma stated in the proof of Corollary~\ref{cor:conv_rho},
\[
R(t)\underset{t\to\infty}\to \int_{\St}F_\infty\circ\xi_\St d\mu_\star^{A_\infty}=\frac{\mu_\star^{A_\infty}\bigl(e^{-\star A_\infty\circ\xi_\St}\bigr)}{\no(\overline{F}_\infty)}=\frac{1}{\mu_\star^{0}\bigl(e^{A_\infty\circ\xi_\St}\bigr)\no(\overline{F}_\infty)},
\]
using~\eqref{eq:mu_muA}. We then obtain
\begin{align*}
t\E\Big|\frac{\int_{0}^{t}\frac{\sqrt{2}}{\no(\overline{F}_\tau)}\langle \nabla_x\Psi(A_\tau,X_\tau),\Sigma dW(\tau)\rangle}{1+\int_{0}^{t}F_\tau\bigl(\xi_\St(X_\tau^A)\bigr)d\tau}\Big|^2&=\E\frac{\big|\int_{0}^{t}\frac{\sqrt{2}}{\no(\overline{F}_\tau)}\langle \nabla_x\Psi(A_\tau,X_\tau),\Sigma dW(\tau)\rangle\big|^2}{t|R(t)|^2}\\
&=\E\frac{\big|\int_{0}^{t}\frac{\sqrt{2}}{\no(\overline{F}_\tau)}\langle \nabla_x\Psi(A_\tau,X_\tau),\Sigma dW(\tau)\rangle\big|^2}{t}+{\rm o}(1)\\
&=\frac{2\int_{0}^{t}\E\frac{|\Sigma^\star\nabla_x\Psi(A_\tau,X_\tau)|^2}{\no(\overline{F}_\tau)^2}d\tau}{t}+{\rm o}(1)\\
&=\frac{\frac{2}{{\no(\overline{F}_\infty)^2}}\int_{0}^{t}\E|\Sigma^\star\nabla_x\Psi(A_\infty,X_\tau)|^2d\tau}{t}+{\rm o}(1)\\
&=2\E\overline{\mu}_t\bigl(|\Sigma^\star\nabla_x\Psi(A_\infty,\cdot)|^2\bigr)+{\rm o}(1)\\
&\underset{t\to \infty}\to 2\mu_\star\bigl(|\Sigma^\star\nabla_x\Psi(A_\infty,\cdot)|^2\bigr),
\end{align*}
thanks to Lemma~\ref{lem:epsilon}, applied to the function $|\Sigma^\star\nabla_x\Psi(A_\infty,\cdot)|^2\in\mathcal{C}$, thanks to Proposition~\ref{propo:Poisson_1}. Observe that the limit does not depend on the normalization operator $\no$. We thus obtain~\eqref{eq:var}, more precisely,
\[
t\E\big|\overline{\mu}_t(\varphi)-\mu_\star(\varphi)\big|^2\underset{t\to \infty}\to \mathbf{V}_\infty(\varphi)=2\mu_\star\bigl(|\Sigma^\star\nabla_x\Psi(A_\infty,\cdot)|^2\bigr).
\]

\subsection{Comparison with the non-adaptive biasing method}\label{sec:proof_variance_A}

We now check that the expression obtained above for the asymptotic variance in the adaptive method, coincides with the expression of the asymptotic variance in the non-adaptive method, when choosing $A=A_\infty$.

Let $A\in\mathcal{A}$ (see~\eqref{eq:fonctions_F}), and $\varphi\in\mathcal{C}$. Using~\eqref{eq:mu_t^A}, and the solution of the Poisson equation~\eqref{eq:Poisson},
\begin{align*}
\overline{\mu}_t^A(\varphi)-\mu_\star(\varphi)
&=\frac{1+\int_{0}^{t}\exp\bigl(-A\circ\xi_\St(X_\tau^A)\bigr)
\bigl[\varphi(X_\tau^A)-\mu_\star(\varphi)\bigr]d\tau}
{1+\int_{0}^{t}\exp\bigl(-A\circ\xi_\St(X_\tau)\bigr)d\tau}\\
&=\frac{1+\int_{0}^{t}\mathcal{L}_{\mathcal{X}}^{A}\Psi(A,X_\tau^A)d\tau}
{1+\int_{0}^{t}\exp\bigl(-A\circ\xi_\St(X_\tau)\bigr)d\tau}\\
&=\frac{\Psi(A,X_t^A)-\Psi(A,X_0^A)
-\int_{0}^{t}\sqrt{2}\langle \nabla_x\Psi(A,X_\tau^A),\Sigma dW(\tau)\rangle}
{1+\int_{0}^{t}\exp\bigl(-A\circ\xi_\St(X_\tau)\bigr)d\tau}
\end{align*}

Since $A\in\mathcal{A}$, $\exp\bigl(-A\circ\xi_\St\bigr)\ge m>0$, for some $m>0$. Then
\[
t\E\big|\overline{\mu}_t^A(\varphi)-\mu_\star(\varphi)\big|^2-t\E\Big|\frac{\int_{0}^{t}\sqrt{2}\langle \nabla_x \Psi(A,X_\tau^A),\Sigma dW(\tau)\rangle}{1+\int_{0}^{t}\exp\bigl(-A\circ\xi_\St(X_\tau^A)\bigr)d\tau}\Big|^2={\rm O}\Bigl(\frac{1}{t}\Bigr).
\]
Let $R^A(t)=\frac{1+\int_{0}^{t}\exp\bigl(-A\circ\xi_\St(X_\tau^A)\bigr)d\tau}{t}$. Then almost surely $R^A(t)\underset{t\to\infty}\to \mu_\star^A(e^{-A\circ\xi_\St})=\frac{1}{\mu_\star(e^{A\circ\xi_\St})}$.

With the same arguments as in Section~\ref{sec:proof_variance_variance} above,
\begin{align*}
t\E\Big|\frac{\int_{0}^{t}\sqrt{2}\langle \nabla_x \Psi(A,X_\tau^A),\Sigma dW(\tau)\rangle}{1+\int_{0}^{t}\exp\bigl(-A\circ\xi_\St(X_\tau^A)\bigr)d\tau}\Big|^2&=\E\frac{\big|\int_{0}^{t}\sqrt{2}\langle \nabla_x \Psi(A,X_\tau^A),\Sigma dW(\tau)\rangle\big|^2}{t|R^A(t)|^2}\\
&=\E\frac{\big|\int_{0}^{t}\sqrt{2}\langle \nabla_x \Psi(A,X_\tau^A),\Sigma dW(\tau)\rangle\big|^2}{t}+{\rm o}(1)\\
&=\frac{2\int_{0}^{t}\E|\Sigma^\star\nabla_x\Psi(A,X_\tau^A)|^2d\tau}{t}+{\rm o}(1)\\
&=2\E\overline{\mu}_t^A\bigl(|\Sigma^\star\nabla_x\Psi(A,\cdot)|^2\bigr)+{\rm o}(1)\\
&\underset{t\to \infty}\to 2\mu_\star\bigl(|\Sigma^\star\nabla_x\Psi(A,\cdot)|^2\bigr).
\end{align*}

We thus conclude that
\[
t\E\big|\overline{\mu}_t^A(\varphi)-\mu_\star(\varphi)\big|^2\underset{t\to \infty}\to \mathbf{V}_\infty(\varphi,A)=2\mu_\star\bigl(|\Sigma^\star\nabla_x\Psi(A,\cdot)|^2\bigr).
\]
The asymptotic variance in the ABP method is thus equal to the asymptotic variance in the non-adaptive method with $A=A_\infty=\underset{t\to\infty}\lim A_t$, as expected:
\[
\mathbf{V}_\infty(\varphi)=2\mu_\star\bigl(|\Sigma^\star\nabla_x\Psi(A_\infty,\cdot)|^2\bigr)=\mathbf{V}_\infty(\varphi,A_\infty).
\]

\section{The SPDE case}\label{sec:SPDE}

The aim of this section is to generalize the approach developed in other sections of this article, to deal with metastable stochastic processes in infinite dimension. More precisely, we describe an ABP method designed to compute averages $\mu_\star(\varphi)$, where $\mu_\star$ is a probability distribution defined on an infinite dimensional (Hilbert) space; the corresponding diffusion processes are given by some parabolic semilinear Stochastic Partial Differential Equations (SPDEs).

In Section~\ref{sec:model_SPDE}, we describe the model, and we explain how it fits in the framework of Section~\ref{sec:framework}. In particular, this description justifies the introduction of the abstract objects in Section~\ref{sec:framework}.

Some arguments and some statements need to be substantially modified, compared with the finite dimensional situation, see Sections~\ref{sec:results_SPDE} and~\ref{sec:modifs_SPDE}, and details in Section~\ref{sec:appendix_infinie}.

\subsection{The model}\label{sec:model_SPDE}

In this section, we consider infinite dimensional diffusion processes, which are solutions of parabolic, semilinear, SPDEs, driven by space-time white noise, in space dimension $1$, which may be written in the following form:
\begin{equation}\label{eq:SPDE_tx}
du^0(t,x)=\frac{\partial^2 u^0(t,x)}{\partial x^2}dt-\nabla \mathcal{V}(u^0(t,x))dt+\sqrt{2}dW(t,x),
\end{equation}
for $x\in (0,1)$, with (for instance) homogeneous Dirichlet boundary conditions.

The function $\mathcal{V}:\R\to \R$ is a smooth mapping. With the choice $\mathcal{V}(x)=\frac{x^4}{4}-\frac{x^2}{2}$, one obtains the Allen-Cahn equation, which is the paradigmatic example of metastable SPDE considered in the literature: see for instance~\cite{BG13},~\cite{BdH_book},~\cite{BGGR15},~\cite{FJl82}. In particular, (non-adaptive) importance sampling techniques are considered for this problem in~\cite{VW12}.

In this article, $\mathcal{V}$ is assumed to have bounded derivatives, in order to simplify the presentation and the functional setting. Metastable states are solutions of the stationary PDE
\[
\frac{\partial^2 u(x)}{\partial x^2}-\nabla \mathcal{V}(u(x))=0.
\]
Assume that the potential energy function $\mathcal{V}$ is even; then $x\mapsto u^0(x)=0$ is one solution. Moreover, if there exists another solution $x\mapsto u^+(x)$, $x\mapsto u^-(x)=-u^+(x)$ is also a solution. These solutions are critical points of the energy functional
\[
u\mapsto \int_{0}^{1}\bigl[\frac{1}{2}\big|\frac{\partial u(x)}{\partial x}\big|^2+\mathcal{V}\bigl(u(x)\bigr)\bigr]dx,
\]
and may be local minima, saddle points, etc...

\bigskip

It is convenient and standard to write~\eqref{eq:SPDE_tx} as a Stochastic Evolution Equation in the Hilbert space $H=L^2(0,1)$, see for instance the monograph~\cite{DPZ}:
\begin{equation}\label{eq:SPDE_dyn}
du_t^0=Lu_t^0dt-D V(u_t^0)dt+\sqrt{2}dW(t),
\end{equation}
where $D$ denotes the Fr\'echet derivative, and
\begin{itemize}
\item $\bigl(e_n\bigr)_{n\in\N^*}$ is the complete orthonormal system of $H$ given by $e_n(x)=\sqrt{2}\sin(n\pi x)$;
\item the unbounded linear operator $L:H\to H$ satisfies $Lu=-\sum_{n\in \N^*}\pi^2n^2\langle u,e_n\rangle e_n$;
\item $V(u)=\int_{0}^{1}\langle \nabla \mathcal{V}(\theta u),u\rangle d\theta$ for all $u\in H$;
\item $\bigl(W(t)\bigr)_{t\ge 0}$ is a cylindrical Wiener process on $H$, {\it i.e.} $W(t)=\sum_{n\in \N^{*}}\beta_n(t)e_n$ for a family $(\beta_n)_{n\in \N^*}$ of independent, one-dimensional, standard Wiener processes.
\end{itemize}
Equation~\eqref{eq:SPDE_dyn} admits a unique mild solution (see~\cite{DPZ}) with values in $H$, defined for $t\ge 0$, {\it i.e.} $u^0$ is the unique process solution satisfying the equation
$$u_t^0=e^{tL}u_0-\int_{0}^{t}e^{(t-s)L}D V(u_s^0)ds+\sqrt{2}\int_{0}^{t}e^{(t-s)L}dW(s),$$
where $\bigl(e^{tL}\bigr)_{t\in[0,+\infty)}$ is the semi-group on $H$ generated by $L$: $e^{tL}u=\sum_{n\in \N^*}e^{-\pi^2n^2t}\langle u,e_n\rangle e_n$.

In the context of this section, Assumption~\ref{ass:V_OL} is satisfied when the following condition is satisfied:
\begin{equation}\label{eq:ass_V_SPDE}
\underset{x\in\R}\sup~|\mathcal{V}''(x)|<\pi^2.
\end{equation}
In other words, the Lipschitz constant of the non-linear coefficient $u\in H\mapsto DV(u)\in H$ is bounded from above by all the eigenvalues of $-L$. Ergodicity of the SPDE~\eqref{eq:SPDE_dyn} is then obtained by the following arguments, see for instance~\cite[Section~6.3]{DPZ_2} for additional details. Let $u_0,v_0\in H$ denote two initial conditions, and define $\bigl(u_t^0\bigr)_{t\ge 0}$ and $\bigl(v_t^0\bigr)_{t\ge 0}$ the solutions of~\eqref{eq:SPDE_dyn} driven by the same Wiener process $\bigl(W(t)\bigr)_{t\ge 0}$. Then $r_t=u_t^0-v_t^0$ satisfies
\[
\frac{dr_t}{dt}=Lr_t+DV(v_t^0)-DV(u_t^0),
\]
and thus
\begin{align*}
\frac{1}{2}\frac{\|dr_t\|_H^2}{dt}&=-\langle (-L)r_t,r_t\rangle+\langle DV(v_t^0)-DV(u_t^0),r_t\rangle\\
&\le -\pi^2\|r_t\|_H^2+\underset{x\in\R}\sup~|V''(x)| \|r_t\|_H^2\le -\gamma \|r_t\|_H^2,
\end{align*}
with $\gamma>0$, thanks to the condition~\eqref{eq:ass_V_SPDE}. By Gronwall's Lemma, $\E\|v_t^0-u_t^0\|_H^2\le e^{-\gamma t}\|v_0-u_0\|_H^2$, which yields uniqueness of an invariant distribution for the SPDE~\eqref{eq:SPDE_dyn}, as well as exponential convergence to equilibrium. There are several general ways to prove the existence of an invariant distribution, see~\cite[Chapter~6]{DPZ_2}. Alternatively, in the situation treated in the present work, the SPDE has a gradient structure and an explicit formula for the invariant distribution is available, see~\cite[Theorem~8.6.3]{DPZ_2}:
\[
\mu_\star(du)=\frac{\exp\bigl(-V(u))\bigr)}{Z}\lambda(du)
\]
for some $Z\in(0,\infty)$, where the reference measure $\lambda$ is a Gaussian measure on $H$, defined below in Section~\ref{sec:invar_SPDE}.

\bigskip

We are now in position to explain how the SPDE dynamics fits into the general framework presented in this article, in Section~\ref{sec:framework}.

\subsubsection{Setting}

In the SPDE example, one has the following elements, see Section~\ref{sec:dynamics}.
\begin{itemize}
\item {\bf State space:} $\St=H$ (infinite dimensional, separable, Hilbert space).
\item {\bf Reaction coordinate:} assume that $\Ma_\md=\tore$ (with $\md=1$), $E_1=\R$ (with $d=1$). Then for instance $\xi(u)=\xi_\St(u)=\frac{1}{2}+\frac{1}{\pi}\arctan\bigl(\frac{1}{2}\int_{0}^{1}u(x)dx\bigr)$.
\item {\bf Drift coefficient:} $\Dr(V,A)=Lu-D\bigl(V-A\circ\xi_\St\bigr)$. {\bf Diffusion operator:} $\Sigma$ is the identity on $\St$.
\end{itemize}

Since $L$ is an unbounded linear operator on $H$, note that the drift is only defined on a domain $D(L)\subset H$. This is one of the technical issues which are specific to the infinite dimensional framework.

\begin{rem}
Note that, in general, there does not exist a function $\mathcal{V}_A:\R\to \R$ such that the function $V-A\circ\xi:L^2(0,1)\to \R$ satisfies $D(V-A\circ\xi)(u)(x)=\nabla \mathcal{V}_A(u(x))$: the bias is a nonlocal function of $u$, since it depends on the spatial average $\int_0^1u(y)dy$, instead of $u(x)$ only.
\end{rem}

The biased version~\eqref{eq:BP} of the SPDE~\eqref{eq:SPDE_tx} is written as
\begin{equation}\label{eq:BP_SPDE}
du^A(t)=Lu^A(t)dt-D\bigl(V-A\circ\xi_\St\bigr)(u^A(t))dt+\sqrt{2}dW(t),
\end{equation}
with mild formulation
\[
u^A(t)=e^{tL}u_0-\int_{0}^{t}e^{(t-s)L}D V(u^A(s))ds+\int_{0}^{t}e^{(t-s)L}D\bigl(A\circ \xi_\St\bigr)(u^A(s))ds+\sqrt{2}\int_{0}^{t}e^{(t-s)L}dW(s).
\]

\subsubsection{Invariant probability distribution}\label{sec:invar_SPDE}

We now construct the Total Energy function, and the reference measure $\lambda$ on $H$.

First, the definition of the mapping $V\mapsto \mathcal{E}(V)$ is straightforward: $\mathcal{E}(V)=V$. The reference measure $\lambda$ on $\St$ is defined as follows: it is the centered Gaussian probability distribution on $H$ with covariance operator $L^{-1}$. This measure can be constructed as follows: let $\bigl(g_n\bigr)_{n\in\N^\star}$ be a sequence of independent standard real-valued Gaussian random variables (centered and with variance $1$); then $\lambda$ is the probability distribution of the $H$-valued random variable $\sum_{n\in\N^\star}\frac{1}{n\pi}g_n e_n$.

\begin{rem}
One may check that $\lambda$ defined as above is the distribution of the Brownian Bridge on $(0,1)$. This interpretation is specific to the choice of $L$ and plays no role in this article. On the contrary, the construction above, based on eigenvalues and eigenvectors of $L$, is general.
\end{rem}

It is straightforward to check that $\lambda$ is the unique invariant distribution of~\eqref{eq:SPDE_dyn} when $V=0$. More generally, for any function $A:\tore\to \R$ of class $\mathcal{C}^\infty$, the probability distribution $\mu_\star^A$ on $H$, defined by
\[
\mu_\star^A(du)=\frac{\exp\bigl(-(V(u)-A(\xi(u)))\bigr)}{Z^A}\lambda(du)
\]
where $Z^A\in(0,\infty)$ thanks to~\eqref{eq:ass_V_SPDE}, is the unique invariant distribution of the biased SPDE~\eqref{eq:BP_SPDE}, see for instance~\cite{DPZ_2}.

\subsubsection{Free Energy function}

It remains to discuss how the Free Energy function $A_\star$ is defined: this is done using Definition~\ref{def:FE}, like in the finite dimensional case. It is natural to choose $\pi$ to be the Lebesgue measure on $\tore$. Indeed, the image measure of the Gaussian distribution $\lambda$ by the linear mapping $u\mapsto \int_0^1 u(y)dy$ is a non-degenerate Gaussian distribution on $\R$; thus the image of $\lambda$ by $\xi_\St$ is equivalent to the Lebesgue measure on $\tore$. Then $\pi_\star^0$ the image of $\mu_\star^0$ by $\xi$ is equivalent to $\pi$.

\subsection{ABP dynamics and convergence results}\label{sec:results_SPDE}

Let us first describe the dynamics of the ABP method, which generalizes~\eqref{eq:ABP_full} in the case where the diffusion process is governed by a SPDE:
\begin{equation}\label{eq:ABP_SPDE}
\begin{cases}
du(t)=Lu(t)dt-D\bigl(V-A_t\circ\xi_\St\bigr)(u(t))dt+\sqrt{2}dW(t)\\
\overline{\mu}_t=\frac{\overline{\mu}_0+\int_{0}^{t}\exp\bigl(-A_r\circ\xi_\St(u(r))\bigr)\delta_{u(r)}dr}{1+\int_{0}^{t}\exp\bigl(-A_r\circ\xi_\St(u(r))\bigr)dr}\\
\exp\bigl(-A_t(z)\bigr)=\int_{\tore^d}K\bigl(z,\xi_\St(u)\bigr)\overline{\mu}_t(du),~\forall z\in\tore^m,
\end{cases}
\end{equation}
For simplicity, we have chosen the normalization operator $\No$, with $\no(F)=\int_\tore F(z)dz$. The kernel function $K:\tore\times\tore\to (0,\infty)$ satisfies Assumption~\ref{ass:kernel}.

As explained in Section~\ref{sec:model_SPDE} above, it is convenient to consider the mild formulation for the SPDE dynamics: the first equation in~\eqref{eq:ABP_SPDE} is understood as
\[
u(t)=e^{tL}u_0-\int_{0}^{t}e^{(t-s)L}D\bigl( V-A_s\circ\xi\bigr)(u(s))ds+\sqrt{2}\int_{0}^{t}e^{(t-s)L}dW(s).
\]
Using Lemma~\ref{lemma:apriori} and standard techniques, the following generalization of Theorem~\ref{th:well-posed} is obtained.
\begin{theo}\label{th:well-posed_SPDE}
Consider the framework of Section~\ref{sec:model_SPDE} (in particular assume that~\eqref{eq:ass_V_SPDE} is satisfied), and assume that the kernel function $K$ satisfies Assumption~\ref{ass:kernel}.
\begin{itemize}
\item There exists a unique continuous process $t\mapsto (u(t),\overline{\mu}_t,A_t)$, taking values in $H\times \mathcal{P}(H) \times \mathcal{C}^{0}(\Ma_\md,(0,\infty))$, which is solution of the ABP system~\eqref{eq:ABP_SPDE}.
\item For all $k\ge 1$, $\sup_{t\ge 0}\E\|u(t)\|_H^k<+\infty$.
\item There exist $m\in(0,\infty)$ and $\bigl(M^{(r)}\bigr)_{r\in\left\{0,1,\cdots\right\}}\in(0,\infty)$ such that, almost surely, $A_t\in\mathcal{A}$, for all $t\in\R^+$, where
\begin{equation}
\begin{cases}
\mathcal{F}=\left\{F\in \mathcal{C}^\infty(\Ma_\md); \min F\geq m, \max|\partial^{k}F|\leq M^{(k)}, k\ge 0\right\},\\
\mathcal{A}=\left\{A=-\log(\overline{F});~F\in\mathcal{F}\right\}.
\end{cases}
\end{equation}
\end{itemize}
\end{theo}

We are able to prove generalizations of Theorem~\ref{th:cv_as_phi} and of Corollary~\ref{cor:cv_A}

\begin{theo}\label{th:cv_SPDE}
\begin{itemize}
\item Let $\varphi\in\mathcal{C}^\infty(H,\R)$ be a bounded function, with bounded derivatives of any order. Then, almost surely,
\[
\overline{\mu}_t(\varphi)\underset{t\to \infty}\to \mu_\star(\varphi).
\]
\item Define, for all $z\in \Ma_\md$,
\[
A_\infty(z)=-\log(\mu_\star\bigl(K(z,\cdot)\bigr)).
\]
Then, almost surely, for every $\ell\in\left\{0,1,\ldots\right\}$, uniformly on $\Ma_\md$,
\[
\partial^\ell A_t\underset{t\to \infty}\to \partial^\ell A_\infty.
\]
\end{itemize}
\end{theo}

The efficiency results from Section~\ref{sec:efficiency} may also be generalized: more precisely, the convergence result~\eqref{eq:cor_conv_rho}, and Proposition~\ref{propo:variance} remain valid.

\subsection{Some modifications for SPDEs}\label{sec:modifs_SPDE}

Compared with the finite dimensional situation, observe that we only state the almost sure convergence of averages $\overline{\mu}_t(\varphi)$ in Theorem~\ref{th:cv_SPDE}. The arguments used in the proof of Corollary~\ref{cor:cv_as} do not easily generalize to the infinite dimensional setting, to prove almost sure convergence of $\overline{\mu}_t$.

There are also modifications when dealing with the solutions of the Poisson equations. To simplify the discussion, assume first that $A=0$. Then the Poisson equation~\eqref{eq:Poisson} is written in the infinite dimensional setting, as
\begin{equation}\label{eq:Poisson_infinie}
\langle Lu-DV(u),D\Psi(u)\rangle+\frac{1}{2}\sum_{n=1}^{\infty}D^2\Psi(u).(e_n,e_n)=\varphi(u)-\mu_\star(\varphi),\quad \forall u\in H,
\end{equation}
where the unknown is the function $\Psi:H\to\R$. In the Poisson equation above, the first order derivative $D\Psi(u)\in H$ is interpreted as an element of $H$ thanks to Riesz's Theorem.

For an arbitrary function $\Psi$ of class $\mathcal{C}^2$ on $H$, it is not true in general that the left-hand side is well-defined, for all $u\in H$, or even when $u=u(t)$ is the diffusion process evaluated at a time $t\ge 0$. Indeed, $L$ is an unbounded operator, so $Lu$ is not an element of $H$ in general. Moreover, the series may not be convergent.

In fact, the Poisson equation may be solved and all the terms make sense thanks to regularity properties, which may be written in the form~\eqref{eq:Poisson_SPDE_estimate}, where auxiliary norms are introduced: for any $\alpha\in(0,1)$, and any $h\in H$, let
\[
\|h\|_{\alpha}^2=\sum_{n=1}^{\infty}\lambda_n^{2\alpha}|\langle h,e_n\rangle|^2\in[0,\infty]~,~\|h\|_{-\alpha}^2=\sum_{n=1}^{\infty}\lambda_n^{-2\alpha}|\langle h,e_n\rangle|^2<\infty.
\]
We refer to~\cite[Chapters~4,5]{Cerrai:01}, for general results concerning the smoothing properties of the transition semi-group in infinite dimension, and to~\cite[Proposition 6.1]{BK16}, and \cite[Chapter 4, Section 8]{Kopec:14}, for their application to the analysis of Poisson equations. Rigorous properties are often stated for spatial Galerkin approximations, with bounds not depending on the dimension. We do not provide such details here, and directly write the results in the Hilbert space $H$.

Using arguments from the references mentioned above, and taking care of the dependence with respect to the function $A$ to obtain uniform bounds on the set $\mathcal{A}$, generalizations of Propositions~\ref{propo:Poisson_1} and~\ref{propo:Poisson_2} are obtained.

\begin{propo}\label{propo:Poisson_SPDE}
Let $A\in\mathcal{A}$, and $\varphi:H\to \R$, of class $\mathcal{C}^\infty$, bounded and with bounded derivatives of any order.

There exists a unique solution $\Psi(A,\cdot)$ of the Poisson equation~\eqref{eq:Poisson},
\[
\langle Lu-D\bigl(V-A\circ\xi\bigr)(u),D\Psi(u)\rangle+\frac{1}{2}\sum_{n=1}^{\infty}D^2\Psi(u).(e_n,e_n)=e^{-A(\xi(u))}[\varphi(u)-\mu_\star(\varphi)],\quad \forall u\in H,
\]
with the condition $\int_H\Psi d\mu_\star=0$.

This solution is given by
\[
\Psi(A,u)=-\int_{0}^{\infty}\E_u\bigl[\overline{\varphi}^A\bigl(u(t)\bigr)]dt,
\]
for all $u\in H$, where $\overline{\varphi}^A(u)=e^{-A(\xi(u))}[\varphi(u)-\mu_\star(\varphi)]$.

Moreover, the following properties are satisfied.
\begin{itemize}
\item There exists $C(\varphi)\in(0,\infty)$ such that, for all $A\in\mathcal{A}$ and $u\in H$,
\[
|\Psi(A,u)|\le C(\varphi)(1+\|u\|_H^2).
\]
\item For every $\alpha\in(0,\frac{1}{2})$, there exists $C(\alpha,\varphi)\in(0,\infty)$, such that, for all $A\in\mathcal{A}$ and $u\in H$
\begin{equation}\label{eq:Poisson_SPDE_estimate}
\begin{cases}
|\langle D_u\Psi(A,u),h\rangle |\le C(\alpha,\varphi)(1+\|u\|_H^2)\|h\|_{-2\alpha},~\forall~h\in H,\\
|D_u^2\Psi(A,u).(h,k)|\le C(\alpha,\varphi)(1+\|u\|_H^2)\|h\|_{-\alpha}\|k\|_{-\alpha},~\forall~h,k\in H.
\end{cases}
\end{equation}
\item For every $\alpha\in(0,\frac{1}{4})$ and every $n\in\N$, there exists $C(\alpha,n,\varphi)\in(0,\infty)$, such that $\E\|u(t)\|_\alpha^n\le C(\alpha,n,\varphi)\bigl(1+\frac{\|u(0)\|_H}{t^\alpha}\bigr)^n$.
\item  The function $(t,u)\in [0,\infty)\times H\mapsto \Psi(A_t,u)$ is of class $\mathcal{C}^{1,2}$, and for every $u\in H$ and every $t\ge 0$, almost surely
\[
\big|\frac{\partial \Psi(A_t,u)}{\partial t}\big|\le \frac{C(1+\|u\|^2)}{1+t},
\]
where $\bigl(A_t\bigr)_{t\ge 0}$ is the $\mathcal{A}$-valued process defined in~\eqref{eq:ABP_SPDE}.
\end{itemize}

\end{propo}

A sketch of proof is postponed to Appendix~\ref{sec:appendix_infinie}. More precisely, we focus there on the estimates~\eqref{eq:Poisson_SPDE_estimate}, with $A=0$ (to simplify the presentation), for $\alpha>0$, since they are the main novelty in the infinite dimensional framework.

The estimates~\eqref{eq:Poisson_SPDE_estimate} justify that all the terms in the left-hand side of~\eqref{eq:Poisson_infinie} make sense. First, if one assumes that $\|u\|_{\epsilon}<\infty$ for some $\epsilon>0$, choosing $2\alpha=1-\epsilon<1$, gives $|\langle Lu,D\Psi(u)\rangle|<\infty$. Second, $\sum_{n=1}^{\infty}\|e_n\|_{-\alpha}^{2}=\sum_{n=1}^{\infty}\lambda_n^{-2\alpha}<\infty$ for $\alpha>\frac14$.

Adapting the strategy of proof of Theorem~\ref{th:cv_as_phi}, developed in Section~\ref{sec:proof_decomp_error}, and using Proposition~\ref{propo:Poisson_SPDE} to control the terms, it is then straightforward to prove that
\[
\E\big|\overline{\mu}_t(\varphi)-\mu_\star(\varphi)\big|^2\le \frac{C(\varphi)}{t}\underset{t\to\infty}\to 0.
\]
The proof of the almost sure convergence result is obtained using the boundedness of $\varphi$, and the same argument as in the finite dimensional case. This concludes the proof of the first part of Theorem~\ref{th:cv_SPDE}. The second part of Theorem~\ref{th:cv_SPDE}, concerning the almost sure convergence of $A_t$, is proved exactly as Corollary~\ref{cor:cv_A}.

\bigskip

Thanks to the general framework developed in Section~\ref{sec:framework}, the ABP method is also applicable in the infinite dimensional setting, for metastable Stochastic PDEs.

\section*{Acknowledgments}

The authors would like to thank the referees for their valuable comments which helped us improve the presentation of this manuscript. They would also like to thank Tony~Leli\`evre and Gabriel~Stoltz for stimulating discussions on this article and comments on a preliminary version of this manuscript. The work is partially supported by SNF grants $200020 149871$ and $200021 163072$.

\begin{appendix}
\section{Results concerning Poisson equations}

\subsection{The finite dimensional case}\label{sec:appendix_finie}

The aim of this section is to give a proof of Proposition~\ref{propo:Poisson_2}, stated in Section~\ref{sec:proof_decomp_error}. More precisely, the key point is to prove that the estimates are uniform with respect to $A\in\mathcal{A}$, where $\mathcal{A}$ is defined by~\eqref{eq:fonctions_F}.

To simplify the presentation, the analysis is restricted to functions $\varphi$ which are bounded and have bounded derivatives of any order. The general case $\varphi\in\mathcal{C}$, having polynomial growth (see~\eqref{eq:C}), may be treated using weight functions, using Assumptions~\ref{ass:V_OL} and~\ref{ass:V_L}. Thanks to Property~\ref{prop:unif}, the weight functions may be chosen independently of $A\in\mathcal{A}$, hence estimates are uniform for $A\in\mathcal{A}$.

Let $\varphi$ be fixed, and recall the notation $\overline{\varphi}^A=e^{-A\circ\xi_\St}\bigl(\varphi-\mu_\star(\varphi)\bigr)$. Moreover, $\Psi(A,\cdot)$ is given by~\eqref{eq:Psi_c}.

\subsubsection{Auxiliary result: exponential convergence to equilibrium}

Let $W:\St\to \R_+$ be defined as follows. If the dynamics is given by the Brownian dynamics (Section~\ref{sec:OL}), set
\[
W(x)=\|x\|.
\]
If the dynamics is given by the Langevin dynamics (Section~\ref{sec:L}), set
\[
W(q,p)=\sqrt{ V(q) + Q(q,p)}
\]
with $Q(q,p) = \frac{\gamma^2}{4}\|q\|^2 + \frac{\gamma}{2} \langle q, p \rangle + \frac{1}{2} \|p\|^2$. Here we assume without loss of generality that $V \geq 0.$

The case of the extended dynamics (Section~\ref{sec:E}) is treated like the Brownian dynamics case.

Let $\bigl(P_t^A\bigr)_{t\ge 0}$ denote the semi-group associated with the SDE~\eqref{eq:BP}, with $A\in\mathcal{A}$. Recall that $\mu_\star^A$ defined by~\eqref{eq:muA} is the unique invariant distribution for this process.

The aim of this section is to prove that convergence to equilibrium is exponentially fast, uniformly with respect to $A\in\mathcal{A}$.
\begin{propo}\label{propo:AuxiliaryPoisson}
There exists $\vartheta>0$ and $C\in(0,\infty)$ such that for every measurable function $\varphi : \St\to \R$ with $\|\varphi\|_W := \sup \frac{|\varphi(x)|}{1+ W(x)} < \infty, A\in\mathcal{A}$ and $t\ge 0$,
\begin{equation}
\big|P_t^A \varphi(x) - \mu_\star^A(\varphi) \big|\le Ce^{-\vartheta t}(1+  W(x)) \|\varphi\|_W,
\end{equation}
In addition, for functions $\varphi:\St\to\R$ such that $\|\varphi\|_W<\infty$, almost surely
\begin{equation}\label{eq:ergoA}
\frac{1}{t}\int_0^t \varphi(X_r^A)dr\underset{t\to \infty}\to\int \varphi d\mu_\star^A.
\end{equation}
\end{propo}

We first state a version of Harris Theorem. Let $\mathcal{E}$ be a measurable set, and $W:\mathcal{E} \mapsto \R_+$ be a measurable map. Following~\cite{Hairer-Mattingly}, 
define for every $0 \leq \beta  \leq 1$ and $f  : \mathcal{E} \mapsto \R$ measurable (possibly unbounded), $$\|f\|_{\beta,W} =   \sup_{x,y} \frac{|f(x)-f(y)|}{2 + \beta (W(x)+W(y))}.$$

\begin{lemma}\label{Harris}
Let $P$ and $Q$ be two Markov kernels over $\mathcal{E}.$ Assume there exist $0 \le \rho < 1, \kappa \ge 0, R \ge \frac{2 \kappa}{1-\rho}, \epsilon > 0, \delta \geq 0$ and $\psi$ a probability distribution over $\mathcal{E}$ such that
\begin{enumerate}
\item $P W \leq \rho W + \kappa, \, Q W \leq \rho W + \kappa;$
\item For all $x \in W_R := \{y \in \mathcal{E}: \: W(y) \leq R\}$ $P(x,dy) \geq \epsilon \psi(dy)$ and $|\delta_x P - \delta_x Q|_{1,W} \leq \delta.$
\end{enumerate}
Then, there exist $0 \leq \beta \leq 1$ and $0 \leq \theta < 1$ such that
$$\|P \|_{\beta, W} \leq  \theta, \mbox{ and } \|Q \|_{\beta, W} \leq  \theta + \delta.$$ Here $\|P\|_{\beta,W}$ stands for $\sup\{ \|P f \|_{\beta, W} : \: \|f\|_{\beta, W} \leq 1\}.$
\end{lemma}

The first statement (concerning $P$) rephrases Theorem 3.1 in \cite{Hairer-Mattingly}. The proof of the second one (concerning $Q$) easily follows from the proof of the first one.

\begin{lemma}\label{lem:langevin}
For each $A \in \mathcal{A}$ there exists $T_A> 0, 0 \leq \beta_A \leq 1, 0 \leq \theta_A < 1$ and $\delta_A > 0$ such that for all $B \in \mathcal{A}$,
$$\|A-B\|_{C^1} \leq \delta \Rightarrow \|P_{T}^B\|_{\beta, W} \leq \theta,$$
where $\|A-B\|_{C^1}=\max |A-B|+\max|\partial A-\partial B|$.
\end{lemma}

\begin{proof}
Let $V_A(q) = V(q) - A(\xi(q))$. Replacing $V$ by $V + c$ for some $c > 0$ we can assume without loss of generality  that $V_A \geq 0$ for all $A \in\mathcal{A}.$
Thanks to Property~\ref{prop:unif}, there exist positive constants $\alpha=\alpha_{\mathcal{A}}, \kappa=\kappa_{\mathcal{A}}$ such that for all $A \in\mathcal{A}$
$$\mathcal{L}^A W_A^2 \leq - 2 \alpha W_A^2 + 2 \alpha \kappa^2,$$
where $\mathcal{L}^A$ is the infinitesimal generator of the SDE~\eqref{eq:BP} and $W_A$ is defined like $W$ with $V_A$ in place of $V$. Then, by standard It\^o calculus,
\begin{equation}
\label{eq:lyap0}
P_t^A W_A^2 \leq e^{-2\alpha t} W_A^2 + \kappa^2.
\end{equation}
Replacing $\kappa^2$ by $\kappa^2 + 2 (|m|\wedge|M|)$ and using the fact that  $\|W^2-W^2_A\|  \leq \|A\|_{\infty} \leq |m|\wedge|M|,$  we then
 obtain
\begin{equation}
\label{eq:lyap1}
P_t^A W^2 \leq e^{-2\alpha t} W^2 + \kappa^2.
\end{equation}
Also, by H\"older inequality,
\begin{equation}
\label{eq:lyap2}
P_t^A W \leq  \sqrt{e^{-2\alpha t} W^2 + \kappa^2} \leq e^{- \alpha t} W + \kappa.
\end{equation}

By classical ellipticity (Brownian) or hypoellipticity (Langevin) results (see e.g~\cite{Ichihara}), for any given $A\in\mathcal{A}$:
\begin{enumerate}
 \item[(a)] For all $t > 0$ there exists a smooth function $(x,y) \mapsto p_t^A(x,y)$ such that
 $$P_t^A(x,dy) = p_t^A(x,y)dy$$
 \item[(b)] $(P_t^A)_{t\ge 0}$ is a strong Feller semi-group.
\end{enumerate}

Given $A\in\mathcal{A}, x_0 \in \St$ and $t_0 > 0$ one can then find $y_0 \in S$ such that
 \begin{equation}
 \label{ppos}
 p_{t_0}^{A}(x_0,y_0) > 0
 \end{equation}
The strong Feller property combined with the existence of an invariant probability having full support  (here $\mu_\star^A$) makes $(P_t^A)$ positively recurrent (see e.g~ \cite{Ichihara}, Section 5). In particular, the almost sure convergence property~\eqref{eq:ergoA} is satisfied, and for all $x \in\St$ and every neighborhood $U$ of $x$, there exists $\tau > 0$ such that
 \begin{equation}
 \label{access}
 P_{\tau}^{A}(x,U) > 0.
 \end{equation}
Using~\eqref{ppos} and~\eqref{access}, it is then proved that, for every compact set $K \subset \St$, there exist $T > 0$, a bounded open set $V$ and $\epsilon > 0$, such that
$$P_T(x,dy) \geq \epsilon \mathds{1}_V(y) dy,~\forall x\in K.$$
Applying Lemma~\ref{Harris}, with $P = P_T^A$ and $K = W_R$,
$$\|P_T^A\|_{\beta, W} < \theta$$ for some $0\le \beta\le 1$ and $0 \leq \theta < 1.$

We now claim that
$|\delta_x P_T^A - \delta_x P_T^B|_{1,W} \to 0$ uniformly in $x \in W_R$ when $\|A-B\|_{C^1} \to 0.$

Let $f$ be such that $\|f\|_{1,W} \leq 1.$  Replacing $f$ by $f - f(x_0)$, without loss of generality it is assumed that $|f(x)|    \leq C + W(x)$, with $C = 2 + W(x_0).$
By Girsanov Theorem,
$$P^A_T f(x) - P^B_T f(x) = \E_x (f(X_T^A) - f(X_T^B)) = \E_x (f(X_T^A)(1- M_T))$$ where $(M_t)$ is the martingale defined as $$M_t = \exp{( - \int_0^t \langle u_s, d\tilde{W}_s \rangle - \frac{1}{2}  \int_0^t \|u_s\|^2 ds)}$$ and $u_s = \nabla (A \circ \xi - B\circ \xi)(X_s^B)$ (Brownian case), $u_s = (2\gamma)^{-1/2}(\nabla (A \circ \xi_\St - B\circ \xi_\St)(X_s^B))$ (Langevin case). Thus, for all $x \in W_R$
$$|P^A_T f(x) - P^B_T f(x)| \leq \E_x (( C + W(X_T^A))|1-M_T|) \leq (C +   \sqrt{P_T^A W^2(x)}) \sqrt{ \E_x(M_T^2 -1)}.$$
$$\leq (C + R + \kappa) \sqrt{ \E_x(M_T^2 -1)},$$
thanks to H\"older inequality and to \eqref{eq:lyap1}. Observe that $M_t^2 = \tilde{M}_t e^{\int_0^t \|u_s\|^2 ds}$ where $(\tilde{M}_t)$ is a nonnegative martingale.
Therefore  $1 \leq \E(M_T^2) \leq e^{c_T \|A-B\|^2_{C^1}}$ with $c_T = T \max(1,(2\gamma)^{-1})\|D\xi\|^2.$

This concludes the proof of the claim. The result then follows from applying Lemma~\ref{Harris}.
\end{proof}

We are now in position to conclude.
\begin{proof}[Proof of Proposition~\ref{propo:AuxiliaryPoisson}]
By Ascoli theorem, $\mathcal{A}$ is relatively compact for the $C^1$ topology. Thus, thanks to Lemma~\ref{lem:langevin},
 there exist a finite covering of  $\mathcal{A}$ by open sets (for the $C^1$ topology) $O_1, \ldots, O_N$ (i.e~  $\mathcal{A} \subset \cup_{i = 1}^N O_i$),
and parameters $0 \leq \beta_i \leq 1, t_i > 0,$ and $0 \leq \theta_i < 1,$ such that for all $A \in O_i$
$$ \|P_{t_i}^A\|_{\beta_i,W} \leq \theta_i.$$
Let $\theta=\underset{i=1,\ldots,N}\max \theta_i<1$. Note that
for all $\beta > 0$
$$\|\varphi\|_{1,W} \leq \|\varphi\|_{\beta,W} \leq \frac{1}{\beta} \|\varphi\|_{1,W} \leq \frac{1}{\beta} \|\varphi\|_{W},$$
while, by (\ref{eq:lyap2}),  for all $r \geq 0$
$$\|P_r^A \varphi\|_{W} \leq \|\varphi\|_{W} (1 + \kappa).$$ 
Thus, for $A \in O_i, t = k t_i + r, k \in \N$ and $0 \leq r < t_i,$
$$\|P_t^A \varphi\|_{1,W} \leq \|P_t^A \varphi\|_{\beta_i,W} \leq \theta^k \|P_r^A \varphi\|_{\beta_i,W} \leq \frac{\theta^k}{\beta_i} \|\varphi\|_{W} (1 + \kappa).$$
That is $$\|P_t^A \varphi\|_{1,W} \leq e^{-\vartheta t} C \|\varphi\|_W$$ with $\vartheta = \min_i \frac{-\log(\theta)}{t_i}$ and $C = \max_i \frac{1+\kappa}{\beta_i \theta}.$
Equivalently, for all $x,y$,
$$|P_t^A \varphi(x) - P_t^A \varphi(y)| \leq  e^{-\vartheta t} C \|\varphi\|_W (2 + W(x) + W(y)).$$ Hence, integrating in $y$,
$$|P_t^A \varphi(x) - \mu_A^*(\varphi)| \leq  e^{-\vartheta t} C \|\varphi\|_W (2 + W(x) + \mu_A^*W) \leq  e^{-\vartheta t} C \|\varphi\|_W (2 + W(x) + \kappa).$$
This concludes the proof.
\end{proof}

\subsubsection{Proof of Proposition~\ref{propo:Poisson_2}}

In the proof below, the values of $C\in(0,\infty)$ and $p\in\N^\star$ may change from line to line. Note that if $\varphi$ is bounded, then $\|\varphi\|_W\le \|\varphi\|_{\infty}=\sup \varphi(x)$.

The properties of $V$ given by Assumptions~\ref{ass:V_OL} and~\ref{ass:V_L} play a key role in the estimate. Recall that Property~\ref{prop:unif} then allows to get estimates which are uniform with respect to $A\in\mathcal{A}$. As already explained, the technical computations are not reported here.
\begin{itemize}
\item[(i)] Thanks to Property~\ref{prop:unif}, and to Proposition~\ref{propo:AuxiliaryPoisson}, applied with $\varphi=\overline{\varphi}^A$, there exist $\vartheta\in(0,\infty)$, $C(\varphi)\in (0,\infty)$ and $p\in\N^\star$, such that for every $A\in\mathcal{A}$, then for all $x\in\St$ and $t\ge 0$, one has
\begin{equation}\label{eq:expo_conv}
\big|\E_x\bigl[\overline{\varphi}^{A}(X_t^A)\bigr]\big|\le C\|\varphi\|_\infty e^{-\vartheta t}\bigl(1+W(x)\bigr).
\end{equation}
Integrating from $t=0$ to $t=\infty$, using~\eqref{eq:Psi_c} and the polynomial growth assumption on $V$, gives~\eqref{eq:propo_Poisson_2_Linfty}.

\item[(ii)] 
The inequality~\eqref{eq:expo_conv} may be rewritten as follows: for all $A\in\mathcal{A}$, $x\in\St$ and $t\ge 0$,
\begin{equation}\label{eq:expo_conv_2}
\|P_t^A\overline{\varphi}^A\|_{W}\le C\|\varphi\|_\infty e^{-\vartheta t},
\end{equation}
where we recall that $\bigl(P_t^A\bigr)_{t\ge 0}$ is the transition semi-group associated with $X^A$. The elliptic and hypoelliptic need to be treated separately.

Consider first the elliptic case (Brownian dynamics). The idea is to adapt the arguments in~\cite[Chapter~2, Section~6]{Kopec:14}, and to check that all estimates are uniform with respect to $A\in\mathcal{A}$. First, by direct estimates of the derivatives (using in particular the semi-convexity property of $V$), when $t\in[0,1]$,
\[
|\nabla_x P_t^A\overline{\varphi}^A(x)|\le C(1+W(x))\|\nabla_x \overline{\varphi}^A\|_\infty\le C(1+W(x))(\|\varphi\|_\infty+\|\nabla_x\varphi\|_\infty).
\]
Second, for $t\ge 1$, let $\phi_t^A=P_{t-1}^A\overline{\varphi}^A$. Using the semi-group property, and the Bismut-Elworthy-Li formula, with constants which do not depend on $A\in\mathcal{A}$,
\[
|\nabla_x P_t^A\overline{\varphi}^A(x)|=|\nabla_x P_1\phi_t^A(x)|\le C\|\phi_t^A\|_W(1+W(x)).
\]
Using~\eqref{eq:expo_conv_2} to have an estimate of $\|\phi_t^A\|_W$, then integrating separately from $t=0$ to $t=1$ and from $t=1$ to $t=\infty$ gives~\eqref{eq:propo_Poisson_2_Dy}.

Consider now the hypoelliptic case (Langevin dynamics). The idea is to adapt the arguments in~\cite[Chapter~3, Section~6]{Kopec:14}, and to check that all estimates are uniform with respect to $A\in\mathcal{A}$. Again~\eqref{eq:expo_conv_2} is a fundamental ingredient. Estimates in Sobolev norms of $P_t^A\overline{\varphi}^A$ and of its derivatives are obtained. Then pointwise estimates are obtained using a Sobolev imbedding theorem. All the estimates are uniform with respect to $A\in\mathcal{A}$. The long and technical calculations are omitted.

\item[(iii)] Since $F_t\in\mathcal{F}$ for all $t\ge 0$, almost surely, thanks to Theorem~\ref{th:well-posed}, then $\underset{z\in\Ma_\md}\min\overline{F}_t(z)\ge m$ for all $t\ge 0$, almost surely, for some $m\in(0,\infty)$.

Moreover, $\overline{F}_t(z)=\overline{\mu}_t\bigl(K(z,\cdot)\bigr)$; thanks to Assumption~\ref{ass:kernel} and to the ODE~\eqref{eq:ODE}, for every $k\in\left\{0,1,\ldots\right\}$, there exists $C^{(k)}\in(0,\infty)$ such that
\begin{equation}\label{eq:borne_ddr}
\underset{z\in \Ma_\md}\sup\big|\frac{d\bigl(\partial^kA_t(z)\bigr)}{dt}\big|\le \frac{C^{(k)}}{1+t}.
\end{equation}

For every $t>0$, every $\epsilon\in(-t,1)$, note that
\[
\mathcal{L}_{\mathcal{Y}}^{A_{t+\epsilon}}\Psi(A_{t+\epsilon},\cdot)-\mathcal{L}_{\mathcal{Y}}^{A_t}\Psi(A_t,\cdot)=0,
\]
thanks to~\eqref{eq:generators}. Passing to the limit $\epsilon\to 0$ yields
\begin{align*}
\mathcal{L}_{\mathcal{Y}}^{A_t}\frac{\partial \Psi(A_t,\cdot)}{\partial t}&=-\Bigl(\frac{\partial \mathcal{L}_{\mathcal{Y}}^{A_t}}{\partial t}\Bigr)\Psi(A_t,\cdot)\\
&=\Bigl(\frac{d A_t\circ\xi_\St}{dt}\mathcal{L}_{Y}^{A_t}+\frac{d}{dt}\Bigl(e^{A_t\circ\xi_\St}\langle \Dr(V,A_t),\nabla\cdot\rangle\Bigr)\Bigr)\bigl(\Psi(A_t,\cdot)\bigr).
\end{align*}
Considering each example for the definition of the drift function $\Dr(V,B_s)$, it is straightforward to check that $-\Bigl(\frac{d\mathcal{L}_{\mathcal{Y}}^{A_t}}{ds}\Bigr)\Psi(A_t,\cdot)\in\mathcal{C}$ is a function of class $\mathcal{C}^\infty$ with polynomial growth; and, moreover, that for every $k\in\left\{0,1,\ldots\right\}$, there exist $p_k\ge 0$ and $C^{(k)}\in(0,\infty)$ such that
\[
\underset{y\in \St}\sup\big|D^k\Bigl(\frac{\partial \mathcal{L}_{\mathcal{Y}}^{A_t}}{\partial t}\Bigr)\Psi(A_t,\cdot)(x)\big|\le \frac{C^{(k)}(1+|x|^{p_k})}{1+t},
\]
thanks to the inequality~\eqref{eq:borne_ddr}, and the estimate~\eqref{eq:propo_Poisson_2_Dy} on the gradient $\nabla_x\Psi(A_t,x)$.

Thanks to Proposition~\ref{propo:Poisson_1}, one then concludes the proof of~\eqref{eq:propo_Poisson_2_derivative}.
\end{itemize}

\subsection{The infinite dimensional case}\label{sec:appendix_infinie}

The aim of this section is to provide a proof of the estimates~\eqref{eq:Poisson_SPDE_estimate} which are specific to the infinite dimensional framework. As explained above, we only focus on the case $A=0$, and to simplify notation, $\overline{\varphi}=\overline{\varphi}^0$ and $\Psi=\Psi(0,\cdot)$.

Introduce the semi-group $\bigl(P_t\bigr)_{t\ge 0}$, such that for all $t\ge 0$
\[
P_t\overline{\varphi}(u)=\E_u[\overline{\varphi}^0(u(t))].
\]

\subsubsection{First-order derivative}
We claim that
\begin{equation}\label{eq:proof_Poisson_SPDE}
\begin{aligned}
|\langle D(P_t\overline{\varphi}(u),h\rangle|&\le e^{-\gamma t}\sup_{v\in H}\|D\overline{\varphi}(v)\| \|h\|,~\forall~t\ge 0\\
|\langle D(P_t\overline{\varphi}(u),h\rangle|&\le \frac{C_\alpha \sup_{v\in H}\|D\overline{\varphi}(v)\|}{t^{2\alpha}}\|h\|_{-2\alpha},~\forall~t\in(0,1].
\end{aligned}
\end{equation}

Then using the semi-group property $P_t=P_1P_{t-1}$, for all $t\ge 1$,
\begin{align*}
|\langle D(P_t\overline{\varphi}(u),h\rangle|&\le C_\alpha \sup_{v\in H}\|D(P_{t-1}\overline{\varphi})(v)\|\|h\|_{-2\alpha}\\
&\le C_\alpha e^{-\gamma(t-1)}\sup_{v\in H}\|D\overline{\varphi}(v)\|\|h\|_{-2\alpha}.
\end{align*}

By integrating, this yields the first estimate in~\eqref{eq:Poisson_SPDE_estimate}.

It remains to prove the claim~\eqref{eq:proof_Poisson_SPDE}. Note that
\[
\langle DP_t\overline{\varphi}(u),h\rangle=\E[\langle D\overline{\varphi}^0(u(t)),\eta^h(t,u)\rangle],
\]
where $\eta^h(0,u)=h$ and
\[
d\eta^h(t,u)=L\eta^h(t,u)dt-\mathcal{V}''(u(t))\eta^h(t,u)dt.
\]
The proof of the first inequality of~\eqref{eq:proof_Poisson_SPDE} is straightforward:
\begin{align*}
\frac{1}{2}\frac{d\|\eta^h(t,u)\|^2}{dt}&=\langle L\eta^h(t,u),\eta^h(t,u)\rangle-\langle \mathcal{V}''(u(t))\eta^h(t,u),\eta^h(t,u)\rangle\\
&\le -(\lambda_1-\sup_{x\in\mathbb{R}}|\mathcal{V}''(x)|)\|\eta^h(t,u)\|^2\\
&\le e^{-\gamma t}\|h\|^2,
\end{align*}
with $\gamma=\lambda_1-\sup_{x\in\mathbb{R}}|\mathcal{V}''(x)|>0$ thanks to~\eqref{eq:ass_V_SPDE}.

The second inequality of~\eqref{eq:proof_Poisson_SPDE} is obtained using the mild formulation of the equation for $\eta^h(t,u)$ and regularization properties of the semi-group $\bigl(e^{tL}\bigr)_{t\ge 0}$:
\begin{align*}
\|\eta^h(t,u)\|_H&\le \|e^{tL}h\|_H+\int_{0}^{t}\|e^{(t-s)L}\bigl(\mathcal{V}''(u(s))\eta^h(s,u)\bigr)\|_Hds\\
&\le C(\alpha)t^{-2\alpha}\|(-L)^{-2\alpha}h\|_H+C\int_{0}^{t}\|\eta^h(s,u)\|_Hds,
\end{align*}
and $\|(-L)^{-2\alpha}h\|_H=\|h\|_{-2\alpha}$. Thanks to Gronwall Lemma, there exists $C(\alpha)\in(0,\infty)$ such that for all $t\in(0,1]$,
\[
\|\eta^h(t,u)\|_H\le C(\alpha)t^{-2\alpha}\|h\|_{-2\alpha},
\]
which yields the required estimate.

\subsubsection{Second-order derivative}

We claim that, for some $\tilde{\gamma}\in(0,\gamma)$,
\begin{equation}\label{eq:proof_Poisson_SPDE_2}
\begin{aligned}
|D^2(P_t\overline{\varphi}(u).(h,k)|&\le C\bigl(\sup_{v\in H}\|D\overline{\varphi}(v)\|+\sup_{v\in H}\|D^2\overline{\varphi}(v)\|\bigr)e^{-\tilde{\gamma} t} \|h\|\|k\|,~\forall~t\ge 0\\
|D^2(P_t\overline{\varphi}(u).(h,k)|&\le \frac{C_\alpha\bigl(\sup_{v\in H}\|D\overline{\varphi}(v)\|+\sup_{v\in H}\|D^2\overline{\varphi}(v)\|\bigr)}{t^{2\alpha}}\|h\|_{-\alpha}\|k\|_{-\alpha},~\forall~t\in(0,1].
\end{aligned}
\end{equation}
The proof uses the following identity:
\[
D^2(P_t\overline{\varphi}(u).(h,k)=\E\bigl[\langle D\overline{\varphi}(u(t)),\zeta^{h,k}(t,u)\rangle\bigr]+\E\bigl[D^2\overline{\varphi}(u(t)).(\eta^h(t,u),\eta^k(t,u))\bigr],
\]
where $\zeta^{h,k}(0,u)=0$ and
\[
d\zeta^{h,k}(t,u)=L\zeta^{h,k}(t,u)dt-\mathcal{V}''(u(t))\zeta^{h,k}(t,u)dt-\mathcal{V}^{(3)}(u(t))\eta^h(t,u)\eta^h(t,u)dt.
\]

The two following inequalities are used in the proof:
\begin{itemize}
\item the Gagliardo-Nirenberg inequality, for every $v\in H_0^1(0,1)$
\[
\|v\|_{L^\infty(0,1)}\le C\|v\|_{L^2(0,1)}^{\frac12}\|v\|_{H^1(0,1)}^{\frac12},
\]
combined with $\|v\|_{H^1(0,1)}\le C\|(-L)^{\frac12}v\|_{L^2(0,1)}=C\|v\|_{\frac12}$,
\item the Sobolev inequality $\|\cdot\|_{L^4(0,1)}\le C\|\cdot\|_{\frac18}$, which implies $\|e^{tL}\|_{\mathcal{L}(L^2(0,1),L^4(0,1))}\le Ct^{-\frac18}$.
\end{itemize}

To prove the first inequality in~\eqref{eq:proof_Poisson_SPDE_2}, an energy estimate and the use of the Gagliardo-Nirenberg and Young inequalities, give
\begin{align*}
\frac{1}{2}\frac{d\|\zeta^{h,k}(t,u)\|_{L^2(0,1)}^2}{dt}&+\|(-L)^{\frac12}\zeta^{h,k}(t,u)\|_{L^2(0,1)}^2\le \|\mathcal{V}''\|_\infty\|\zeta^{h,k}(t,u)\|_{L^2(0,1)}^2\\
&+\|\mathcal{V}^{(3)}\|_\infty \|\zeta^{h,k}(t,u)\|_{L^\infty(0,1)}\|\eta^h(t,u)\|_{L^2(0,1)}\|\eta^k(t,u)\|_{L^2(0,1)}\\
&\le (\lambda_1-\gamma+\epsilon)\|\zeta^{h,k}(t,u)\|_{L^2(0,1)}^2+\epsilon\|(-L)^{\frac12}\zeta^{h,k}(t,u)\|_{L^2(0,1)}^2\\
&+\frac{C}{\epsilon}\|\eta^h(t,u)\|_{L^2(0,1)}^2\|\eta^k(t,u)\|_{L^2(0,1)}^2,
\end{align*}
for $\epsilon>0$ sufficiently small.

Using the Poincar\'e inequality $\|(-A)^{\frac12}\cdot\|_{L^2(0,1)}^{2}\ge \lambda_1\|\cdot\|_{L^2(0,1)}$, then
\begin{align*}
\frac{1}{2}\frac{d\|\zeta^{h,k}(t,u)\|_{L^2(0,1)}^2}{dt}&\le -\bigl(\gamma-(1+\lambda_1)\epsilon\bigr)\|\zeta^{h,k}(t,u)\|_{L^2(0,1)}^2+C_\epsilon e^{-4\gamma t}\|h\|_{L^2(0,1)}^2\|k\|_{L^2(0,1)}^2.
\end{align*}
By Gronwall Lemma, since $\zeta^{h,k}(0,u)=0$,
\begin{align*}
\|\zeta^{h,k}(t,u)\|_{L^2(0,1)}^2&\le \int_{0}^{t}e^{-2\bigl(\gamma-(1+\lambda_1)\epsilon\bigr)(t-s)}e^{-4\gamma s}ds\|h\|_{L^2(0,1)}^2\|k\|_{L^2(0,1)}^2\\
&\le C_\epsilon e^{-2\bigl(\gamma-(1+\lambda_1)\epsilon\bigr)t}\|h\|_{L^2(0,1)}^2\|k\|_{L^2(0,1)}^2
\end{align*}

This concludes the proof of the first estimate in~\eqref{eq:proof_Poisson_SPDE_2}.

To obtain the second inequality in~\eqref{eq:proof_Poisson_SPDE_2}, it is sufficient to estimate (using the mild formulation)
\begin{align*}
\|\zeta^{h,k}(t,u)\|&\le \int_{0}^{t}\|e^{(t-s)L}\bigl(\mathcal{V}''(u(s))\zeta^{h,k}(s,u)\bigr)\|ds+\int_{0}^{t}\|e^{(t-s)L}\bigl(\mathcal{V}^{(3)}(u(s))\eta^h(s,u)\eta^k(s,u)\bigr)\|ds\\
&\le C\int_{0}^{t}\|\zeta^{h,k}(s,u)\|ds+\int_{0}^{t}\|e^{(t-s)L}\|_{\mathcal{L}(L^1,L^2)}\|\eta^h(s,u)\| \|\eta^k(s,u)\|ds\\
&\le C\int_{0}^{t}\|\zeta^{h,k}(s,u)\|ds+\int_{0}^{t}\frac{C_{\alpha,\epsilon}}{(t-s)^{\frac14+\epsilon}s^{2\alpha}}ds\|(-L)^{-\alpha}h\|\|(-L)^{-\alpha}k\|\\
&\le C\int_{0}^{t}\|\zeta^{h,k}(s,u)\|ds+\frac{C_\alpha}{t^{2\alpha}}\|h\|_{-\alpha}\|k\|_{-\alpha},
\end{align*}
and the conclusion follows from Gronwall Lemma.

\end{appendix}

\bibliographystyle{abbrv}


\end{document}